\documentclass[reqno, 12pt]{amsart}

\usepackage{amsmath}
\usepackage{amsthm}
\usepackage{amsfonts}
\usepackage{amssymb}
\usepackage{mathrsfs}
\usepackage{epsfig}
\usepackage{slashed}
\usepackage{mathtools}
\usepackage{enumerate}
\usepackage{fullpage}
\usepackage{xcolor}
\usepackage{comment}
\usepackage[hidelinks]{hyperref}

\usepackage{enumitem}

\newcommand{\mb}{\mathbf}
\newcommand{\mc}{\mathcal}

\renewcommand{\Re}{\mathrm{Re}\,}

\newcommand{\ran}{\mathrm{ran}\,}

\newcommand{\N}{\mathbb{N}}
\newcommand{\R}{\mathbb{R}}
\newcommand{\C}{\mathbb{C}}

\newcommand*\closure[1]{\overline{#1}}

\DeclarePairedDelimiter\abs{\lvert}{\rvert}

\DeclareMathOperator{\diag}{diag}

\newcommand\restr[2]{{  \left.\kern-\nulldelimiterspace #1 \vphantom{\big|}  \right|_{#2} }}

\newcommand\norm[1]{\left\Vert#1\right\Vert}

\newtheorem{lemma}{Lemma}[section]
\newtheorem{theorem}[lemma]{Theorem}
\newtheorem{corollary}[lemma]{Corollary}
\newtheorem{proposition}[lemma]{Proposition}
\newtheorem{remark}[lemma]{Remark}

\theoremstyle{definition}
\newtheorem{definition}[lemma]{Definition}
\newtheorem*{definition*}{Definition}

\newtheorem*{assumption*}{Assumption}

\numberwithin{equation}{section}

\title[]{Stable blowup for supercritical wave maps into perturbed spheres}

\author{Roland Donninger}
\address{Universität Wien, Fakultät für Mathematik,
  Oskar-Morgenstern-Platz 1, 1090 Vienna, Austria}
\email{roland.donninger@univie.ac.at}

\author{Birgit Sch\"orkhuber}
\address{Universit\"at Innsbruck, Institut f\"ur Mathematik, Technikerstra{\ss}e 13, 6020 Innsbruck,
Austria}
\email{Birgit.Schoerkhuber@uibk.ac.at}

\author{Alexander Wittenstein}
\address{Karlsruhe Institute of Technology, Institute for Analysis,  Englerstra{\ss}e 2, 76131 Karlsruhe, Germany}
\email{alexander.wittenstein@kit.edu}

\thanks{The third author was partially funded by the
Deutsche Forschungsgemeinschaft (DFG, German Research Foundation) - Project-ID 258734477 - SFB 1173.  
This research was funded in whole or in part by the Austrian Science
Fund (FWF) 10.55776/P34560. For open access purposes, the authors have
applied a CC BY public copyright license to any author-accepted
manuscript version arising from this submission. \\
The authors would
like to thank the Erwin Schrödinger International Institute for Mathematics and Physics (ESI) for hospitality. This
work was finalized during the thematic program “Nonlinear Waves and Relativity” at ESI in 2024.}

\begin{document}

\begin{abstract}
We consider wave maps from (1+$d$)-dimensional Minkowski space, $d\geq3$, into rotationally symmetric manifolds which arise from small perturbations of the sphere $\mathbb S^d$. We prove the existence of co-rotational self-similar finite time blowup solutions with smooth blowup profiles. Furthermore, we show the nonlinear asymptotic stability of these solutions under suitably small co-rotational perturbations on the full space. 
\end{abstract}

\maketitle

\section{Introduction}

We consider maps $U: \R^{1+d} \to N$, where $\R^{1+d}$ is the $(1+d)$-dimensional Minkowski space and $N$  a Riemannian manifold, specifically, a $d$-dimensional, rotationally symmetric warped product manifold, see \cite{Che11}, \cite{O'Neill83}, \cite{ShaTah94} for the general definition. Such a map $U$ is called \textit{wave map}, if it is a critical point of the following Lagrangian
\begin{align*}
\mc{L}[U] := \frac{1}{2} \int_{\R^{1+d}} \abs{\nabla_x U(t,x)}_h^2 -\abs{\partial_t U(t,x)}_h^2 \, dt \, dx, 
\end{align*}
where $h$ is a  metric on $N$. The Euler-Lagrange equations associated to this functional, which are in this case called the \textit{wave maps equation}, are given in local coordinates by \footnote{as usual Greek indices are running from 0 to $d$, Latin ones from 1 to $d$ and the Einstein summation convention is in force.}
\begin{align} \label{WM}
\Box \, U^{a} + \Gamma^{a}_{bc}(U)\, \partial^{\mu} U^{b}\, \partial_{\mu} U^{c} = 0, \quad a = 1,\ldots, d
\end{align}
and they constitute a system of semilinear wave equations. Here, $\Box =  -\partial_t^2 + \Delta_x$ is the linear wave operator on $\R^{1+d}$ and we raise indices with respect to the Minkowski metric $m^{\mu \nu} = m_{\mu \nu} = \diag(-1,1,\ldots,1)$. Furthermore, 
$\Gamma^{a}_{bc}$ are the Christoffel symbols associated to the metric $h$ on $N$. Eq.~\eqref{WM} is invariant under the scaling transformation
\begin{align*}
    U_{\lambda}(t,x) := U(t\slash \lambda, x\slash \lambda), \quad \lambda > 0,
\end{align*}
and by inspection of the scaling of the associated energy 
\begin{align*}
    E[U](t) := \frac{1}{2} \int_{\R^d} \abs{\partial_t U(t,x)}^2_h + \abs{\nabla_x U(t,x)}^2_h\,dx
\end{align*}
one finds \begin{align*}
    E[U_{\lambda}](t) = \lambda^{d-2} E[U](t \slash \lambda).
\end{align*}
This implies that the wave maps equation is energy  critical in dimension two and energy supercritical for $d \geq 3$. In the following, we restrict ourselves to the latter case.

It is well-known that in the energy supercritical case and for large classes of warped product target manifolds Eq.~\eqref{WM} admits finite-time blowup via self-similar solutions \cite{ShaTah94, CazShaTah98, DonGlo19}. However, besides their existence little is known about the role of these solutions in the generic time evolution. 
Most results are available for $N = \mathbb S^d$, where existence \cite{Sha88}, \cite{TurSpe90}, \cite{BizBie15} and stability of self-similar blowup has been established in the past years in various settings \cite{Don11, DonSchAic12,  CosDonXia16, ChaDonGlo17, CosDonGlo17, BieDonSch21, Glo23, DonWal23, DonWal24, DonOst24}, see also Section \ref{Sec:BlowupResults}. The main goal of this paper is to prove \textit{stable self-similar blowup} in all supercritical dimensions  for more general targets which can be viewed as small perturbations of the sphere.

\subsection{The main result}
To state our main results we introduce the following families of warped product manifolds.

\begin{definition}\label{Def:Perturbed_sphere}
   Let $\alpha \in C^{\infty}(\R)$ be a non-trivial real-valued $2\pi$-periodic function,  which is even,  nonnegative and  satisfies $\alpha(0) = \alpha(\pi) = 0$. Set $\epsilon_0 := (\max_{u \in [0,\pi]} \alpha(u))^{-1} > 0$. Then we define for $\epsilon \in \R$, $|\epsilon| < \epsilon_0$, the warped product manifold $(S^d_{\epsilon}, h)$ by
\begin{align}\label{perturbed sphere}
S^d_{\epsilon}  := (0,\pi) \times_{w_{\epsilon}} \mathbb{S}^{d-1}
\end{align}
equipped with the warping function $w_{\epsilon} : \R \to \R$,
\begin{align}\label{warping function}
    w_{\epsilon}(u) := \sin(u)(1 + \epsilon \, \alpha(u)).
\end{align}
In coordinates $(u,\Omega) \in (0,\pi) \times \mathbb{S}^{d-1}$ the metric on $S^d_{\epsilon}$ is given by 
\[ h = d u^2 +   w_{\epsilon}(u)^2 d \Omega^2,\]
with $d \Omega^2$ denoting the standard round metric on $\mathbb{S}^{d-1}\hookrightarrow \R^d$.
\end{definition}
   
The conditions in Definition \ref{Def:Perturbed_sphere} ensure that  $w_{\epsilon}$ is a smooth, odd and $2\pi$-periodic function which is  strictly positive in the open interval $(0,\pi)$. In addition, $w_{\epsilon}'(0) = 1$ as well as $w_{\epsilon}'(\pi) = -1$. This implies that  $S^d_{\epsilon}$  is a smooth, $d-$dimensional Riemannian manifold, see e.g. \cite{Bes78}.  For $\epsilon = 0$, we have $w_0(u) = \sin(u)$ and in this case $S^d_{\epsilon}$ can be identified with $\mathbb{S}^d \setminus 
\{\mathfrak N, \mathfrak S \}$, where $\mathfrak N$ and $\mathfrak S$ denote the north and south pole of the sphere. For small $|\epsilon| > 0$ the warping function $w_{\epsilon}$ can be viewed as a small perturbation of the warping function of the $d-$sphere $\mathbb{S}^d$ so that we will call $S^d_{\epsilon}$  a \textit{perturbed sphere} in that case. \medskip\\
It is convenient to consider on $S^d_{\epsilon}$ so-called normal coordinates $U = (U^1,\dots,U^d)$, where
\begin{equation}\label{Def:NormCoord}
	U^j:=u\, \Omega^j, \quad \text{for} \quad j=1,\dots,d,
\end{equation}
see, e.g. \cite{ShaTah94}. In this way, $S^d_{\epsilon}$ can be identified  (when including the north pole corresponding to the limiting value $u=0$) with the ball $B^d_{\pi}(0) \subset \R^d$.
Hence, we consider the initial value problem
\begin{align} \label{WM_IVP_Normal}
\begin{split}
\Box \, U^{a}  & + \Gamma^{a}_{bc}(U)\, \partial^{\mu} U^{b}\, \partial_{\mu} U^{c}   = 0, \quad a \in \{1,\ldots, d \} \\
U(0,\cdot) &  = U_0, \quad \partial_t U(0,\cdot) = U_1
\end{split}
\end{align}
on $I \times \R^d$, where $I \subset \R$ is an interval containing zero and data $U_0,U_1: \R^d \to \R^d$. We restrict our attention to co-rotational maps $U(t,\cdot): \R^d \to \R^d$ which are by definition of the form
\begin{align}\label{Def:CorAnsatz}
U(t,x) = u(t,|x|) \frac{x}{|x|},
\end{align}
for a \textit{radial profile} $u: I \times [0,\infty) \to \R$ which has to satisfy the equation
\begin{align}\label{equation for radial profile}
        \left(\partial_t^2 - \partial^2_r - \frac{d-1}{r} \partial_r \right)u(t,r) + \frac{d-1}{r^2}w_{\epsilon}(u(t,r))\,w_{\epsilon}'(u(t,r)) =0
\end{align} 
together with the boundary condition $u(t,0) = 0$ for all $t \in I$. 
As co-rotational symmetry is conserved by the flow it suffices to study the evolution of $u$ governed by Eq.~\eqref{equation for radial profile}.

It is well-known that in the case $\epsilon = 0$, Eq.~\eqref{equation for radial profile} admits self-similar solutions in all space dimensions $d \geq 3$. The so-called \textit{ground state self-similar solution} is known in closed form, see  \cite{TurSpe90}, \cite{BizBie15}, and given by 
\begin{align}\label{ground state}
    u_0^T(t,r) = f_0\left(\frac{r}{T-t}\right) \quad \text{for} \quad f_0(\rho) = 2 \arctan\left(\frac{\rho}{\sqrt{d-2}}\right), \quad T > 0.
\end{align}
 In the past years, the stability of $u_0^T$ has been studied intensively. In this paper, we address the problem of stable self-similar blowup for wave maps into $S^d_{\epsilon}$ for $\epsilon \neq 0$. Our first result proves the existence of globally (in space) smooth self-similar blowup profiles for Eq.~\eqref{equation for radial profile} in any space dimension $d \geq 3$, and for any $\epsilon \in \R$ sufficiently small. \medskip \\
In the following statements, we consider for a fixed function $\alpha$ as in Definition \ref{Def:Perturbed_sphere} the family of corresponding warping functions $w_{\epsilon}$, $|\epsilon| \leq \epsilon_0$.

\begin{theorem}\label{Theorem: Blowup Solution}
Let $d \geq 3$. There exists an $\epsilon^* \in \R$, $0 < \epsilon^* \leq \epsilon_0$ such that for every $\epsilon \in \R$ with $\abs{\epsilon} \leq \epsilon^*$, Eq.~\eqref{equation for radial profile} admits a self-similar solution
$u_{\epsilon}^T \in C^{\infty}([0,T) \times [0,\infty)) \cap L^{\infty}([0,T] \times [0,\infty))$ of the form
    \begin{align*}
        u_{\epsilon}^T(t,r) = f_{\epsilon}\left(\frac{r}{T-t}\right), \quad T > 0,
    \end{align*}
such that its gradient blows up in $r=0$, i.e.,
\[ \lim_{t \to T^{-}} |\partial_r u_{\epsilon}^T(t,0)| = + \infty .\]
The profile $f_{\epsilon}$ can be written as 
\[  f_{\epsilon} = f_0 + \phi_{\epsilon} \]
with $f_0$ defined in Eq.~\eqref{ground state} and a perturbation $\phi_{\epsilon} \in L^{\infty}([0,\infty)) \cap C^{\infty}([0,\infty))$. The function $\phi_{\epsilon}$ depends Lipschitz continuously on the parameter $\epsilon$ in the sense that 
\begin{align}\label{Th:Lipschitz_bound}
        \norm{ (\cdot)^{-1} ( \phi_{\epsilon}-\phi_{\kappa})}_{W^{2,\infty}([0,\infty))} \lesssim \abs{\epsilon-\kappa}
    \end{align}
for all $\abs{\epsilon},\abs{\kappa} \leq \epsilon^*$ and $\phi_0 = 0$.
Furthermore, $f_{\epsilon}$ satisfies $f_{\epsilon}(0) = 0$ and the limit  $\lim_{\rho \to \infty} f_{\epsilon} (\rho)$ exists. Finally, for every $k \in \N_0$ there are constants $ C_{1,k}, C_{2,k}  > 0$ depending on $\epsilon$  such that 
    \begin{align*}
        \abs{f_{\epsilon}^{(k)}(\rho)} \leq C_{1,k} \, \langle \rho \rangle ^{-k}  \quad \text{and} \quad \abs{(f_{\epsilon} + (\cdot) f_{\epsilon}')^{(k)}(\rho)} \leq  C_{2,k}\, \langle \rho  \rangle ^{-1-k}
\end{align*}
for all $\rho  \in [0,\infty)$.
\end{theorem}

Theorem \ref{Theorem: Blowup Solution} is obtained  by functional analytic methods using a perturbative construction which exploits the invertibility of the linearization around $f_0$ in suitable function spaces. This approach allows us to prove rigorous results on the stability of these solutions under small co-rotational perturbations. We  phrase the result in normal coordinates and consider without loss of generality small perturbations around the blowup solution with blowup time $T=1$.

 \begin{theorem}\label{Th:Stability_NormalCoord}
Let $d \geq 3$ and choose $\epsilon^* > 0$ as in Theorem \ref{Theorem: Blowup Solution}. For $\epsilon \in \R$, $|\epsilon| \leq \epsilon^*$ define
\[ U^T_{\epsilon}(t,x) :=  f_{\epsilon}\left (\frac{|x|}{T-t} \right )  \frac{x}{|x|}.\]
Consider co-rotational initial data of the form
\begin{align*}
U_0 = U^1_{\epsilon}(0,\cdot) + \nu_0, \quad U_1 = \partial_t U^1_{\epsilon}(0,\cdot) + \nu_1,
\end{align*}
where for $i \in \{0,1\}$, $\nu_i: \R^d \to \R^d, \, \nu_i(x) = x v_i(|x|)$ and $v_i: [0,\infty) \to \R$ are such that $v_i(|\cdot|)  \in \mc S(\R^d)$. Let $(s,k) \in \R \times \N$  satisfy 
    \begin{align}\label{condition}
        \frac{d}{2} <s \leq \frac{d}{2}+\frac{1}{2d+2}, \quad k > d+2.
    \end{align} 
Then there exists  an $\overline{\epsilon} \in \R$, $0 < \overline{\epsilon} \leq \epsilon^*$ such that for every $\epsilon \in \R$ with $\abs{\epsilon} \leq \overline{\epsilon}$ the following holds: There are $\delta, M_0 > 0$ such that for $(\nu_0,\nu_1)$ as above satisfying
     \begin{align}\label{Eq:Data_Smallness}
        \norm{(\nu_0,\nu_1)}_{\dot{H}^s\cap\dot{H}^k(\R^d,\R^d)\times\dot{H}^{s-1}\cap\dot{H}^{k-1}(\R^d,\R^d)}< \frac{\delta}{M_0},
    \end{align}
there exists a $T = T_{\epsilon} \in [1-\delta , 1+\delta]$ and a unique co-rotational map $U \in C^{\infty}([0,T)\times\R^d,\R^d)$ that satisfies Eq.~\eqref{WM_IVP_Normal} for all $(t,x) \in [0,T) \times \R^d$. The gradient of $U$ blows up at the origin as $t \to T^{-}$ and we have the decomposition 
    \begin{align*}
        U(t,x) = U^{T}_{\epsilon}\left (\frac{x}{T-t} \right ) + \nu \left (t, \frac{x}{T-t} \right ),
       \end{align*}
for a function $\nu :[0,T) \times \R^d \to \R^d$ which satisfies
\begin{align}\label{Decay_Nu}
 \| \nu(t, \cdot) \|_{\dot H^r(\R^d,\R^d)} +  \|((T-t) \partial_t + \Lambda) \nu(t, \cdot) \|_{\dot H^{r-1}(\R^d,\R^d)} \to 0 
 \end{align}
as $t \to T^{-}$ for all $r \in [s,k]$ where $\Lambda$ is defined as the operator $\Lambda = x \cdot \nabla$. In addition,
\[ U(t,(T-t)\cdot) \to U_{\varepsilon} \]
uniformly on compact subsets of $\R^d$ as $t \to T^{-}$. 
\end{theorem}

Theorem \ref{Theorem: Blowup Solution} and Theorem \ref{Th:Stability_NormalCoord} guarantee for the considered class of target manifolds the existence of \textit{stable} blowup via co-rotational self-similar solutions in all space dimensions $d \geq 3$. In low space dimensions, this is the first result of this kind for targets other than the sphere, see Section \ref{Sec:BlowupResults}. The  assumptions on the initial data and the underlying topology arise naturally by reformulating the problem as a classical stability problem using similarity coordinates, see Section \ref{Sec:Reformulation_Outline}. Some further remarks are in order.

\begin{remark}{(Spectral problem)}
In the proof of Theorem \ref{Th:Stability_NormalCoord} spectral analysis presents the key difficulty. The ODE problem that arises in the investigation of the stability of self-similar blowup in nonlinear wave equations is notoriously hard and current techniques require explicit knowledge of the profile, see e.g.~\cite{Don24} for a recent exposition. In this paper, we exploit the perturbative nature of our profile construction which enables us to extract the necessary information on the spectral structure of the linearization around $f_{\epsilon}$ even though the profile is not known in closed form. 
\end{remark}

\begin{remark}{(Function spaces)}
We have chosen to investigate the stability problem in a strip $[0,T) \times \R^d$ using the framework of homogenous intersection Sobolev spaces based on recent work by Glogi\'c \cite{Glo23}.  However, an analogous stability result could equally be established in backward light cones (at least at a regularity level strictly above scaling) using known results for $\epsilon =0$, see below. 
\end{remark}

We note that the approach developed in this paper is not limited to the wave maps equation. In fact, our results suggest that not only the existence of self-similar blowup solutions in nonlinear wave equations but also their dynamical stability properties are stable under suitable small (scale invariant) perturbations of the nonlinearity. 

\subsection{Blowup in energy supercritical wave maps - Known results} \label{Sec:BlowupResults}

As the literature on wave maps is vast, we restrict the discussion to works which are relevant in the context of our main results. In particular, we only discuss the energy supercritical case $d \geq 3$. \\
Concerning the existence of self-similar blowup for wave maps into the sphere, the first result was obtained by Shatah \cite{Sha88} for maps from $\R^{3+1}$ into $\mathbb{S}^3$. Later, Turok and Spergel \cite{TurSpe90} found in the same setting an explicit example of a self-similar solution (which is expected to be the one from Shatah). Bizo\'{n} and Biernat \cite{BizBie15} provided the analogous closed form expression in all space dimensions $d\geq4$, which is given in Eq.~\eqref{ground state}. For $3 \leq d \leq 6$, $f_0$ can indeed be considered as the ``ground state" of a family of self-similar blowup profiles whose existence has been proven by Bizo\'{n} \cite{Biz00} and Biernat, Bizo\'{n} and Maliborski \cite{BieBizMal17}.

Numerical investigations by  Bizo\'{n}, Chmaj and Tabor \cite{BizChmTab00}, see also \cite{BizBie15}, indicate that blowup via $f_0$ is generic (at least within the class of co-rotational initial data). The first rigorous proof of the  stability of $f_0$  under small co-rotational perturbations (in  $d=3$ and restricted to a backward light cone) has been established in the series of works by the first author \cite{Don11}, jointly with Aichelburg and the second author \cite{DonSchAic12} and with Costin and Xia  \cite{CosDonXia16}, respectively.   Generalizations to higher dimensions were  obtained in \cite{ChaDonGlo17},  \cite{CosDonGlo17}, where the latter work by Costin, the first author and Glogi\'c provides an important simplification of the ODE analysis underlying the spectral problem. Recent works by the first author and Wallauch \cite{DonWal23, DonWal24} establish stability in backward light cones at the optimal regularity level in $d=3,4$. 

In another line of research the stability is studied in so-called hyperboloidal similarity coordinates, which allow to follow the evolution past the blowup time, at least outside the singularity. This has been initiated by the first two authors together with Biernat in \cite{BieDonSch21} and  generalized by the first author and Ostermann in \cite{DonOst24}. As already mentioned above, the stability analysis on a strip $[0,T) \times \R^d$ has been implemented recently by Glogi\'c \cite{Glo23} in all space dimensions $d \geq 3$.

Finally,  non-self-similar blowup occurs at least in dimensions $d \geq 7$, as has been demonstrated by Ghoul, Ibrahim and Nguyen \cite{GhuIbrNGu2018}.\medskip\\
The existence of finite-time blowup for wave maps from $\R^{d+1}$ into general warped product manifolds $N^m = (0,a) \times_{w} \mathbb S^{m-1}$, $m \in \N$, $m \geq 2$, $a > 0$, has been addressed by Shatah and Tahvildar-Zadeh \cite{ShaTah94} and  Cazenave, Shatah and Tahvildar-Zadeh \cite{CazShaTah98}: 
In odd space dimensions $d \geq 3$, by using variational methods and ODE tools, the authors construct $k-$equivariant self-similar profiles ($k=1$ corresponds to the co-rotational case) locally in a backward light cone which lead to finite time blowup via finite speed of propagation. The admissible values of $k$ and $m$ depend on the properties of the corresponding warping function $w$. By lifting these solutions to one space dimension higher, finite time blowup is established in even space dimensions $d \geq 4$. The stability properties of the solutions constructed in \cite{ShaTah94} and \cite{CazShaTah98} are unknown, except for $N = \mathbb S^d$. Finally, for $d \geq 8$, the first author and Glogi\'c \cite{DonGlo19} provided examples of negatively curved targets which admit an explicit self-similar blowup solution and proved its stability in backward light cones in the case $d=9$.

\subsection{Reformulation of the problem and outline of the proof}\label{Sec:Reformulation_Outline}

It is common to write co-rotational maps as $U(t,x) = x v(t,|x|)$ since $v$ then satisfies a radial nonlinear wave equation with a smooth nonlinearity, see e.g. \cite{ShaTah94}.

To realize this, we change variables and define $\widetilde{v}(t,r) := r^{-1} u(t,r)$  which transforms Eq.~\eqref{equation for radial profile} into a $(d+2)$-dimensional radial semilinear wave equation with a now smooth nonlinearity

\begin{align}\label{Cauchy problem 2}
            \left(\partial_t^2 - \partial^2_r - \frac{d+1}{r} \partial_r \right)\widetilde v(t,r) - \frac{d-1}{r^3}\left(r\,\widetilde{v}(t,r)-w_{\epsilon}(r\,\widetilde{v}(t,r))\,w_{\epsilon}'(r\,\widetilde{v}(t,r))\right) =0.
\end{align}

In the investigation of Eq.~\eqref{Cauchy problem 2} it is convenient to define $n := d+2$ and $v(t,x) := \widetilde{v}(t,\abs{x})$ for $x \in \R^n$ and then formulate the equation as a wave equation on $\R^n$,
\begin{align}\label{NLWhigherdim}
 	(\partial^2_t  - \Delta_x) v(t,x) = \frac{n-3}{\abs{x}^3}\left(\abs{x}v(t,x) - w_{\epsilon}(\abs{x}v(t,x))w_{\epsilon}'(\abs{x}v(t,x))\right), \quad x \in \R^n.
\end{align}

We then introduce similarity coordinates 
 \begin{align*}
 	\tau = \log\left(\frac{T}{T-t}\right), \quad \xi = \frac{x}{T-t}
 \end{align*}
for $T > 0$ and $(t,x) \in [0,T) \times \R^n$ and rewrite Eq.~\eqref{NLWhigherdim} as a first order system of the form
\begin{align}\label{Outl:EvolSim}
    \partial_{\tau}\Psi(\tau)=\mb{L}\Psi(\tau)+ \mb{N_{\epsilon}}(\Psi(\tau)), 
\end{align}
where $\mb{L}$ generates the free wave evolution and $\mb{N_{\epsilon}}$ denotes the nonlinearity which depends on $\epsilon$ via the warping function $w_{\epsilon}$, see Section \ref{Section: Similarity Variables} for the details. We study the problem in intersection Sobolev spaces of radial functions 
\[ \mc{H}_r^{s,k}= (\dot H_r^{s}(\R^n) \cap \dot H_r^{k}(\R^n)) \times (\dot H_r^{s-1}(\R^n) \cap \dot H_r^{k-1}(\R^n)) \]
for suitable exponents $\frac{n}{2} - 1 < s < \frac{n}{2} < k$, $k \in \N$.

\subsubsection{Existence of self-similar blowup profiles}
For $\epsilon=0$, the blowup solution \eqref{ground state} transforms into a static solution $\mb \Psi_0$ of \eqref{Outl:EvolSim}, i.e.,
\begin{align}\label{Outl:StatSim0}
    \mb{L}\mb  \Psi_0+ \mb{N_{0}}(\mb  \Psi_0) = 0. 
\end{align}
In order to find for $\epsilon \neq 0$ a solution to
\begin{align}\label{Outl:StatSim}
    \mb{L} \mb  \Psi_{\epsilon} + \mb{N_{\epsilon}} (\mb  \Psi_{\epsilon}) = 0
\end{align} 
we insert the ansatz  $\mb  \Psi_{\epsilon} = \mb  \Psi_0 +\mb \Phi_{\epsilon}$ into \eqref{Outl:StatSim} and write the resulting equation for the perturbation $\mb \Phi_{\epsilon}$ as 
\begin{align}\label{Outl:Pert}
    - \mb{L}_0 \mb \Phi_{\epsilon} = \mb{V}_{\epsilon}(\mb \Psi_0)\mb \Phi_{\epsilon}+ \mb{\widetilde{N}}_{\epsilon}(\mb  \Phi_{\epsilon}) +\mb{\mathcal{R}}_{\epsilon}(\mb \Psi_0)
\end{align}
where $\mb{L}_0 := \mb{L} +\mb{V}_0(\mb \Psi_0)$ denotes the linearization around $\mb \Psi_0$ for $\epsilon = 0$,  $\mb{V}_{\epsilon}(\mb \Psi_0)$ is a potential term depending on $\epsilon$ via the warping function, $\mc{R}_{\epsilon}(\mb \Psi_0)$ denotes the remainder and $\mb{\widetilde{N}}_{\epsilon}$ the nonlinearity, which is quadratically small in its argument. The spectral analysis of \cite{Glo23} implies that  $\mb{L}_0$, defined as an unbounded operator on a suitable domain $\mc D(\mb{L}_0) \subset \mc{H}_r^{s,k} \to\mc{H}_r^{s,k} $, is invertible. 

In order to control the right-hand side of \eqref{Outl:Pert} we prove a parameter-dependent version of Schauder type estimates originally stated in \cite{Glo23}. For this we assume 
\begin{align}\label{Out:sk}
\frac{n}{2}-1 < s \leq \frac{n}{2}-1 + \frac{1}{2n-2}, \quad k>n,
\end{align}
see Appendix \ref{AppendixA}. Using this, we show that the first two terms on the right hand side of Eq.~\eqref{Outl:Pert} define maps $\mc{H}_r^{s,k} \to \mc{H}_r^{s,k+1}$ which are Lipschitz-continuous with respect to the parameter $\epsilon$ and with respect to the argument. Moreover, the remainder is in $\mc{H}_r^{s,k+1}$ and depends Lipschitz-continuously on $\epsilon$,  see Lemma \ref{Local Lipschitz estimate for nonlinearity} - \ref{Lemma: Remainder Term}. 

 With these prerequisites, we solve Eq.~\eqref{Outl:Pert} for suitably small $\epsilon > 0$ by applying Banach's fixed point theorem. The smoothness of the constructed $\mb \Phi_{\epsilon}$ can be shown inductively using the mapping properties of the  involved operators together with Sobolev embedding. The claimed decay properties of the resulting smooth solution $\mb \Psi_{\epsilon} = (\psi_{\epsilon,1}, \psi_{\epsilon,2})$ to Eq.~\eqref{Outl:StatSim} do not follow from this perturbative approach (although some decay can be obtained from Strauss type inequalities). Instead, we use ODE analysis similar to the classcial work of Kavian and Weissler \cite{KavWei90} to show that $\psi_{\epsilon,j} = \mc O(|\xi|^{-j})$ at infinity. This property will be crucial in the analysis. 
Finally, translating the result back to the original coordinates and variables yields Theorem \ref{Theorem: Blowup Solution}. 

\subsubsection{Stability of self-similar blowup}

We now consider the time-dependent problem \eqref{Outl:EvolSim} and linearize around $\mb \Psi_{\epsilon}$ by making the ansatz $\Psi(\tau) = \mb \Psi_{\epsilon} + \Phi_{\epsilon}(\tau)$
to obtain 
\begin{align}\label{Out:Cauchy}
 	\begin{cases}
 		\partial_\tau \Phi_{\epsilon}(\tau) =\mb{L}\Phi_{\epsilon}(\tau) + \mb{L}'_{\epsilon}\Phi_{\epsilon}(\tau) + \widehat{\mb{N}}_{\epsilon}(\Phi_{\epsilon}(\tau)),\\	
 		\Phi_{\epsilon}(0)=\mb{U}_{\epsilon,T},
 	\end{cases}
\end{align}
where $\mb{L}$ is again the free wave operator, $\mb{L}_{\epsilon}'\mb u = (0, V_{\epsilon}(\psi_{\epsilon,1}) u_1)$
for $\mb u = (u_1,u_2) \in  \mc{H}_r^{s,k}$ includes the potential term and $\widehat{\mb{N}}_{\epsilon}$ denotes the nonlinear remainder. It is known that under our assumptions on the exponents $(s,k)$, $\mb{L}$ generates a strongly continuous one-parameter semigroup  $(\mb{S}(\tau))_{\tau\geq 0}$ on $\mc{H}_r^{s,k}$, which decays exponentially to zero as $\tau \to \infty$, see  \cite{Glo23}, \cite{Don24}. The decay properties of $\psi_{\epsilon,1}$ at infinity allow us to show compactness of the perturbation $\mb{L}_{\epsilon}'$. This implies in particular that the full linear operator $\mb{L}_{\epsilon} := \mb{L} + \mb{L}_{\epsilon}'$ generates a semigroup $(\mb{S}_{\epsilon}(\tau))_{\tau\geq 0}$. In order to obtain suitable growth bounds for the linearized evolution, we investigate the spectrum of the generator. From time-translation invariance we expect $\lambda = 1$ to be an eigenvalue. We make this rigorous in Lemma \ref{Le:g} and show that 
  \begin{align*}
        \mb{g}_{\epsilon} := \begin{pmatrix}
            \psi_{\epsilon,1} + \Lambda \psi_{\epsilon,1}\\
        2 \psi_{\epsilon,1} + 3 \Lambda \psi_{\epsilon,1} + \Lambda^2 \psi_{\epsilon,1}
        \end{pmatrix} 
    \end{align*}
is the unique eigenfunction corresponding to $\lambda = 1$. The decay of $\psi_{\varepsilon,1}$ is crucial in order to show that $\mb{g}_{\epsilon} \in \mc D(\mb{L}_{\epsilon})$. For $\epsilon = 0$, the spectral properties of $\mb{L}_{0}$ are well-known. More precisely, we know from \cite{Glo23} that there exists an $\tilde \omega > 0$ such that
 \begin{align*}
        \sigma(\mb{L}_{0}) \subset \{ \lambda \in \C : \Re \lambda <  -\tilde \omega \} \cup \{1\}.
 \end{align*}
By a perturbation argument, using the fact that $1 \in  \sigma(\mb{L}_{\epsilon})$, we infer that 
  \begin{align*}
        \sigma(\mb{L}_{\epsilon}) \subset \{ \lambda \in \C : \Re \lambda <  -\omega_0 \} \cup \{1\}
 \end{align*}
for some $0 < \omega_0 < \tilde \omega$ and sufficiently small $\epsilon \in \R$. The spectral structure in combination with resolvent bounds proven uniformly in the parameter imply  
\[  \norm{\mb{S}_{\epsilon}(\tau) (\mb{I}-\mb{P}_{\epsilon}) \mb{u}}_{s,k} \lesssim e^{-\omega_0 \, \tau} \norm{(\mb{I}-\mb{P}_{\epsilon})\mb{u}}_{s,k},   \]
and $    \mb{P}_{\epsilon}\, \mb{S}_{\epsilon}(\tau)  = e^{\tau}\,\mb{P}_{\epsilon}$ for $\mb{P}_{\epsilon}$ denoting the spectral projection onto the eigenspace spanned by $ \mb{g}_{\epsilon}$, see \cite{Ost23}, Theorem A.1. For the nonlinearity $\widehat{\mb{N}}_{\epsilon}$, we establish Lipschitz bounds by imposing again the above assumptions \eqref{Out:sk} on the Sobolev exponents and apply the estimates of Appendix \ref{AppendixA}. With these results at hand we construct in Section \ref{Section: Nonlinear Cauchy Problem} global, exponentially decaying strong $\mc{H}_r^{s,k}$-solutions of Eq. \eqref{Out:Cauchy}.
For this, we use the standard approach and first suppress the exponential growth induced by the symmetry eigenvalue $\lambda = 1$ by a correction with values in $\mathrm{ran} \mb{P}_{\epsilon}$. In a second step we account for this using the $T-$dependence of the initial condition to determine the suitable blowup time $T_{\epsilon}$. In Proposition \ref{Upgrade to classical solution} we upgrade the constructed strong solutions to classical ones. By defining

\begin{align}\label{Def:SelfSem_v}
v_{\epsilon}^T(t,x) := \frac{1}{T-t} \psi_{\epsilon}\left (\frac{x}{T-t} \right ), \quad   \psi_{\epsilon}(\xi) := |\xi|^{-1} f_{\epsilon}(|\xi|)
\end{align}
for $x \in \R^n$ and $t \in [0,T)$ we get the following result. 

\begin{theorem}\label{Theorem: stability of blowup solution}
Let $n \geq 5$ and choose $\epsilon^* > 0$ as in Theorem \ref{Theorem: Blowup Solution}. For $\epsilon \in \R$, $|\epsilon| \leq \epsilon^*$, let $v^T_{\epsilon}$ be defined as in Eq.~\eqref{Def:SelfSem_v} and let $(s,k) \in \R \times \N$  satisfy
    \begin{align}\label{condition on exponents in n-dimensions}
        \frac{n}{2}-1<s \leq \frac{n}{2}-1+\frac{1}{2n-2}, \quad k > n.
    \end{align}
    Then there exist $\omega>0$ and $0 < \overline{\epsilon} \leq \epsilon^*$ such that for every $\epsilon \in \R$, $\abs{\epsilon} \leq \overline{\epsilon}$ there are $\delta >0$ and  $M > 1$ such that the following holds: For any pair of radial, real-valued functions $\varphi_0, \varphi_1 \in \mc S(\R^n)$ satisfying
     \begin{align}
        \norm{(\varphi_0,\varphi_1)}_{\dot{H}^s\cap\dot{H}^k(\R^n)\times\dot{H}^{s-1}\cap\dot{H}^{k-1}(\R^n)}< \frac{\delta}{M},
    \end{align}
   there exists $T = T_{\epsilon} \in [1-\delta , 1+\delta]$ and a unique radial solution $v \in C^{\infty}([0,T)\times\R^n)$ to Eq.~\eqref{NLWhigherdim} with	
    \begin{align*}
        v(0,\cdot) =v_{\epsilon}^1(0,\cdot) + \varphi_0, \quad \partial_t v(0,\cdot) = \partial_t v_{\epsilon}^1(0,\cdot) + \varphi_1.
       \end{align*}
Moreover, $v$ blows up at $(T,0)$ and can be decomposed as 
 \[ v(t,x) = v_{\epsilon}^{T}(t,x) + \frac{1}{T-t} \varphi \left (\log\left(\frac{T}{T-t}\right), \frac{x}{T-t} \right)  \]
for all $(t,x) \in [0,T) \times \R^n$, where $\varphi \in C^{\infty}([0,T)\times\R^n)$ is radially symmetric and satisfies
\begin{align}
\|\varphi (-\log(T-t) + \log T,\cdot) \|_{\dot H^{r}(\R^n)} \lesssim \delta (T-t)^{\omega}
\end{align}
for all $r \in [s,k]$. Furthermore,
\begin{align}
\|(\partial_0   + \Lambda + 1) \varphi(-\log(T-t) + \log T,\cdot) \|_{\dot H^{r-1}(\R^n)} \lesssim \delta (T-t)^{\omega}.
\end{align}
\end{theorem}

Finally,  we rephrase the results of Theorem \ref{Theorem: stability of blowup solution} in terms of normal coordinates using the equivalence of norms of corotational maps and their radial profiles, see \cite{Glo22}, to obtain Theorem \ref{Th:Stability_NormalCoord}.

\subsection{Notation} For two real numbers $A,B \in \R$ we write $ A \lesssim B$ if there exists an absolute constant $C >0$ such that $A \leq C B$ holds.  For a constant $\omega \in \R$ we denote by $\mathbb{H}_{\omega} := \{ z \in \C : \Re(z) > \omega\}$ the open right half-plane of $\omega$.\\
Furthermore, if $\mathcal{H}$ is a Hilbert space and $(L, \mathcal{D}(L))$ a closed linear operator on $\mathcal{H}$ we denote by $R_L(\lambda) := \left(\lambda - L\right)^{-1}$ the resolvent operator for $\lambda$ lying in the resolvent set $\rho(L)$.

For $n \in \N$, we denote by 
\[ C^\infty_r(\R^n):=\{f\in C^\infty(\R^n): f\text{ is radial}\} \]
the set of smooth radially symmetric functions and by $C^\infty_{c,r}(\R^n)$ the subset of those, which have compact support. 
By
\[ C_e^\infty[0,\infty):=\{f\in C^\infty([0,\infty)):f^{(2j+1)}(0)=0\text{ for }j\in\mathbb N_0\}, \]
we define set of smooth ``even'' functions and note that there is a one-to-one correspondence between $C^\infty_r(\R^n)$ and $C_e^\infty[0,\infty)$.
As usual, $\mathcal{S}(\R^n)$ and  $\mathcal{S}_r(\R^n)$  denote the set of Schwartz functions and radial Schwartz functions, respectively. 

We define the Fourier transform  $\mathcal{F}u$ of $u\in C_c^{\infty}(\mathbb{R}^n)$ by
\begin{align*}
    \mathcal{F}u(\xi)=\frac{1}{(2\pi)^{\frac{n}{2}}}\int_{\mathbb{R}^n}u(x)e^{-i\xi\cdot x}\,dx.
\end{align*}

For  $u,v\in C_c^{\infty}(\mathbb{R}^n)$ and $s\geq 0$ we define as usual the inner product
\begin{align*}
    \langle u,v \rangle_{\dot H^s(\mathbb{R}^n)}:= \langle |\cdot|^s\mathcal{F}u, |\cdot|^s\mathcal{F}v \rangle_{L^2(\mathbb{R}^n)}, 
\end{align*}
and the corresponding norm $\norm{u}^2_{\dot H^s(\mathbb{R}^n)} := \langle u,u \rangle_{\dot H^s(\mathbb{R}^n)}$. If $k \in \mathbb{N}_0$ is a nonnegative integer
\begin{equation*}
    \norm{u}_{\dot{H}^k(\mathbb{R}^n)}\simeq \sum_{\abs{\beta}=k}\norm{\partial^{\beta}u}_{L^2(\mathbb{R}^n)}
\end{equation*}
for all $u\in C_c^{\infty}(\mathbb{R}^n)$. Finally, for $s \geq 0$ the homogeneous radial Sobolev space $\dot H^{s}_{r}(\R^n)$ denotes the space which is obtained by completion of radial test functions $C_{c,\,r}^{\infty}(\mathbb{R}^n)$ with respect to the above defined norm.

\section{Similarity variables and functional analytic setup}\label{Section: Similarity Variables}

For $d \geq 3$ and a fixed function $\alpha$, see Definition \ref{Def:Perturbed_sphere},  let  $S^{d}_{\epsilon}$, for $\epsilon \in \R$ with $|\epsilon| \leq \epsilon_0$, denote the corresponding family of perturbed spheres. For the warping function $w_{\epsilon}$ as defined in Eq.~\eqref{warping function} we consider for $x \in \R^n, n := d+2$ the equation
\begin{align}\label{n-dimensional semilinear wave eq}
 	\partial^2_t v(t,x)
 	 - \Delta_x v(t,x) = \frac{n-3}{\abs{x}^3}\left(\abs{x}v(t,x) - w_{\epsilon}(\abs{x}v(t,x))w_{\epsilon}'(\abs{x}v(t,x))\right).
\end{align}

\subsection{Reformulation of the problem in similarity coordinates}
Restricting to $t \in [0,T)$ for $T >0$ and $x \in \R^n$ we rewrite Eq.~\eqref{n-dimensional semilinear wave eq} in similarity coordinates 
\begin{align*}
    \tau:=\log{\left (\frac{T}{T-t}\right)} \quad  \quad \xi:=\frac{x}{T-t}.
\end{align*}
Note that this transformation maps the strip $S_T:=[0,T)\times\mathbb{R}^n$ into the upper half-space $H_+:=[0,\infty)\times\mathbb{R}^n$. Additionally, we define a new rescaled dependent variable via
\begin{align*}
    \psi(\tau,\xi):= Te^{-\tau}\,v(T-Te^{-\tau},Te^{-\tau}\xi).
\end{align*}
Consequently, the evolution of $v$ inside $S_T$ corresponds to the evolution of $\psi$ inside $H_+$. Furthermore, the differential operators with respect to $t$ and $x$ are given in the new variables by 
\begin{align*}
    \partial_t=\frac{e^{\tau}}{T}(\partial_{\tau}+\Lambda),\quad \text{and} \quad \partial_{x_i}= \frac{e^{\tau}}{T} \partial_{\xi_i},
\end{align*}
for $i \in \{1,\dots,n \}$, where the operator $\Lambda$ is again defined via $\Lambda f(\xi):=\xi \cdot \nabla f(\xi)$. Eq.~\eqref{n-dimensional semilinear wave eq} then transforms into
 \begin{align}\label{equation in similarity coordinates}
 \begin{split}
     (\partial_{\tau}^2+3\partial_{\tau}+2\partial_{\tau}\Lambda-\Delta_{\xi}+\Lambda^2 & +3\Lambda+2)\psi(\tau,\xi) \\
     & = \frac{n-3}{\abs{\xi}^3}\left(\abs{\xi}\psi(\tau,\xi)-w_{\epsilon}(\abs{\xi}\psi(\tau,\xi))\,w_{\epsilon}'(\abs{\xi}\psi(\tau,\xi))\right).
 \end{split}
 \end{align}
By defining
\begin{align*}
    \psi_1(\tau,\xi):=\psi(\tau,\xi),\quad \psi_2(\tau,\xi):=(\partial_{\tau}+\Lambda+1)\psi(\tau,\xi)
\end{align*} and $\Psi(\tau):=(\psi_1(\tau,\cdot),\psi_2(\tau,\cdot))$ we get an evolution equation for $\Psi$ of the form
\begin{align}\label{evolution equation}
    \partial_{\tau}\Psi(\tau)=\mb{\widetilde{L}}\Psi(\tau)+ \mb{N_{\epsilon}}(\Psi(\tau)), 
\end{align}
where
\begin{align}\label{wave operator in similarity coordinates}
    \mb{\widetilde{L}}:= \begin{pmatrix}
-\Lambda-1 & 1\\
\Delta & -\Lambda-2
\end{pmatrix}
\end{align}
is the wave operator in similarity coordinates and the nonlinearity is given by
\begin{align}\label{Def: Nonlinearity}
    \mb{N_{\epsilon}}(\mb{u})=\begin{pmatrix}
        0\\
        N_{\epsilon}(u_1)
    \end{pmatrix} \quad N_{\epsilon}(u_1)(\xi):= \frac{n-3}{\abs{\xi}^3}\eta_{\epsilon}(\abs{\xi}u_1(\xi)),
\end{align}
for $\mb u = (u_1,u_2)$ where 
\begin{align}\label{eta}
    \eta_{\epsilon}(y):=y-w_{\epsilon}(y)w_{\epsilon}'(y).
\end{align}

\subsection{Time-independent solutions}
For $\epsilon = 0$ a static, smooth solution to Eq.~\eqref{evolution equation} is given by 
\begin{align}\label{Eq:Def_Psi0}
    \mb \Psi_0:=\begin{pmatrix}
        \psi_{0,1}\\
        \psi_{0,2}
    \end{pmatrix}, \text{ with } \psi_{0,1}(\xi)= \abs{\xi}^{-1}f_0(\abs{\xi}) \text{ and } \psi_{0,2} = \psi_{0,1} + \Lambda \psi_{0,1},
\end{align}
where the profile $f_0$ is defined in Eq.~\eqref{ground state}. More precisely, 
\begin{align}
    \mb{0} = \widetilde{\mb{L}}\mb{\Psi}_0 + \mb{N}_0(\mb{\Psi}_0).
\end{align}

In order to prove the existence of a static solution for  small $\epsilon \neq 0$, i.e., a smooth function $\mb \Psi_{\epsilon}$ which satisfies
\begin{align}\label{static equation}
    \mb{0} = \widetilde{\mb{L}}\mb \Psi_{\epsilon} + \mb{N}_{\epsilon}(\mb \Psi_{\epsilon}),
\end{align}
we make the following perturbative ansatz
\begin{align*}
    \mb \Psi_{\epsilon} = \mb  \Psi_0 + \mb  \Phi_{\epsilon}.
\end{align*}
By plugging this into Eq.~\eqref{static equation} we obtain an equation for the perturbation $\mb  \Phi_{\epsilon}$ which we write in the following form
\begin{align}\label{perturbation equation}
    \mb{0} &= \mb{\widetilde{L}}\mb \Phi_{\epsilon} + \mb{N}_{\epsilon}(\mb \Psi_0 + \mb \Phi_{\epsilon}) - \mb{N}_0(\mb \Psi_0)\nonumber\\ &= \mb{\widetilde{L}}_0 \mb \Phi_{\epsilon}+\mb{V}_{\epsilon}(\mb \Psi_0)\mb \Phi_{\epsilon}+\mb{\mathcal{R}}_{\epsilon}(\mb \Psi_0)+\mb{\widetilde{N}}_{\epsilon}(\mb  \Phi_{\epsilon}),
\end{align}
where $\mb{\widetilde{L}}_0 := \mb{\widetilde{L}} +\mb{V}_0(\mb \Psi_0)$, 
\begin{align*}
    \mb{V}_0(\mb  \Psi_0)\mb{u}&:= \begin{pmatrix}
        0\\
        V_0(\psi_{0,1})\,u_1
    \end{pmatrix} \quad \text{ for } \quad V_0(\psi_{0,1})(\xi) := \frac{n-3}{\abs{\xi}^2}\eta_0'(\abs{\xi}\psi_{0,1}(\xi)),\\
    \mb{V}_{\epsilon}(\mb  \Psi_0)\mb{u}&:= \begin{pmatrix}
        0\\
        V_{\epsilon}(\psi_{0,1})\,u_1
    \end{pmatrix} \quad \text{ for } \quad V_{\epsilon}(\psi_{0,1})(\xi) := \frac{n-3}{\abs{\xi}^2}\left(\eta_{\epsilon}'(\abs{\xi}\psi_{0,1}(\xi))-\eta_0'(\abs{\xi}\psi_{0,1}(\xi))\right),\\
    \mb{\mathcal{R}}_{\epsilon}(\mb  \Psi_0)(\xi)&:=\begin{pmatrix}
        0\\
        \frac{n-3}{\abs{\xi}^3}\left(\eta_{\epsilon}(\abs{\xi}\psi_{0,1}(\xi))-\eta_0(\abs{\xi}\psi_{0,1}(\xi))\right)
    \end{pmatrix} 
\end{align*}
and
\begin{align*}
    \mb{\widetilde{N}}_{\epsilon}(\mb{u}):=\begin{pmatrix}
        0\\
        \widetilde{N}_{\epsilon}(u_1) 
    \end{pmatrix}
\end{align*}
for
\begin{align*}
     \widetilde{N}_{\epsilon}(u_1)(\xi):= \frac{n-3}{\abs{\xi}^3}\left(\eta_{\epsilon}(\abs{\xi}(\psi_{0,1}(\xi)+u_1(\xi)))-\eta_{\epsilon}(\abs{\xi}\psi_{0,1}(\xi))-\eta_{\epsilon}'(\abs{\xi}\psi_{0,1}(\xi))\abs{\xi}u_1(\xi)\right).
\end{align*}

We are going to formulate Eq.~\eqref{perturbation equation} as a fixed point problem in suitable function spaces which we discuss in the following.

\subsection{Intersection Sobolev spaces}\label{Sec:functionspaces}

For two exponents $0 \leq s_1 < s_2$ we define an inner product on  $C_c^{\infty}(\mathbb{R}^n)$ by
\begin{align*}
    \langle u,v \rangle_{\dot H^{s_1}\cap \dot H^{s_2}(\mathbb{R}^n)} := \langle u,v \rangle_{\dot H^{s_1}(\mathbb{R}^n)} + \langle u,v \rangle_{\dot H^{s_2}(\mathbb{R}^n)},
\end{align*}
Moreover, for $\mb{u} = (u_1,u_2), \mb{v} = (v_1,v_2) \in C_c^{\infty}(\mathbb{R}^n) \times C_c^{\infty}(\mathbb{R}^n)$ and $1 \leq s_1 < s_2$ we set 
\begin{align*}
    \langle \mb{u},\mb{v} \rangle_{s_1,s_2}:=\langle u_1,v_1 \rangle_{\dot H^{s_1}\cap \dot H^{s_2}(\mathbb{R}^n)} + \langle u_2,v_2 \rangle_{\dot H^{s_1-1}\cap \dot H^{s_2-1}(\mathbb{R}^n)}.
\end{align*}
We define the Hilbert space of radial functions $\mc{H}_r^{s_1,s_2}$ as the completion of $C_{c,\,r}^{\infty}(\mathbb{R}^n) \times C_{c,\,r}^{\infty}(\mathbb{R}^n)$ under the norm induced by $\langle \cdot,\cdot \rangle_{s_1,s_2}$,

\[ \|\mb u\|_{s_1,s_2} := \sqrt{ \langle\mb  u, \mb u \rangle}_{s_1,s_2}.\]

In the following, we gather some properties, which are used throughout the paper.

First, we note that for $0\leq \widetilde{s_1}\leq s_1 < s_2 \leq \widetilde{s_2}$ there is the continuous embedding

\begin{align}\label{Sobolev embedding}
    \dot H^{\widetilde{s_1}}_{r}(\R^n)\cap \dot H^{\widetilde{s_2}}_{r}(\mathbb{R}^n) \hookrightarrow \dot H^{s_1}_{r}(\R^n)\cap \dot H^{s_2}_{r}(\mathbb{R}^n),
\end{align}

see Lemma 4.4 in \cite{Glo23}. In particular,  $ \mc{H}_r^{\widetilde{s_1},\,\widetilde{s_2}} \hookrightarrow  \mc{H}_r^{s_1,\,s_2}$.  Moreover, for
$m \in \N_0, 0\leq s_1  < \frac{n}{2}, s_2 > \frac{n}{2}+m$, 
\begin{align}\label{C^m embedding}
    \dot H^{s_1}_{r}(\R^n)\cap \dot H^{s_2}_{r}(\mathbb{R}^n) \hookrightarrow C^m_r(\R^n).
\end{align}
Hence, 
\begin{align}\label{C^m embedding two components}
     \mc{H}_r^{s_1,\,s_2} \hookrightarrow C_{r}^{m}(\mathbb{R}^n) \times C_{r}^{m-1}(\mathbb{R}^n),
\end{align}
for  $n \geq 3,\, m\geq 1, 1\leq s_1 < \frac{n}{2},  s_2 > \frac{n}{2}+m$,  where the rightmost space is equipped with the natural topology inherited from $W^{m,\infty}(\R^n) \times W^{m-1,\infty}(\R^n)$, see again Lemma 4.4 in \cite{Glo23}.

The next Lemma shows that if a radial function has sufficient decay at infinity it belongs to (intersection) radial homogeneous Sobolev spaces.

\begin{lemma}\label{Lemma: decay implies Sobolev}
    Let $n \in \N_{\geq 2}, k >0$ and $f \in C^{\infty}_{r}(\R^n)$. If we have for all $\beta \in \N_0^n$
    \begin{align}\label{decaying property}
        \abs{\partial^{\beta}f(x)} \lesssim \abs{x}^{-k-\abs{\beta}} 
    \end{align}
    for every $x\in\R^n$ then $f$ belongs to $\dot{H}^s_{r}(\R^n)$ for every $s\geq0$ with $s > \frac{n}{2}-k$.\\ In particular, if $1 \leq s_1 < s_2$, $s_1 > \frac{n}{2}-1$ and $\mb{f} = \left(f_1,f_2\right) \in C^{\infty}_{r}(\R^n) \times C^{\infty}_{r}(\R^n)$ satisfies for every multi-index $\beta \in \N^n_0$ 
    \begin{align}\label{decay two components}
        \abs{\partial^{\beta}f_1(x)} \lesssim \abs{x}^{-1-\abs{\beta}} \quad \text{and} \quad \abs{\partial^{\beta}f_2(x)} \lesssim \abs{x}^{-2-\abs{\beta}} 
    \end{align}
    for every $x\in\R^n$ then it belongs to $\mc{H}_r^{s_1,s_2}$. 
\end{lemma}

The proof of this statement is given in Appendix \ref{AppendixB}.

\begin{corollary}
We infer that for $n \geq 5$ the blowup solution $\mb \Psi_0 $ defined in Eq.~\eqref{Eq:Def_Psi0} belongs to $ \mc{H}_r^{s_1,s_2}$ for all $\frac{n}{2}-1 < s_1 < s_2$.
\end{corollary}

Furthermore, we have the following Banach algebra property.

\begin{lemma}\label{Lemma: Algebra property}
Let $n \geq 5$ and let $0 \leq s_1 < \frac{n}{2} < s_2$. Then $\dot{H}^{s_1}_r(\R^n) \cap \dot{H}^{s_2}_r(\R^n)$ is closed under multiplication, in particular, 
\begin{align}\label{Algebra property}
    \norm{fg}_{\dot{H}^{s_1}\cap \dot{H}^{s_2}(\mathbb{R}^n)} \lesssim \norm{f}_{\dot{H}^{s_1}\cap \dot{H}^{s_2}(\mathbb{R}^n)} \norm{g}_{\dot{H}^{s_1}\cap \dot{H}^{s_2}(\mathbb{R}^n)}
\end{align}
for all $f,g \in \dot{H}^{s_1}_r(\R^n)\cap \dot{H}^{s_2}_r(\mathbb{R}^n)$.
\end{lemma}

An elementary proof of this Lemma is given in \cite{GloKisSch23},  Appendix A.2.

Finally, the following version of a generalized Strauss-inequality has been proven for functions in $C_{c,r}^{\infty}(\R^n)$ in \cite{Glo23}, Proposition B.1 (the  general statement follows from a density argument together with the embedding from \eqref{Sobolev embedding}).

\begin{lemma}
    Let $n \geq 2, \frac{1}{2} < s <\frac{n}{2}$ and $0 \leq s_1 < \frac{n}{2} < s_2$. Then for every $\beta \in \N_0^n$ with $\abs{\beta} + \frac{n}{2} < s_2$ and $\abs{\beta} + s \geq s_1$
    \begin{align}\label{Eq: generalized Strauss inequality}
        \norm{\abs{\cdot}^{\frac{n}{2}-s}\partial^{\beta}u}_{L^{\infty}(\R^n)} \lesssim \norm{u}_{\dot{H}^{\abs{\beta}+s}(\R^n)}
    \end{align}
    for all $u \in \dot{H}_r^{s_1}(\R^n) \cap \dot{H}_r^{s_2}(\R^n)$.
\end{lemma}

\subsection{Known results for \texorpdfstring{$\epsilon = 0$}{}}\label{Section: Known Results}

The following results, which we summarize for later reference, have been proven in Propositions 4.1, 6.1 and 6.2 in \cite{Glo23}. 

\begin{proposition}\label{Prop: Spectral properties of L_0}
Let $n\geq5,\,\frac{n}{2}-1<s<\frac{n}{2}$ and $k>\frac{n}{2}+1,\, k\in \mathbb{N}$. Then the operators $\mb{\widetilde{L}}, \mb{\widetilde{L}}_0:  C^{\infty}_{c,r}(\R^n) \times C^{\infty}_{c,r}(\R^n) \subset \mc{H}_r^{s,k} \to \mc{H}_r^{s,k}$ are closable. For the closed operators $(\mb{L},\mathcal{D}(\mb{L}))$ and $(\mb{L}_0,\mathcal{D}(\mb{L}_0))$ we have  $\mathcal{D}(\mb{L}_0) = \mathcal{D}(\mb{L})$ and the following properties hold:
    \begin{itemize}[leftmargin=6mm]
	\setlength{\itemsep}{1mm}
	\item[\emph{i)}]
    	  (Free wave evolution) The operator $\mb{L}$  generates a strongly continuous semigroup $(\mb{S}(\tau))_{\tau\geq 0}$ of bounded operators on $\mc{H}_r^{s,k}$, which satisfies
   \begin{align}\label{semigroup S estimate}
    \norm{\mb{S}(\tau)\mb{u}}_{s,k}\leq e^{\left(\frac{n}{2}-1-s\right)\tau}\norm{\mb{u}}_{s,k}
\end{align}
for all $\mb{u} \in \mc{H}^{s,k}_r$ and $\tau\geq0$.
       \item[\emph{ii)}] (Spectral gap property of $\mb L_0$) There exists $0 < \widetilde{\omega} < s + 1 - \frac{n}{2}$ such that
        \begin{align*}
            \{ \lambda \in \sigma(\mb{L}_0): \Re \lambda \geq -\widetilde{\omega} \} =\{1\}.
        \end{align*}
        The point $\lambda = 1$ is a simple eigenvalue with an explicit eigenfunction
        \begin{align*}
            \mb{g}_0:= \begin{pmatrix}
        g_0\\
        2 g_0 + \Lambda g_0
    \end{pmatrix}, \quad \text{where} \quad g_0(\xi)= \psi_{0,1}(\xi) + \Lambda \psi_{0,1}(\xi) = \frac{1}{\abs{\xi}^2+n-4}. 
        \end{align*}
        Moreover, the map $\mb{P}_0: \mc{H}_r^{s,k} \to \mc{H}_r^{s,k}$ defined by
        \begin{align*}
            \mb{P}_0 := \frac{1}{2\pi i} \int_{\partial B_{1\slash 2}(1)}\mb{R}_{\mb{L}_0}(\lambda)\, d\lambda
        \end{align*}
        is a bounded projection with $\ran (\mb{P}_0) = \ker (\mb{I}-\mb{L}_0) = \langle \mb{g}_0 \rangle$. 
       \item[\emph{iii)}] (Linearized evolution)The operator $\mb{L}_0$  generates a strongly continuous semigroup $(\mb{S}_0(\tau))_{\tau\geq 0}$ of bounded operators on $\mc{H}_r^{s,k}$. Furthermore, for any $0 < \omega_0 < \tilde \omega$ there is a $C \geq 1$ such that   
        \begin{align}\label{decay on stable subspace}
            \norm{\mb{S}_0(\tau)(\mb{I}-\mb{P}_0)\mb{u}}_{s,k} \leq C e^{-\omega_0 \tau} \norm{(\mb{I}-\mb{P}_0)\mb{u}}_{s,k} 
        \end{align}
        for all $\mb{u} \in \mc{H}^{s,k}_r$ and all $\tau\geq0$.
    \end{itemize}
\end{proposition}

We remark that the exponent $s$ in the above Proposition is chosen to be strictly greater than $\frac{n}{2}-1$ so that one has exponential decay of the free part $\mb{L}$. The fact that $0$ does not lie in the spectrum of $\mb{L}_0$ will be key for the subsequent analysis.

Now throughout the main parts of the proofs provided in Section \ref{Section: Existence Blowup Solution} and Section \ref{Section: Stability Analysis}, we assume $n \geq 5$ and work in $\mc H^{s,k}_r(\R^n)$ assuming that $(s,k) \in \R \times \N$ satisfy
    \begin{align}\label{condition on exponents}
        \frac{n}{2}-1 < s \leq \frac{n}{2}-1 + \frac{1}{2n-2}, \quad k>n.
    \end{align}

This will allow us to prove  a parameter dependent Schauder type estimate akin to \cite{Glo23},  which we state in Proposition \ref{Prop: Schauder}, Appendix \ref{AppendixA}.

\section{Existence of blowup solutions into perturbed spheres}{}\label{Section: Existence Blowup Solution}

In the following, we construct for sufficiently small $\epsilon$ a solution $\mb \Phi_{\epsilon}$ to Eq.~\eqref{perturbation equation}. By Proposition \ref{Prop: Spectral properties of L_0}, the operator $\mb{L}_0$ is bounded invertible and thus \eqref{perturbation equation} can be phrased as a fixed point problem. 

First, we properly define the operators appearing on the right hand side of the equation. To prove the required Lipschitz bounds, we will make extensive use of  Proposition \ref{Prop: Schauder} in Appendix \ref{AppendixA} and assume that the exponents $(s,k) \in \R \times \N$ satisfy Eq.~\eqref{condition on exponents}. We note that under this assumption,  $\mc{H}_r^{s,k} \hookrightarrow C_r^2(\R^n) \times C_r^1(\R^n)$ and thus the nonlinearity, the potential and the remainder can be defined pointwise as in Eq.~\eqref{perturbation equation}.

In the following we use the notation
\[ \mc{B}_{\delta} := \{ \mb u \in \mc{H}_r^{s,k}: \| \mb u \|_{s,k} \leq \delta \} \subset \mc{H}_r^{s,k}. \]

To proceed, we state an auxiliary lemma for the function $\eta_{\epsilon}$, see Eq.~ \eqref{eta},  that will be useful in the following. 
\begin{lemma}\label{lemma: eta}
    Let $a,b,c \in \R$ and $\abs{\epsilon} \leq \epsilon_0$. Then
    \begin{align}\label{first auxiliary equation}
        \eta_{\epsilon}(a) = a^3\int_0^1\int_0^1\int_0^1 x^2y\,\eta_{\epsilon}'''(axyz)\, dz \, dy \, dx,
    \end{align}
    \begin{align}\label{second auxiliary equation}
        \eta_{\epsilon}'(a) = a^2\int_0^1\int_0^1x\,\eta_{\epsilon}'''(axy)\, dy \, dx,
    \end{align}
 and
    \begin{align}\label{third auxiliary equation}
     &    \eta_{\epsilon}(a+c)  - \eta_{\epsilon}(a+b) - \eta_{\epsilon}'(a)(c-b)\nonumber \\
         &  = (c-b) \int_0^1\int_0^1\int_0^1 (b+x(c-b)) (a+y(b+x(c-b)))
      \eta_{\epsilon}'''(z(a+y(b+x(c-b))))\, dz \, dy \, dx.
    \end{align}
    Moreover, for every $\ell \in \N_0$ there exists a constant $C_{\ell}$ such that
    \begin{align}\label{eta fulfills conditions for Schauder estimate}
        \abs{\eta_{\epsilon}^{(\ell+3)}(x)-\eta_{\kappa}^{(\ell+3)}(y)} \leq C_{\ell}\left(\abs{\epsilon-\kappa} +\abs{x-y}\right)
    \end{align}
    holds for every $x,y \in \R$ and $\abs{\epsilon},\abs{\kappa} \leq \epsilon_0$.
\end{lemma}
\begin{proof}
    The first three equalities follow from the fact that $\eta_{\epsilon}(0) = \eta_{\epsilon}'(0) = \eta_{\epsilon}''(0) = 0$ together with the fundamental theorem of calculus. We only prove the first equality, since the other two follow from similar ideas.
    \begin{align*}
        \eta_{\epsilon}(a) = \int_0^a \eta_{\epsilon}'(x) \, dx &= \int_0^a\int_0^x \eta_{\epsilon}''(y) \, dy \, dx = \int_0^a\int_0^x\int_0^y \eta_{\epsilon}'''(z) \, dz \, dy \, dx \\ &= a^3\int_0^1\int_0^1\int_0^1 x^2y\,\eta_{\epsilon}'''(axyz)\, dz \, dy \, dx.
    \end{align*}
    For \eqref{eta fulfills conditions for Schauder estimate} one observes that $\eta_{\epsilon}^{(\ell+3)}$ can be written for every $\ell \in \N_0$ as
    \begin{align*}
        \eta_{\epsilon}^{(\ell+3)} = \eta_0^{(\ell+3)} + \epsilon \, F_1^{(\ell+3)} + \epsilon^2 F_2^{(\ell+3)}
    \end{align*}
    where $\eta_0''', F_1'''$ and $F_2'''$ are smooth $2\pi$-periodic functions.
    
\end{proof}

\begin{lemma}\label{Local Lipschitz estimate for nonlinearity}
Let $n \geq 5$ and  $(s,k) \in \R \times \N$ satisfy \eqref{condition on exponents}. Then, for any $\epsilon \in \R$ with $|\epsilon| \leq \epsilon_0$, $\mb{\widetilde{N}}_{\epsilon}:\mc{H}_r^{s,k} \to \mc{H}_r^{s,k+1}$ and the estimate
     \begin{align}\label{estimate for nonlinearity}
        \| \mb{\widetilde{N}}_{\epsilon}(\mb{u})-\mb{\widetilde{N}}_{\kappa}(\mb{v}) \|_{s,k+1} \lesssim \left(\norm{\mb{u}}_{s,k}+\norm{\mb{v}}_{s,k}\right)\norm{\mb{u}-\mb{v}}_{s,k} + \left(\norm{\mb{u}}_{s,k}^2+ \norm{\mb{v}}_{s,k}^2\right) \abs{\epsilon-\kappa}
    \end{align}
holds for all $\abs{\epsilon},\abs{\kappa}\leq \epsilon_0$ and all $\mb{u},\mb{v}\in \mc{B}_{\delta} \subset \mc{H}_r^{s,k}$ for $0<\delta\leq1$.
   
\end{lemma}

\begin{proof}
We start by showing the Lipschitz bound
    \begin{align}\label{nonlinear estimate for fixed epsilon}
       \| \mb{\widetilde{N}}_{\epsilon}(\mb{u})-\mb{\widetilde{N}}_{\epsilon}(\mb{v}) \|_{s,k+1} \lesssim \left(\norm{\mb{u}}_{s,k}+\norm{\mb{v}}_{s,k}\right)\norm{\mb{u}-\mb{v}}_{s,k}
    \end{align}
    for all $\abs{\epsilon}\leq \epsilon_0$ and all $\mb{u},\mb{v} \in \mc{B}_{\delta}$, $0<\delta\leq1$. With the help of \eqref{third auxiliary equation} we write
	\begin{align*}
		& \widetilde{N}_{\epsilon}(u_1)(\xi)-\widetilde{N}_{\epsilon}(v_1)(\xi) = \nonumber \\ &\frac{n-3}{|\xi|^3}\Big(\eta_{\epsilon}\big(|\xi|(\psi_{0,1}(\xi)+u_1(\xi))\big)-\eta_{\epsilon}\big(|\xi|(\psi_{0,1}(\xi)+v_1(\xi))\big) - \eta_{\epsilon}'(|\xi|\psi_{0,1}(\xi))|\xi|(u_1(\xi)-v_1(\xi)) \Big) \nonumber \\
		& 
		= (n-3)\int_{0}^{1}\int_{0}^{1}\int_{0}^{1} (u_1(\xi)-v_1(\xi))\big(v_1(\xi)+x(u_1(\xi)-v_1(\xi))\big)
			 \nonumber \\ & \cdot  \left( \psi_{0,1}(\xi) + y\big(v_1(\xi)+x(u_1(\xi)-v_1(\xi))\big) \right)  \cdot \eta_{\epsilon}'''\left( z |\xi| \big(\psi_{0,1}(\xi)+ y(v_1+x(u_1-v_1))\big) \right)dz \,dy\, dx. 
	\end{align*}
Since $F_{\epsilon}(x):=\eta_{\epsilon}'''(x)$ satisfies the assumptions of Proposition \ref{Prop: Schauder} the estimate (\ref{nonlinear estimate for fixed epsilon}) follows from the previous equation and (\ref{Schauder estimate}). We now show that 
\begin{align}\label{Ne-Nk}
        \| \mb{\widetilde{N}}_{\epsilon}(\mb{u})-\mb{\widetilde{N}}_{\kappa}(\mb{u}) \|_{s,k+1} \lesssim \norm{\mb{u}}_{s,k}^2\abs{\epsilon-\kappa}
\end{align}
for every $\mb{u}\in \mc{B}_{\delta}$ with $0< \delta \leq 1$ and every $\abs{\epsilon}, \abs{\kappa} \leq \epsilon_0$. For this, we use again Eq.~ \eqref{third auxiliary equation} to obtain 
\begin{align*}
		\widetilde{N}_{\epsilon}(u_1)(\xi)-\widetilde{N}_{\kappa}(u_1)(\xi) & = (n-3) \,u_1(\xi)^2\int_{0}^{1} \int_{0}^{1}\int_{0}^{1}x\left( \psi_{0,1}(\xi) + yxu_1(\xi) \right)\\&  \cdot \left(\eta_{\epsilon}'''\left( z |\xi| \big(\psi_{0,1}(\xi)+ yxu_1(\xi) \big) \right)-\eta_{\kappa}'''\left( z |\xi| \big(\psi_{0,1}(\xi)+ yxu_1(\xi)\big) \right)\right)dz \,dy\, dx
\end{align*}
and Eq.~\eqref{Ne-Nk} follows from an application of Proposition \ref{Prop: Schauder}. By combining Eq.~\eqref{nonlinear estimate for fixed epsilon} and Eq.~\eqref{Ne-Nk} one obtains \eqref{estimate for nonlinearity}. The mapping properties of $\mb {\widetilde{N}}_{\epsilon}$ follow from Eq.~\eqref{estimate for nonlinearity} and the fact that  $\mb {\widetilde{N}_{\epsilon}}(\mb 0) = \mb 0$.
\end{proof}

For the potential, we have the following result.

\begin{lemma}\label{estimate of potential of ground state}
Let $n \geq 5$ and  $(s,k) \in \R \times \N$ satisfy \eqref{condition on exponents}. For any $\epsilon \in \R$ with $|\epsilon| \leq \epsilon_0$, $\mb{V}_{\epsilon}(\mb \Psi_0): \mc{H}_r^{s,k} \to \mc{H}_r^{s,k+1}$. Furthermore,
    \begin{align*}
        \norm{\mb{V}_{\epsilon}(\mb \Psi_0)\mb{u}-\mb{V}_{\kappa}(\mb  \Psi_0)\mb{u}}_{s,k+1}\lesssim\abs{\epsilon-\kappa} \norm{\mb{u}}_{s,k}
    \end{align*}
 for all $\abs{\epsilon}, \abs{\kappa} \leq \epsilon_0$  and all $\mb{u} \in \mc{H}_r^{s,k}$.
\end{lemma}
\begin{proof}
By definition of $\mb{V}_{\epsilon}(\mb  \Psi_0)$ we have to prove the estimate
    \begin{align*}
        \norm{\left(V_{\epsilon}(\psi_{0,1})-V_{\kappa}(\psi_{0,1})\right)u_1}_{\dot{H}^{s-1}\cap\dot{H}^k(\R^n)} \lesssim \abs{\epsilon-\kappa} \norm{u_1}_{\dot{H}^s\cap\dot{H}^k(\R^n)}.
    \end{align*}
    In a similar way to the previous result one shows with \eqref{second auxiliary equation} that
    \begin{align*}
        \left((V_{\epsilon}(\psi_{0,1})  -V_{\kappa }(\psi_{0,1}))u_1\right) & (\xi)   \simeq  \\
       &   \psi_{0,1}(\xi)^2u_1(\xi)\int_0^1\int_0^1x\left(\eta_{\epsilon}'''(\abs{\xi}\psi_{0,1}(\xi)xy)-\eta_{\kappa}'''(\abs{\xi}\psi_{0,1}(\xi)xy)\right)dy\,dx
    \end{align*}
holds for every $\xi \in \R^n.$ The claim then again follows by an application of Proposition \ref{Prop: Schauder}.    
\end{proof}

Finally, we consider the remainder.

\begin{lemma}\label{Lemma: Remainder Term}
Let $n \geq 5$ and  $(s,k) \in \R \times \N$ satisfy \eqref{condition on exponents}. For any $\epsilon \in \R$ with $|\epsilon| \leq \epsilon_0$, $\mc{R}_{\epsilon}(\mb \Psi_0) \in \mc{H}_r^{s,k+1} \subset \mc{H}_r^{s,k}$. Furthermore,
    \begin{align*}
        \norm{\mc{R}_{\epsilon}(\mb  \Psi_0)-\mc{R}_{\kappa}(\mb  \Psi_0)}_{s,k+1} \lesssim \abs{\epsilon-\kappa}
    \end{align*}
\end{lemma}
for all  $\abs{\epsilon}, \abs{\kappa} \leq \epsilon_0$.

\begin{proof}
By the definition of $\mc{R}_{\epsilon}(\mb  \Psi_0)$ we have to estimate $\xi \mapsto\frac{n-3}{\abs{\xi}^3}\left(\eta_{\epsilon}(\abs{\xi}\psi_{0,1}(\xi))-\eta_{\kappa}(\abs{\xi}\psi_{0,1}(\xi))\right)$ in the $\dot{H}^{s-1}\cap\dot{H}^k-$norm. To do so we write with the help of \eqref{first auxiliary equation}
    \begin{align*}
        \frac{n-3}{\abs{\xi}^3}&\left(\eta_{\epsilon}(\abs{\xi}\psi_{0,1}(\xi))-\eta_{\kappa}(\abs{\xi}\psi_{0,1}(\xi))\right)\\ = &(n-3) \psi_{0,1}(\xi)^3\int_0^1\int_0^1\int_0^1x^2y\left(\eta_{\epsilon}'''(\abs{\xi}\psi_{0,1}(\xi)xyz)-\eta_{\kappa}'''(\abs{\xi}\psi_{0,1}(\xi)xyz)\right)dz\,dy\,dx
    \end{align*}
    and conclude the proof by again applying the Schauder estimate from Proposition \ref{Prop: Schauder} to the function $F_{\epsilon}(x) := \eta_{\epsilon}'''(x).$   
\end{proof}

In order to prove gradient blowup as claimed in Theorem \ref{Theorem: Blowup Solution}, we impose an a priori smallness condition on the perturbation.

\begin{lemma}\label{Le:Smallness_Condition}
Let $n \geq 5$ and  $(s,k) \in \R \times \N$ satisfy \eqref{condition on exponents}. There is a $\delta^*_0 > 0$ such that any $\mb u = (u_1,u_2) \in  \mc{B}_{\delta^*_0}$ satisfies  
\[|u_1(0) | < \psi_{0,1}(0),\]
where \[ \psi_{0,1}(\xi) = 2 |\xi|^{-1} \arctan\left(\tfrac{|\xi|}{\sqrt{n-4}}\right). \]
\end{lemma}

\begin{proof}
First, note that $\psi_{0,1}(0) > 0$. By Sobolev embedding, there is a $C > 0$ such that 
\[ |u_1(0) | \leq  \|u_1\|_{L^{\infty}(\R^n)} \leq C \| \mb u  \|_{s,k} \leq C \delta^*_0 \]
for all $\mb u  \in \mc{B}_{\delta^*_0} \subset \mc H^{s,k}_r$. Now choose $\delta^*_0$ sufficiently small. 
\end{proof}

The next result crucially relies on the characterization of the spectrum of the linearized operator for $\epsilon = 0$, see Proposition \ref{Prop: Spectral properties of L_0}.

\begin{proposition}\label{Prop: contraction operator}
Let $n \geq 5$ and  $(s,k) \in \R \times \N$ satisfy \eqref{condition on exponents}. There exist $0 < \delta^* \leq \min\{1, \delta_0^*\}$ and $0  < \epsilon^* \leq \epsilon_0$ (depending on $\delta^*$) such that for every $\epsilon \in \R$ with $\abs{\epsilon}\leq\epsilon^*$ the map
    \begin{align}\label{Eq:K}
        \mb{K}_{\epsilon}: \mc{B}_{\delta^*} \to \mc{B}_{\delta^*},\quad  \mb{K}_{\epsilon}(\mb \Phi) := (-\mb{L}_0)^{-1} (\mb{V}_{\epsilon}(\mb \Psi_0)\mb \Phi+\mb{\mathcal{R}}_{\epsilon}(\mb \Psi_0)+\mb{\widetilde{N}}_{\epsilon}(\mb \Phi) ) 
    \end{align} is a well-defined contraction. In particular,  $\delta^*$ and $\epsilon^*$ can be chosen such that  
    \begin{align}\label{uniform contraction constant}
        \norm{\mb{K}_{\epsilon}(\mb \Phi)-\mb{K}_{\epsilon}(\mb \Psi)}_{s,k} \leq \frac{1}{2} \norm{\mb \Phi - \mb \Psi}_{s,k}
    \end{align}
    holds for every $\abs{\epsilon} \leq \epsilon^*$ and all $\mb \Phi,\mb \Psi \in \mc{B}_{\delta^*}$.
\end{proposition}
\begin{proof}
Using the fact that  $\mb{L}_0: \mc D(\mb{L}_0) \subset \mc H^{s,k} \to \mc H^{s,k}$ is bounded invertible together with the embedding  \eqref{Sobolev embedding} there exists a constant $C>0$, such that
    \begin{align*}
        \norm{\mb{K}_{\epsilon}(\mb \Phi)}_{s,k} \leq C\left( \norm{\mb{V}_{\epsilon}(\mb \Psi_0)\Phi}_{s,k+1}+\norm{\mb{\mathcal{R}}_{\epsilon}(\mb \Psi_0)}_{s,k+1}+\lVert\mb{\widetilde{N}}_{\epsilon}(\mb \Phi)\rVert_{s,k+1}\right)
    \end{align*} holds for every $\mb \Phi\in \mc{B}_{\delta^*}$ and every $\abs{\epsilon}\leq\epsilon^*$ assuming that  $0 < \delta^* \leq \min\{1, \delta_0^*\}$ and $0  < \epsilon^* \leq \epsilon_0$. We then have by the previous results:
    \begin{align*}
        \norm{\mb{V}_{\epsilon}(\mb \Psi_0)\mb \Phi}_{s,k+1}+\norm{\mb{\mathcal{R}}_{\epsilon}(\mb \Psi_0)}_{s,k+1}+\lVert\mb{\widetilde{N}}_{\epsilon}(\mb \Phi)\rVert_{s,k+1} \leq  \left(C_1\,\epsilon^* + C_2\,\epsilon^*/\delta^* +C_3\,\delta^*\right)\,\delta^*
    \end{align*}
    for  constants $C_j > 0$, $j \in \{1,2,3\}$ independent of $\epsilon^*$ and $\delta^*$. Now choose $\delta^*<1/(C_3\,C)$ and  $\epsilon^* = \epsilon^*(\delta^*)$ sufficiently small to guarantee $C(C_1\,\epsilon^* + C_2\,\epsilon^*/\delta^* +C_3\,\delta^*)\leq 1$. Then
    \begin{align*}
        \norm{\mb{K}_{\epsilon}(\mb \Phi)}_{s,k} \leq \delta^* 
    \end{align*} for all $\mb \Phi\in \mc{B}_{\delta^*}$ so that $\mb{K}_{\epsilon}$ is well-defined on $\mc{B}_{\delta^*}$. 
For the contraction property we take $\mb \Phi,\mb \Psi\in\mc{B}_{\delta^*}$ and get
    \begin{align*}
        \norm{\mb{K}_{\epsilon}(\mb \Phi)-\mb{K}_{\epsilon}(\mb \Psi)}_{s,k} &\leq C\left( \norm{\mb{V}_{\epsilon}(\mb \Psi_0)\,\left(\mb \Phi-\mb \Psi\right)}_{s,k+1}+ \lVert\mb{\widetilde{N}}_{\epsilon}(\mb \Phi)-\mb{\widetilde{N}}_{\epsilon}(\mb \Psi)\rVert_{s,k+1}\right)\\ &\leq C\, C_1\,\epsilon^* \norm{\mb \Phi-\mb \Psi}_{s,k} + C\, C_4 \, \delta^* \norm{\mb \Phi- \mb \Psi}_{s,k},
    \end{align*}
for $C_4 > 0$. Now choose $\delta^*>0$ possibly even smaller so that $C\,C_4\,\delta^* < 1/2$ is satisfied. Hence the claim follows by adjusting $\epsilon^*$ and requiring in addition that  $C\,C_1\,\epsilon^*+C\,C_4\,\delta^*\leq 1/2$.
    
\end{proof}

This result immediately implies the existence part in the following statement. 

\begin{proposition}\label{Prop: perturbation}
Let $n \geq 5$ and  $(s,k) \in \R \times \N$ satisfy \eqref{condition on exponents}. Let $\delta^*$ and $\epsilon^*$ be as in Proposition \ref{Prop: contraction operator}.  Then for every $\abs{\epsilon}\leq\epsilon^*$ there is a 
unique $\mb \Phi_{\epsilon} \in \mc{B}_{\delta^*} \cap \mc D(  \mb{L}_0)$ satisfying
\begin{align}\label{Eq:Phi_eps}
   \mb{L}_0 \mb \Phi_{\epsilon}+\mb{V}_{\epsilon}(\mb \Psi_0)\mb \Phi_{\epsilon}+\mb{\mathcal{R}}_{\epsilon}(\mb \Psi_0)+\mb{\widetilde{N}}_{\epsilon}(\mb \Phi_{\epsilon}) =0
\end{align}
in $\mc{H}_r^{s,k}$. Furthermore, $\mb \Phi_{\epsilon}$ depends Lipschitz continuously on $\epsilon$ in the sense that
\begin{align*}
   \| \mb \Phi_{\epsilon}-\mb \Phi_{\kappa}\|_{s,k} \lesssim |\epsilon-\kappa|
\end{align*} holds for all $\abs{\epsilon},\abs{\kappa}\leq \epsilon^*$.
Finally, $\mb \Phi_{\epsilon} \in \bigcap\limits_{\ell \geq s}\dot{H}^{\ell}(\R^n) \times \dot{H}^{\ell-1}(\R^n) \subset C^{\infty}_r(\R^n)\times C^{\infty}_r(\R^n)$ and $\mb \Phi_{\epsilon}$ solves the equation in a classical sense, i.e., $ \mb{L}_0 \mb \Phi_{\epsilon} = \widetilde{ \mb{L}}_0 \mb \Phi_{\epsilon}$. 
\end{proposition}

\begin{proof}
The existence of a unique $\mb \Phi_{\epsilon} \in \mc{B}_{\delta^*} \cap \mc D(  \mb{L}_0)$  satisfying the Eq.~\eqref{Eq:Phi_eps} is a direct consequence of Proposition \ref{Prop: contraction operator} and the contraction mapping principle. For the Lipschitz bound we take $\abs{\epsilon},\abs{\kappa}\leq\epsilon^*$ and use \eqref{uniform contraction constant} to obtain
\begin{align*}
     \norm{\mb \Phi_{\epsilon}-\mb \Phi_{\kappa}}_{s,k} = \norm{\mb{K}_{\epsilon}(\mb \Phi_{\epsilon})-\mb{K}_{\kappa}(\mb \Phi_{\kappa})}_{s,k} \leq  \frac{1}{2} \norm{\mb\Phi_{\epsilon}-\mb \Phi_{\kappa}}_{s,k} + \norm{\mb{K}_{\epsilon}(\mb \Phi_{\kappa})-\mb{K}_{\kappa}(\mb \Phi_{\kappa})}_{s,k}
\end{align*}
from which $\norm{\mb \Phi_{\epsilon}-\mb \Phi_{\kappa}}_{s,k} \leq 2\norm{\mb{K}_{\epsilon}(\mb \Phi_{\kappa})-\mb{K}_{\kappa}(\mb \Phi_{\kappa})}_{s,k}$ follows. Now,  for every $\mb \Phi\in\mc{B}_{\delta^*}$ 
\begin{align*}
   \norm{\mb{K}_{\epsilon}(\mb \Phi)-\mb{K}_{\kappa}(\mb \Phi)}_{s,k}  & \lesssim \norm{\mc{R}_{\epsilon}(\mb \Psi_0)-\mc{R}_{\kappa}(\mb \Psi_0)}_{s,k} + \norm{\left(\mb{V}_{\epsilon}(\mb \Psi_0)-\mb{V}_{\kappa}(\mb \Psi_0)\right)\mb \Phi}_{s,k}  \\
   & + \|\mb{\widetilde{N}}_{\epsilon}(\mb\Phi)-\mb{\widetilde{N}}_{\kappa}(\mb\Phi)\|_{s,k} \lesssim \abs{\epsilon-\kappa}
\end{align*} 
due to Lemma \ref{Local Lipschitz estimate for nonlinearity} and Lemma \ref{estimate of potential of ground state}.

For the regularity we use the Sobolev embedding from \eqref{C^m embedding two components} and show that  $\mb \Phi_{\epsilon} \in \mc{H}_r^{s,\ell}$ for every $\ell \geq k$. Since $\mb\Phi_{\epsilon} \in \mc{H}_r^{s,k}$,  Lemma \ref{Local Lipschitz estimate for nonlinearity} and Lemma \ref{estimate of potential of ground state} yield $\mb{V}_{\epsilon}(\mb \Psi_0)\mb \Phi_{\epsilon}+\mb{\mathcal{R}}_{\epsilon}(\mb \Psi_0)+\mb{\widetilde{N}}_{\epsilon}(\mb \Phi_{\epsilon}) \in \mc{H}_r^{s,k+1}$. We now want to use Eq.~\eqref{Eq:K} to deduce that  $\mb \Phi_{\epsilon}$ must belong to $\mc{H}_r^{s,k+1}$.\\
    For this we will at this particular point not suppress the dependency of $\mb{L}_0$ on the parameters $s$ and $k$ and will therefore denote by $\mb{L}_{0,s_1,s_2}$ the closure of $\widetilde{\mb{L}}_{0,s_1,s_2}$ in $\mc{H}_r^{s_1,s_2}$. Eq.~ \eqref{Eq:K} then writes
    \begin{align*}
        \mb \Phi_{\epsilon}= (-\mb{L}_{0,s,k})^{-1}\left(\mb{V}_{\epsilon}( \mb \Psi_0)\mb \Phi_{\epsilon}+\mb{\mathcal{R}}_{\epsilon}( \mb \Psi_0)+\mb{\widetilde{N}}_{\epsilon}( \mb \Phi_{\epsilon})\right)
    \end{align*} and all what is left to show is that if we restrict the resolvent of $\mb{L}_{0,s,k}$ (in 0) to the subspace $\mc{H}_r^{s,k+1}$ then this operator equals the resolvent (in 0) of the operator $\mb{L}_{0,s,k+1}$. But this follows from the fact that $\mb{L}_{0,s,k+1}$ is a restriction of $\mb{L}_{0,s,k}$ to the subspace $\mc{H}_r^{s,k+1}$ (see for example \cite{CsoGloSch23}, p.21, Lemma 3.5 or \cite{Glo23}, p.29, Lemma C.1.). Via repeating this argument we conclude inductively that $\mb \Phi_{\epsilon}$ belongs to $\mc{H}_r^{s,\ell}$ for every $\ell \geq k$ and is therefore smooth by \eqref{C^m embedding two components}.\medskip\\  
Finally, we show that $\mb{L}_0$ acts as a classical differential operator on $\mb \Phi_{\epsilon}$. Since $\mb \Phi_{\epsilon}$ belongs to $\mc{D}(\mb{L}_0)$ there exists a sequence $\left(\mb{u}_j\right)_{j\in\N} \subset C^{\infty}_{c,r}(\R^n) \times C^{\infty}_{c,r}(\R^n)$ such that 
    \begin{align*}
        \mb{u}_j \to \mb \Phi_{\epsilon} \quad \text{and} \quad \widetilde{\mb{L}}_0 \mb{u}_j \to \mb{L}_0  \mb \Phi_{\epsilon} \quad \text{in} \quad \mc{H}_r^{s,k} \quad \text{for} \quad j \to \infty.
    \end{align*}
Now, by the embedding of $\mc{H}_r^{s,k}$ into $C^2(\R^n) \times C^1(\R^n)$  we infer by the above convergence of $\mb{u}_j$ towards $\mb \Phi_{\epsilon}$ that each term in $\widetilde{\mb{L}}_0 \mb{u}_j$ converges pointwise towards $\widetilde{\mb{L}}_0  \mb \Phi_{\epsilon}$. Again by Sobolev embedding, the limit $\widetilde{\mb{L}}_0 \mb{u}_j \to \mb{L}_0  \mb \Phi_{\epsilon}$ holds pointwise, which implies the claim. This last part of the proof shows in fact that $\widetilde{\mb{L}}_0$ acts as a classical differential operator on any element in its domain, irrespective of higher regularity. 
\end{proof}

\begin{corollary}\label{Corr: existence of blowup solution}
Let $n \geq 5$ and  $(s,k) \in \R \times \N$ satisfy \eqref{condition on exponents}. Let $\delta^*$ and $\epsilon^*$ be as in Proposition \ref{Prop: perturbation}. For $\epsilon \in \R$ with $\abs{\epsilon}\leq\epsilon^*$ let $\mb \Phi_{\epsilon} \in \mc{B}_{\delta^*}$ be the corresponding solution of Eq.~\eqref{Eq:Phi_eps}. Then 
 \[ \mb \Psi_{\epsilon}:= \mb \Psi_0+\mb \Phi_{\epsilon}  \in \bigcap\limits_{\ell \geq s}\dot{H}^{\ell}_r(\R^n) \times \dot{H}^{\ell-1}_r(\R^n) \subset C^\infty_{r}(\R^n) \times C^\infty_{r}(\R^n)\]
solves Eq.~\eqref{static equation} in a classical sense. Furthermore, 
    \begin{align}\label{Lipschitz continuity of the blowup solution}
        \norm{\mb \Psi_{\epsilon} - \mb \Psi_{\kappa}}_{s,k} \lesssim \abs{\epsilon - \kappa}
    \end{align}
for all $\abs{\epsilon}, \abs{\kappa} \leq \epsilon^*$.    
\end{corollary}
    
Finally, we characterize the decay of the solution $\mb \Psi_{\epsilon} = (\psi_{\epsilon,1},\psi_{\epsilon,2})$ at infinity. Corollary \ref{Corr: existence of blowup solution} and
the generalized Strauss inequality \eqref{Eq: generalized Strauss inequality} imply that 
$|\psi_{\epsilon,1}(\xi)| \lesssim \langle \xi \rangle^{\frac{n}{2}-s}$ for $s >\frac{n}{2}-1$ and hence the decay rate is strictly less than 1. In the following,  we provide a more refined analysis based on ODE arguments  to improve this. This is along the lines of \cite{KavWei90}.

\begin{proposition}\label{Prop: decay of blowup solution}
Let $n \geq 5$, $\frac{n}{2} - 1 < s < \frac{n-1}{2}$ and $|\epsilon| \leq \epsilon^*$. Suppose that $\mb \Psi = (\psi_1,\psi_2) \in \bigcap\limits_{\ell \geq s}\dot{H}^{\ell}_r(\R^n) \times \dot{H}^{\ell-1}_r(\R^n) \subset C^{\infty}_r(\R^n) \times C^{\infty}_r(\R^n)$ is a classical solution to Eq.~\eqref{static equation}. Then its components satisfy for every $\beta \in \N^n_0$,
    \begin{align}\label{decaying properties of blowup solution}
        \abs{\partial^{\beta}\psi_1(\xi)} \lesssim_{\epsilon} \langle \xi \rangle^{-1-\abs{\beta}} \quad \text{and} \quad \abs{\partial^{\beta}\psi_2(\xi)} \lesssim_{\epsilon} \langle \xi \rangle^{-2-\abs{\beta}}
    \end{align}
    for all $\xi \in \R^n$. Furthermore, the limits $\lim\limits_{\abs{\xi} \to \infty}\abs{\xi} \psi_1(\xi)$ and $\lim\limits_{\abs{\xi} \to \infty}\abs{\xi}^2( \psi_1(\xi) + \Lambda  \psi_1(\xi))$ exist.
\end{proposition}

\begin{proof}
 Since $\psi_1$ and $\psi_2$ are smooth radial functions we know that there exist $\widetilde{\psi}_1, \widetilde{\psi}_2 \in C^{\infty}_e[0,\infty)$ such that $\psi_1(\xi) = \widetilde{\psi}_1(\abs{\xi})$ and $\psi_2(\xi) = \widetilde{\psi}_2(\abs{\xi})$ for all $\xi \in \R^n$. From \eqref{evolution equation} we  obtain $\widetilde{\psi}_2(\rho) = \widetilde{\psi}_1(\rho) + \rho \, \widetilde{\psi}'_1(\rho)$ for every $\rho \geq 0$ so that $\widetilde{\psi}_1$ solves the following ordinary differential equation:
\begin{align}\label{ODE for first component}
    (1-\rho^2) u''(\rho) + \left(\frac{n-1}{\rho}-4\rho\right)u'(\rho) - 2 u(\rho) + \frac{n-3}{\rho^3}\eta_{\epsilon}(\rho \,u(\rho)) = 0.
\end{align}
We now argue as in \cite{KavWei90}, p.1388, Theorem 3.1.
With \eqref{first auxiliary equation} the equation can be written as
\begin{align}\label{cubic wave equation with error term}
    (1-\rho^2) u''(\rho) + \left(\frac{n-1}{\rho}-4\rho\right)u'(\rho) - 2 u(\rho) +u^3(\rho) F_{\epsilon}(\rho \,u(\rho)) = 0
\end{align}
for an $F_{\epsilon} \in C^{\infty}_e(\R)$ which is uniformly bounded in $\rho$ and $\abs{\epsilon} \leq \epsilon^*$. After substituting $r = \log(\rho)$ in \eqref{cubic wave equation with error term} we obtain the following equation for $v(r)= u(e^r)$ 
\begin{align}\label{Eq after log transform}
    v''(r) + 3 v'(r) + 2 v(r) - v^3(r) F_{\epsilon}(e^r\,v(r)) = e^{-2r}(v''(r) + (n-2) v'(r)).
\end{align}
We will now show that $v$ behaves asymptotically like a solution to the linear autonomous part from equation \eqref{Eq after log transform}. For proving this it is essential to write \eqref{Eq after log transform} in the following way
\begin{align}\label{second order ode v}
    v''(r) + b(r) \, v'(r) + c(r) \, v(r) = 0
\end{align}
with 
\begin{align*}
    b(r)= \frac{3-(n-2)e^{-2r}}{1-e^{-2r}} \quad \text{and} \quad c(r)= \frac{2-v^2(r)\,F_{\epsilon}(e^r\,v(r))}{1-e^{-2r}}.
\end{align*}
Now we want to investigate the asymptotic behavior of $b$ and $c$ at infinity. It is immediately evident that one has 
\begin{align*}
    \abs{b(r) - 3} \lesssim e^{-2r}.
\end{align*}
The asymptotic behavior of $c$ now follows from the uniform boundedness of $F_{\epsilon}$ and the fact that we can already extract some decay out of $v$ due to the generalized Strauss inequality \eqref{Eq: generalized Strauss inequality}. Indeed, by this inequality we know that $\psi_1$ has a decaying rate of at least $\gamma := \frac{n}{2}-s$. Since $s$ satisfies \eqref{condition on exponents} we note that $\gamma$ is strictly greater than $1/2$. Therefore, we obtain for large $r$ due to the boundedness of $F_{\epsilon}$  
\begin{align}\label{convergence of coefficients}
     \abs{c(r)-2} \lesssim e^{-2\gamma r}.
\end{align}
We now define
\begin{align*}
    U(r) = \begin{pmatrix}
  2e^{-r} - e^{-2r} & e^{-r} - e^{-2r}\\ 
  2e^{-2r} - 2 e^{-r} & 2 e^{-2r} - e^{-r}
\end{pmatrix} \quad \text{and} \quad V(r) = \begin{pmatrix}
  v(r)\\ 
  v'(r)
\end{pmatrix}
\end{align*}
and show that $W(r) := U(-r)V(r)$ converges pointwise for $r \to \infty$. Using \eqref{second order ode v} it follows that $V$ solves
\begin{align*}
	 V'(r) = \begin{pmatrix}
  0 & 1\\ 
  -c(r) & -b(r)
\end{pmatrix} V(r).
\end{align*}
Together with the fact that $U$ solves 
\begin{align*}
    U'(r) = \begin{pmatrix}
  0 & 1\\ 
  -2 & -3
\end{pmatrix} U(r), \quad U(0) = I
\end{align*}
    we obtain
\begin{align*}
\frac{d}{dr}\left[ U(-r) V(r) \right] = U(-r)  \begin{pmatrix}
  0 & 0\\ 
  2-c(r) & 3-b(r)
\end{pmatrix} V(r).
\end{align*}
With that we obtain $W'(r) = A(r) W(r)$ for
    \begin{align*}
        A(r) = U(-r) \begin{pmatrix}
  0 & 0\\ 
  2-c(r) & 3-b(r)
\end{pmatrix} U(r).
    \end{align*}
From the differential equation that $W(r)$ solves we obtain for all $r \geq r_0$ 
\begin{align}\label{integral eqn for W}
    W(r) = W(r_0) + \int_{r_0}^r A(t) W(t) \,dt
\end{align}
and can therefore deduce with Grönwall's inequality 
\begin{align}\label{Grönwall for W}
    \norm{W(r)} \leq \norm{W(r_0)} \exp\left(\int_{r_0}^r\norm{A(t)}\,dt\right)
\end{align}
where we use $\norm{\cdot}$ for an arbitrary matrix norm. We now obtain
\begin{align*}
    \norm{A(r)} \leq \norm{U(-r)} \max\{\abs{b(r)-3},\abs{c(r)-2}\} \norm{U(r)} \lesssim e^{-(2\gamma-1) r}
\end{align*}
so that $A$ is integrable at infinity. In particular, we can then deduce the boundedness of $W$ from \eqref{Grönwall for W}. Therefore, we obtain from \eqref{integral eqn for W} the convergence of $W(r)$ as $r$ goes to infinity.\\ If we now denote this limit by $(w_1,w_2)$ we get from the definition of $U$ and $V$ the following two asymptotic behaviors 
\begin{align*}
    \left[2v(r)+v'(r)\right]e^{r} - \left[v(r)+v'(r)\right] e^{2r} &\to w_1,\\ \left[2v(r)+2v'(r)\right]e^{2r} - \left[2v(r)+v'(r)\right] e^{r} &\to w_2.
\end{align*}
Comparing these two lines gives us 
\begin{align*}
    \left[2v(r) + v'(r)\right]e^r&\to 2w_1+w_2,\\
    \left[v(r) + v'(r)\right]e^{2r}&\to w_1+w_2.
\end{align*}
In particular, $\left[v(r) + v'(r)\right]e^r\to 0$ and therefore $v(r)e^r\to 2w_1+w_2$ as well as  $v'(r)e^r\to -2w_1-w_2$. Therefore, we have shown the following estimates for large $r$:
\begin{align*}
    \abs{v(r)} \leq C\, e^{-r}, \quad \abs{v'(r)} \leq C\, e^{-r}, \quad \abs{v(r)+v'(r)} \leq C\, e^{-2r} 
\end{align*}
and that the limits $\lim\limits_{r \to \infty} v(r)e^r$ and $\lim\limits_{r \to \infty} (v(r) + v'(r))e^{2r}$ exist.
Transforming everything back into the original variable gives us
\begin{align*}
    \abs{\widetilde{\psi}_1(\rho)} \leq C\, \rho^{-1}, \abs{\widetilde{\psi}_1'(\rho)} \leq C\, \rho^{-2}, \quad \abs{\widetilde{\psi}_2(\rho)} \leq C\, \rho^{-2}
\end{align*}
for all large enough $\rho>0$ and the limits 
\[ \lim\limits_{\abs{\xi}\to\infty}\abs{\xi}\psi_1(\xi), \quad 
 \lim\limits_{\abs{\xi} \to \infty}\abs{\xi}^2( \psi_1(\xi) + \Lambda  \psi_1(\xi))\]
 exist.\\
The decay of the higher derivatives of $\psi_1$ can be shown inductively by using that $\widetilde{\psi}_1$ solves equation \eqref{cubic wave equation with error term} and by the fact that all of the higher derivatives are initially bounded via the generalized Strauss inequality.

The decay of the higher derivatives of the second component $\psi_2$ can be shown as follows:\\
First, we substitute $\rho = \frac{r}{T-t}$ in \eqref{ODE for first component} and get the following equation for $\widetilde{\psi}_1$
\begin{align}\label{equation for self similar solution}
    0 = \left(1-\frac{r^2}{(T-t)^2}\right) u''\left(\frac{r}{T-t}\right)& + \left(\frac{(n-1)(T-t)}{r}-\frac{4r}{T-t}\right)u'\left(\frac{r}{T-t}\right)\\& - 2 u\left(\frac{r}{T-t}\right) + \frac{(n-3)(T-t)^3}{r^3}\eta_{\epsilon}\left(\frac{r}{T-t} \,u\left(\frac{r}{T-t}\right)\right).\nonumber
\end{align}
If we now differentiate this equation with respect to $t$ we get
\begin{align*}
0 =& \frac{r}{(T-t)^2}\left(1-\frac{r^2}{(T-t)^2}\right)\widetilde{\psi}_1'''\left(\frac{r}{T-t}\right)-\frac{2r^2}{(T-t)^3}\widetilde{\psi}_1''\left(\frac{r}{T-t}\right)\\-&\left(\frac{d+1}{r}+\frac{4r}{(T-t)^2} \right)\widetilde{\psi}_1'\left(\frac{r}{T-t}\right)+\frac{r}{(T-t)^2}\left(\frac{(d+1)(T-t)}{r}-\frac{4r}{T-t}\right)\widetilde{\psi}_1''\left(\frac{r}{T-t}\right)\\-&\frac{2r}{(T-t)^2}\widetilde{\psi}_1'\left(\frac{r}{T-t}\right)-\frac{3(d-1)(T-t)^2}{r^3}\eta_{\epsilon}\left(\frac{r}{T-t}\widetilde{\psi}_1\left(\frac{r}{T-t}\right)\right)\\+&\frac{(d-1)(T-t)^3}{r^3}\left(\frac{r}{(T-t)^2}\widetilde{\psi}_1\left(\frac{r}{T-t}\right)+\frac{r^2}{(T-t)^3}\widetilde{\psi}_1'\left(\frac{r}{T-t}\right)\right)\eta_{\epsilon}'\left(\frac{r}{T-t}\widetilde{\psi}_1\left(\frac{r}{T-t}\right)\right).
\end{align*}
Then we multiply with $T-t$ and write the equation again in the $\rho$ variable
\begin{align*}
0 =& \rho\left(1-\rho^2\right)\widetilde{\psi}_1'''(\rho)-2\rho^2\widetilde{\psi}_1''(\rho)-\left(\frac{d+1}{\rho}+4\rho \right)\widetilde{\psi}_1'(\rho)\\+&\rho\left(\frac{(d+1)}{\rho}-4\rho\right)\widetilde{\psi}_1''(\rho)-2\rho\widetilde{\psi}_1'(\rho)\\+&\frac{(d-1)}{\rho^3}\left(\rho\widetilde{\psi}_1(\rho)+\rho^2\widetilde{\psi}_1'(\rho)\right)\eta_{\epsilon}'(\rho\widetilde{\psi}_1(\rho)-\frac{3(d-1)}{\rho^3}\eta_{\epsilon}(\rho\widetilde{\psi}_1(\rho)).
\end{align*}
Using that $\widetilde{\psi}_1$ solves
\begin{align*}
	(1-\rho^2)u''(\rho) + \left(\frac{d+1}{\rho}-4\rho\right)u'(\rho) -2u(\rho) + \frac{d-1}{\rho^3}\eta_{\epsilon}(\rho\,u(\rho))=0
\end{align*}
we notice that $\widetilde{\psi}_2 = \widetilde{\psi}_1 + \rho \, \widetilde{\psi}_1'$ solves the following linear second order ODE 
\begin{align}\label{radial eigenvalue equation}
    (1-\rho^2)w''(\rho) + \left(\frac{n-1}{\rho}-6\rho\right)w'(\rho) -6w(\rho) + \frac{n-3}{\rho^2}\eta_{\epsilon}'(\rho\,\widetilde{\psi}_1)w(\rho)=0.
\end{align}

As we will see later on, this is precisely the eigenvalue equation of the linearization of \eqref{evolution equation} around $\mb \Psi_{\epsilon}$ to the eigenvalue $\lambda = 1$ (i.e. one also obtains this equation by differentiating \eqref{equation for self similar solution} with respect to $T$). But for now it is enough to notice that this is a linear second order ODE with a regular singularity at infinity. Therefore, we can apply the method of Frobenius, \cite{Teschl2012}, p.119, Theorem 4.5, to obtain a fundamental system of this equation for large enough $\rho > 0$
\begin{align}\label{Eq:fundamental_sys}
    w_1(\rho) = \rho^{-3} \, h_1(\rho^{-1}), \quad w_2(\rho) = \rho^{-2} h_2(\rho^{-1}) + c \log(\rho) \rho^{-3} h_1(\rho^{-1})
\end{align}
where $h_1$ and $h_2$  are analytic around zero with $\lim_{\rho \to \infty}h_i(\rho^{-1}) = 1$. Since $\widetilde{\psi}_2$ can be written as a linear combination of $w_1$ and $w_2$ for large $\rho$ one immediately obtains the desired decay.

\end{proof}

By now reversing the coordinate transformations we are ready to prove the first main result of this paper.

\begin{proof}[Proof of Theorem \ref{Theorem: Blowup Solution}]
Let $d \geq 3$. Set $n := d+2$ and let $(s,k) \in \R \times \N$ satisfy \eqref{condition on exponents}. Take $\delta^*$ and $\epsilon^*(\delta^*)$ from Proposition \ref{Prop: perturbation}. For $\epsilon \in \R$ with $|\epsilon | \leq \epsilon^*$ let  $\mb \Psi_{\epsilon} = (\psi_{\epsilon,1},\psi_{\epsilon,2}) \in C^\infty_{r}(\R^n) \times C^\infty_{r}(\R^n)$ be the function obtained in Corollary \ref{Corr: existence of blowup solution}. Then $\psi_{\epsilon,1} $ solves 
\begin{align}
     (-\Delta +\Lambda^2+3\Lambda+2)\psi_{\epsilon,1}= N_{\epsilon}(\psi_{\epsilon,1})
 \end{align}
and the radial representative $\widetilde{\psi}_{\epsilon,1}$ of $\psi_{\epsilon,1}$ satisfies the nonlinear ODE \eqref{ODE for first component}. Now we define
\[f_{\epsilon}(\rho) :=  \rho \widetilde{\psi}_{\epsilon,1}(\rho) = \rho \widetilde{\psi}_{0,1}(\rho)  + \rho \widetilde{\phi}_{\epsilon,1}(\rho) \]
where $\widetilde{\phi}_{\epsilon,1}$ denotes the radial representative of the first component of $\mb \Phi_{\epsilon}$.
By setting $\phi_{\epsilon}(\rho) := \rho \widetilde{\phi}_{\epsilon,1}$ we have $f_{\epsilon} = f_{0} + \phi_{\epsilon}$
and  $f_{\epsilon}$ solves
\[ (1-\rho^2) f''(\rho) + \left(\frac{d-1}{\rho}-2 \rho\right)f'(\rho) + \frac{d-1}{\rho^2} w_{\epsilon}(f(\rho))w'_{\epsilon}(f(\rho)) = 0, \]
for all $\rho \in (0,\infty)$. Hence,
\[ u_{\epsilon}^T(t,r) := f_{\epsilon}\left(\frac{r}{T-t}\right)\]
provides a solution to Eq.~\eqref{equation for radial profile}. For the derivative we obtain
\begin{align*}
    \abs{\partial_r u_{\epsilon}^T(t,0)} = (T-t)^{-1} \abs{\widetilde{\psi}_{\epsilon,1}(0)}
\end{align*}
for $0 \leq t < T$.
In view of Lemma \ref{Le:Smallness_Condition} and our previous assumptions we infer that 
\[|\widetilde{\phi}_{\epsilon,1}(0) | < \widetilde{\psi}_{0,1}(0),\]
and thus $\widetilde{\psi}_{\epsilon,1}(0) \neq 0$. In particular, the derivative of $u_{\epsilon}^T$ blows up as $t \to T^{-}$. Finally the Lipschitz estimates follow from Proposition \ref{Prop: perturbation} and Sobolev embedding and the decay properties for $f_{\epsilon}$ are an obvious consequence of Proposition \ref{Prop: decay of blowup solution}   
\end{proof}

\section{Stability analysis}\label{Section: Stability Analysis}

The aim of this section is to analyze the  stability properties of the blowup solution $u_{\epsilon}^T\in  C^{\infty}([0,T)\times [0,\infty))$ constructed in Theorem \ref{Theorem: Blowup Solution}. Again, the starting point will be Eq.~\eqref{n-dimensional semilinear wave eq} for which we now consider the Cauchy problem

\begin{align} \label{Eq:Central_CP}
        \partial^2_t v - \Delta_x v = \frac{n-3}{\abs{x}^3}\left(\abs{x}v - w_{\epsilon}(\abs{x}v)w_{\epsilon}'(\abs{x}v)\right)    
\end{align}

where the initial data $ v(0,x)= v_{0,\epsilon}(x)$ and $\partial_t v(0,x)= v_{1,\epsilon}(x)$, $x \in \R^n$ are given by 
    \begin{align}\label{Eq:Central_CP_data}
        (v_{0,\epsilon},v_{1,\epsilon}):=(v_{\epsilon}^1(0,\cdot), \partial_tv_{\epsilon}^1(0,\cdot)) + (\varphi_0,\varphi_1) 
    \end{align}
and $v_{\epsilon}^1$ denotes the corresponding blowup solution $v_{\epsilon}^T$ with blowup time $T=1$,  
\[v_{\epsilon}^T(t,x) = \frac{1}{T-t} \psi_{\epsilon}\left(\frac{x}{T-t} \right), \quad \psi_{\epsilon}(\xi) = |\xi|^{-1} f_{\epsilon}(|\xi|). \]

By rewriting the problem as a first order system in similarity coordinates, Eq.~\eqref{Eq:Central_CP} transforms into Eq.~\eqref{evolution equation}, see Section \ref{Section: Similarity Variables}. By construction, the blowup solution $v_{\epsilon}^T$  corresponds to the  static solution $\mb \Psi_{\epsilon}$ of Eq.~\eqref{evolution equation},

\begin{align}\label{Eq:Def_Psieps}
    \mb \Psi_{\epsilon}=\begin{pmatrix}
        \psi_{\epsilon,1}\\
        \psi_{\epsilon,2}
    \end{pmatrix}, \text{ with } \psi_{\epsilon,1}= \psi_{\epsilon} \text{ and } \psi_{\epsilon,2} = \psi_{\epsilon,1} + \Lambda \psi_{\epsilon,1}.
\end{align}

Consequently, the initial value problem \eqref{Eq:Central_CP}-\eqref{Eq:Central_CP_data} can be rephrased as
\begin{align}
    \begin{cases}\label{Cauchy problem for stability analysis operator formulation}
        \partial_{\tau}\Psi(\tau)=\mb{\widetilde{L}}\Psi(\tau)+ \mb{N_{\epsilon}}(\Psi(\tau)), \\
        \Psi(0)= \mb \Psi_{\epsilon}^{\mb{T}} + \varphi^{\mb{T}},
     \end{cases}
\end{align}
for $\mb \Psi_{\epsilon}^{\mb{T}}: =(T\psi_{\epsilon,1}(T\cdot),T^2\psi_{\epsilon,2}(T\cdot))$ and $\varphi^{\mb{T}} :=(T\varphi_0(T\cdot),T^2\varphi_1(T\cdot))$ the perturbation of the original initial data. We remark that $\widetilde{\mb{L}}$ is still the wave operator in similarity coordinates as we introduced it in \eqref{wave operator in similarity coordinates} and the nonlinearity $\mb{N}_{\epsilon}$ is defined as in Eq.~ \eqref{Def: Nonlinearity}.\\ 
Assuming that the solution to Eq.~\eqref{Cauchy problem for stability analysis operator formulation} will be a small perturbation of $\mb \Psi_{\epsilon}$, we insert the ansatz  $\Psi(\tau) =\mb \Psi_{\epsilon}+\Phi_{\epsilon}(\tau)$ and get an evolution equation for the (time-dependent) perturbation $\Phi_{\epsilon}$,
\begin{align*}
    \partial_{\tau}\Phi_{\epsilon}(\tau)=\mb{\widetilde{L}}\Phi_{\epsilon}(\tau)+\mb{N}_{\epsilon}(\mb \Psi_{\epsilon}+\Phi_{\epsilon}(\tau))- \mb{N}_{\epsilon}(\mb \Psi_{\epsilon}).
\end{align*}
A Taylor expansion of $\mb{N}_{\epsilon}$ around $\mb \Psi_{\epsilon}$ now leads us to the central evolution equation of this section
\begin{align}
    \begin{cases}\label{Cauchy problem for perturbation stability analysis}
        \partial_{\tau}\Phi_{\epsilon}(\tau)=\mb{\widetilde{L}}_{\epsilon}\Phi_{\epsilon}(\tau)+ \mb{\widehat{N}}_{\epsilon}(\Phi_{\epsilon}(\tau)),\\
        \Phi_{\epsilon}(0) = \varphi^{\mb{T}} + \mb \Psi_{\epsilon}^{\mb{T}} - \mb \Psi_{\epsilon}. 
    \end{cases}
\end{align}
 The linear part $\widetilde{\mb{L}}_{\epsilon} := \widetilde{\mb{L}} + \mb{L}_{\epsilon}'$ consists of the wave operator  $\widetilde{\mb{L}}$ and of the potential $\mb{L}_{\epsilon}' $ which is given by 
\begin{align}
    \mb{L}_{\epsilon}'\mb{u}= \begin{pmatrix}
        0\\
        V_{\epsilon}(\psi_{\epsilon,1})\,u_1
    \end{pmatrix} \quad \text{ for } \quad V_{\epsilon}(\psi_{\epsilon,1})(\xi) = \frac{n-3}{\abs{\xi}^2}\eta_{\epsilon}'(\abs{\xi}\psi_{\epsilon,1}(\xi)).
\end{align}
The remaining nonlinearity $\widehat{\mb{N}}_{\epsilon}$ is of the form 
\begin{align}\label{Eq:Def_NonlinEvol}
    \mb{\widehat{N}}_{\epsilon}(\mb{u})=\begin{pmatrix}
        0\\
        \widehat{N}_{\epsilon}(u_1) 
    \end{pmatrix},  
\end{align} with
\begin{align}
    \widehat{N}_{\epsilon}(u_1)(\xi)= \frac{n-3}{\abs{\xi}^3}\left(\eta_{\epsilon}(\abs{\xi}(\psi_{\epsilon,1}(\xi)+u_1(\xi)))-\eta_{\epsilon}(\abs{\xi}\psi_{\epsilon,1}(\xi))-\eta_{\epsilon}'(\abs{\xi}\psi_{\epsilon,1}(\xi))\abs{\xi}u_1(\xi)\right).
\end{align}
Notice that the only trace of the parameter $T$ is in the initial data and that the here defined linearized operator $\widetilde{\mb{L}}_{\epsilon}$ coincides for $\epsilon=0$ with the one we introduced in Section \ref{Section: Known Results}.\\
To now find a solution to Eq.~\eqref{Cauchy problem for perturbation stability analysis} we take the semigroup approach. 

\subsection{Perturbations of the free evolution for small parameters}

The free evolution generated by $\widetilde{\mb{L}}$ is well understood, see the results summarized in Section \ref{Section: Known Results}. For the linearized flow, the following properties of the perturbation $\mb{L}_{\epsilon}'$ are crucial. 
For the rest of the paper let $\epsilon^*$ denote the constant determined by Proposition \ref{Prop: contraction operator}.

\begin{proposition}\label{Prop: compactness and Lipschitz of potential}
Let $n \geq 5$ and $(s,k) \in \R \times \N$ satisfy \eqref{condition on exponents}. For every $\epsilon \in \R$ with $\abs{\epsilon}\leq\epsilon^*$ the operator $\mb{L}_{\epsilon}':\mc{H}_r^{s,k}\to\mc{H}_r^{s,k}$ is compact. Furthermore, the family of operators $\mb{L}_{\epsilon}'$ is Lipschitz continuous in $\epsilon,$ i.e., the inequality 
    \begin{align*}
        \norm{\mb{L}_{\epsilon}'\mb{u}-\mb{L}_{\kappa}'\mb{u}}_{s,k} \lesssim \abs{\epsilon-\kappa} \norm{\mb{u}}_{s,k}
    \end{align*}
    holds for all $\mb{u} \in \mc{H}_r^{s,k}$ and all $\abs{\epsilon},\abs{\kappa}\leq \epsilon^*$.
\end{proposition}
\begin{proof}
Due to the decay of the blowup solution, see Proposition \ref{Prop: decay of blowup solution},  $V_{\epsilon}(\psi_{\epsilon,1})$ satisfies the assumptions of Lemma 5.1 in \cite{Glo22} from which the compactness of $\mb{L}_{\epsilon}'$ follows. For the Lipschitz estimates, we write $V_{\epsilon}(\psi_{\epsilon,1})$ with \eqref{second auxiliary equation} in the following way
    \begin{align}\label{FTC for potential}
        V_{\epsilon}(\psi_{\epsilon,1})(\xi) = (n-3)\psi_{\epsilon,1}(\xi)^2\int_0^1x\int_0^1\eta_{\epsilon}'''(\abs{\xi}\psi_{\epsilon,1}(\xi)xy)dy\,dx
    \end{align}
and conclude by applying the Schauder estimate from Proposition \ref{Prop: Schauder} using the Lipschitz continuity of $\psi_{\epsilon,1}$.

\end{proof}

By the bounded perturbation theorem we therefore get the following statement, where we first equip $\mb{\widetilde{L}}_{\epsilon}$ with the domain $\mathcal{D}(\mb{\widetilde{L}}_{\epsilon}) = \mc{D}(\mb{\widetilde{L}})$, see  Proposition \ref{Prop: Spectral properties of L_0}.

\begin{proposition}\label{proposition: semigroup and Lipschitz continuity of full linear part}
Let $n \geq 5$ and $(s,k) \in \R \times \N$ satisfy \eqref{condition on exponents}. For every $\epsilon \in \R$ with $\abs{\epsilon}\leq \epsilon^*$ the operator $\mb{\widetilde{L}}_{\epsilon}$ is closable and its closure $\mb{L}_{\epsilon}=\mb{L}+\mb{L}_{\epsilon}'$ with $\mathcal{D}(\mb{L}_{\epsilon}) = \mc{D}(\mb{L})= \mc{D}(\mb{L}_0)$ generates a strongly continuous semigroup $(\mb{S}_{\epsilon}(\tau))_{\tau\geq0}$ of bounded operators on $\mc{H}_r^{s,k}$. Furthermore, 
    \begin{align*}
        \norm{\mb{L}_{\epsilon} \mb{u}-\mb{L}_{\kappa} \mb{u}}_{s,k} \lesssim \abs{\epsilon-\kappa} \norm{\mb{u}}_{s,k}
    \end{align*}
    holds for all $\mb{u} \in \mc{H}_r^{s,k}$ and all $\abs{\epsilon},\abs{\kappa}\leq \epsilon^*$.
\end{proposition}

As in the case $\epsilon = 0$ we do not expect decay of the semigroup due to an exponential instability induced by the time translation symmetry of the original problem. More precisely, we have the following result. 
\begin{lemma}\label{Le:g}
    Let $n\geq5$, $(s,k) \in \R \times \N$ satisfy \eqref{condition on exponents} and $\abs{\epsilon} \leq \epsilon^*$. We then have 
    \begin{align}\label{eigenvalue equation}
        \mb{L}_{\epsilon}\, \mb{g}_{\epsilon} = \mb{g}_{\epsilon},
    \end{align} 
    where $\mb{g}_{\epsilon} \in C^{\infty}_r(\R^n) \times C^{\infty}_r(\R^n) \cap \mathcal{D}(\mb{L}_{\epsilon})$ is defined as
    \begin{align*}
        \mb{g}_{\epsilon} := \begin{pmatrix}
            g_{\epsilon,1}\\
            g_{\epsilon,2}
        \end{pmatrix} := \begin{pmatrix}
            \psi_{\epsilon,1} + \Lambda \psi_{\epsilon,1}\\
        2 \psi_{\epsilon,1} + 3 \Lambda \psi_{\epsilon,1} + \Lambda^2 \psi_{\epsilon,1}
        \end{pmatrix} .
    \end{align*}
\end{lemma}

\begin{proof}
We first show that $\mb{g}_{\epsilon}$ lies in the domain of the operator $\mb{L}_{\epsilon}$. Since the domain of $\mb{L}_{\epsilon}$ coincides with the one of $\mb{L}_0$ we can apply Lemma 4.5 from \cite{Glo23}.
Namely, we have to show that the two components of $\mb{g}_{\epsilon}$ satisfy for every multi-index $\beta \in \N^n_0$ an estimate of the form
\begin{align}\label{decaying conditions for eigenfunction}
    \abs{\partial^{\beta}g_{\epsilon,1}(\xi)} \lesssim \langle \xi \rangle^{s-\frac{n}{2}-\abs{\beta}-1} \quad \text{and} \quad \abs{\partial^{\beta}g_{\epsilon,2}(\xi)} \lesssim \langle \xi \rangle^{s-\frac{n}{2}-\abs{\beta}-2}.
\end{align}
Recalling Proposition \ref{Prop: decay of blowup solution} and its proof we infer that the radial representative  of $g_{\epsilon,1}$ satisfies  $\widetilde{g}_{\epsilon,1} =\widetilde{\psi}_{\epsilon,2}$ and this implies 
\begin{align*}
    \abs{\partial^{\beta}g_{\epsilon,1}(\xi)} \lesssim_{\beta} \langle \xi\rangle^{-2-\abs{\beta}}
\end{align*}
for every $\beta \in \N_0^n$ and $\xi \in \R^n$, which yields the first bound in Eq.~\eqref{decaying conditions for eigenfunction} since $s > \frac{n}{2} - 1$. For the second component we note that according to the proof of Proposition \ref{Prop: decay of blowup solution}
 \begin{align*}
        \widetilde{g}_{\epsilon,1}(\rho) = a_1 \, w_1(\rho) + a_2 \, w_2(\rho)
    \end{align*}
    for large $\rho >0$ where $a_1, a_2 \in \C$ and $w_1$ and $w_2$ are given by
\begin{align*}
    w_1(\rho) = \rho^{-3} \, h_1(\rho^{-1}), \quad w_2(\rho) = \rho^{-2} h_2(\rho^{-1}) + c \log(\rho) \rho^{-3} h_1(\rho^{-1})
\end{align*}
with functions $h_1$ and $h_2$ which are analytic around zero. Now, we use the fact that 
 \[g_{\epsilon,2} =   2 g_{\epsilon,1} + \Lambda g_{\epsilon,1} \]
 and observe that the first term of $w_2$ (the one with the least decay) cancels so that we obtain 
\begin{align*}
    \abs{\widetilde{g}^{(\ell)}_{\epsilon,2}(\rho)} \lesssim \log(\rho) \rho^{-3-\ell}
\end{align*}
for every $\ell \in \N_0$ and every large enough $\rho > 0$. Since the logarithm grows slower than any polynomial and from the fact that $s$ is strictly larger than $\frac{n}{2}-1$ we infer that the second bound in \eqref{decaying conditions for eigenfunction} holds and we  conclude  $\mb{g}_{\epsilon} \in \mb{L}_{\epsilon}$.

Now we show that $\mb{g}_{\epsilon}$ is indeed an eigenfunction of $\mb{L}_{\epsilon}$ to the eigenvalue $\lambda = 1$. We first note that by similar reasoning as in Proposition \ref{Prop: decay of blowup solution} we actually have $\mb{L}_{\epsilon} \mb{g}_{\epsilon} = \widetilde{\mb{L}}_{\epsilon} \mb{g}_{\epsilon}$. Eq.~ \eqref{eigenvalue equation}  yields 
\begin{align*}
    \Delta g_{\epsilon,1} - \Lambda^2 g_{\epsilon,1} - 5 \Lambda g_{\epsilon,1} - 6 g_{\epsilon,1} + V_{\epsilon}(\psi_{\epsilon,1}) g_{\epsilon,1} = 0
\end{align*}
so that its radial representative $\widetilde{g}_{\epsilon,1}$ has to solve \eqref{radial eigenvalue equation}. But this is true by definition, since $\widetilde{g}_{\epsilon,1} = \widetilde{\psi}_{\epsilon,2} = \widetilde{\psi}_{\epsilon,1} + \rho\, \widetilde{\psi}_{\epsilon,1}'$. The claim then follows by the definition of $\widetilde{g}_{\epsilon,2}$. 

\end{proof}

For $\epsilon \neq 0$, the potential is not available in closed form which is however crucial for a direct investigation of the spectral problem and the exclusion of other unstable spectral points.   Hence, we resort to perturbative arguments in order to characterize the spectral properties of $ \mb{L}_{\epsilon}$.

\begin{proposition}\label{proposition: stability of spectrum of Le}
Let $n\geq5$ and  $(s,k) \in \R \times \N$ satisfy \eqref{condition on exponents}. For every arbitrary but fixed $ 0 < \omega_0 < \widetilde{\omega}$ where $\widetilde{\omega}$ is the constant from Proposition \ref{Prop: Spectral properties of L_0},  there exists  $\epsilon^{**}>0$ such that for every $\epsilon \in \R$ with $\abs{\epsilon} \leq \epsilon^{**}$
 \begin{align}\label{spectrum of Le}
        \sigma(\mb{L}_{\epsilon}) \subset \{ \lambda \in \C : \Re \lambda <  -\omega_0 \} \cup \{1\}
    \end{align}
    holds, $\lambda = 1$ is a simple eigenvalue of $\mb{L}_{\epsilon}$ and its eigenspace is spanned by $\mb{g}_{\epsilon}$.
    Furthermore, for the Riesz projection
    \begin{align}\label{Riesz projection}
        \mb{P}_{\epsilon} := \frac{1}{2\pi i} \int_{\partial B_{1\slash2}(1)} \mb{R}_{\mb{L}_{\epsilon}}(\lambda) \, d\lambda
    \end{align}
    we have $\ran(\mb{P}_{\epsilon}) = \ker(\mb{I} - \mb{P}_{\epsilon}) = \langle \mb{g}_{\epsilon} \rangle$.
\end{proposition}

\begin{proof}
We first show  
    \begin{align*}
        \sigma(\mb{L}_{\epsilon}) \cap \{ \lambda \in \C : \Re \lambda \geq - \omega_0\} \subset B_{1\slash2}(1),
    \end{align*}
where $B_{1\slash2}(1)$ denotes the open ball in $\C$ of radius $r = \frac{1}{2}$ centered at $1$, and that there exists a constant $C_{\omega_0}>0$ such that the resolvent satisfies 
    \begin{align}\label{Eq:Est_Resol_eps}
        \norm{\mb{R}_{\mb{L}_{\epsilon}}(\lambda)} \leq C_{\omega_0}
    \end{align}
for all $\lambda \in \closure{\mathbb{H}_{-\omega_0}}\, \backslash B_{1\slash2}(1)$ and all $\abs{\epsilon} \leq \epsilon^{**}$.  We argue similarly to \cite{Ost23}, Proposition 3.2.

Let  $\abs{\epsilon}\leq\epsilon^*$. To see that $\closure{\mathbb{H}_{-\omega_0}}\, \backslash B_{1\slash2}(1)$ belongs to the resolvent set of $\mb{L}_{\epsilon}$, we notice that due to Proposition \ref{Prop: Spectral properties of L_0} every $ \lambda \in \closure{\mathbb{H}_{-\omega_0}}\, \backslash \{1\}$ belongs to the resolvent set of $\mb{L}_0$. Therefore, the identity
\begin{align}\label{Identity for Neumann argument}
    \lambda - \mb{L}_{\epsilon} = \left(\mb{I}-(\mb{L}_{\epsilon}-\mb{L}_0)\mb{R}_{\mb{L}_0}(\lambda)\right)(\lambda-\mb{L}_0) 
\end{align}
holds and it follows that $\lambda - \mb{L}_{\epsilon}$ is bounded invertible if and only if this is true for $\mb{I}-(\mb{L}_{\epsilon}-\mb{L}_0)\mb{R}_{\mb{L}_0}(\lambda)$. We will prove the invertibility of the latter one by showing that the Neumann series of $(\mb{L}_{\epsilon}-\mb{L}_0)\mb{R}_{\mb{L}_0}(\lambda)$ converges. For this we split the resolvent in the following way
\begin{align}\label{splitting of resolvent}
    \mb{R}_{\mb{L}_0}(\lambda) \mb{f} = \mb{R}_{\mb{L}_0}(\lambda) (\mb{I} - \mb{P}_0) \mb{f} + \mb{R}_{\mb{L}_0}(\lambda) \mb{P}_0 \mb{f}
\end{align}
where $\mb{P}_0$ is the Riesz projection from Proposition \ref{Prop: Spectral properties of L_0}. From the decay of the semigroup $\left(\mb{S}_0(\tau)\right)_{\tau\in\N}$ on the stable subspace $\ran(\mb{I} - \mb{P}_0)$, see \eqref{decay on stable subspace},  we obtain for every $0 < \omega < \widetilde{\omega}$  the following bound 
\begin{align*}
    \norm{\mb{S}_0(\tau)\left(\mb{I} - \mb{P}_0\right)\mb{u}}_{s,k} \lesssim e^{- \omega \,\tau}\norm{\left(\mb{I} - \mb{P}_0\right)\mb{u}}_{s,k}.
\end{align*}
Since the restriction of $\left(\mb{S}_0(\tau)\right)_{\tau \geq 0}$ to $\ran(\mb{I} - \mb{P}_0)$ equals the  semigroup of the restriction of the operator $\mb{L}_0$ to $\ran(\mb{I} - \mb{P}_0)$ we can apply \cite{EngNag00}, p.55, Theorem 1.10 for every $0 < \omega < \widetilde{\omega}$ to infer the existence of a constant $M_{\omega} \geq 1$ such that
\begin{align}\label{resolvent estimate on stable subspace}
    \norm{\mb{R}_{\mb{L}_0}(\lambda)(\mb{I}-\mb{P}_0)} \leq \frac{M_{\omega}}{\Re \lambda+\omega}
\end{align}
holds for every $\lambda \in \mathbb{H}_{- \omega}$. Hence, (by choosing for example $\omega = \frac{\omega_0 + \widetilde{\omega}}{2}$ in Eq. ~\eqref{resolvent estimate on stable subspace}) there exists a constant $C_{\omega_0} > 0$ such that 
\[  \norm{\mb{R}_{\mb{L}_0}(\lambda)(\mb{I}-\mb{P}_0)}  \leq C_{\omega_0} \]
holds for all $\lambda \in \closure{\mathbb{H}_{-\omega_0}}$. For the second term in Eq.~\eqref{splitting of resolvent} we notice that  $\ker(\mb{I}-\mb{L}_0) = \ran(\mb{P}_0) = \langle \mb{g}_0 \rangle$ so that there exists a unique $\mb{f}_0 \in \ran(\mb{P}_0)$ with
\begin{align}
    \mb{P}_0 \mb{f} = \langle \mb{f}_0, \mb{f} \rangle_{s,k} \, \mb{g}_0
\end{align}
for every $\mb{f} \in \mc{H}_r^{s,k}$. Since $\mb{g}_0$ is an eigenvector of $\mb{L}_0$ to the eigenvalue $\lambda=1$ we obtain
\begin{align}
    \mb{R}_{\mb{L}_0}(\lambda) \mb{P}_0 \mb{f} = (\lambda-1)^{-1} \langle \mb{f}_0, \mb{f} \rangle_{s,k} \, \mb{g}_0.
\end{align}
From these equations we now infer from \eqref{splitting of resolvent} the existence of a constant $C_{\omega_0}>0$ with
\begin{align*}
    \norm{\mb{R}_{\mb{L}_0}(\lambda)} \leq  C_{\omega_0}
\end{align*}
for every $\lambda \in \closure{\mathbb{H}_{-\omega_0}}\, \backslash B_{1\slash2}(1)$. In sight of Proposition \ref{proposition: semigroup and Lipschitz continuity of full linear part} we therefore get
\begin{align*}
    \norm{(\mb{L}_{\epsilon}-\mb{L}_0)\mb{R}_{\mb{L}_0}(\lambda)} \lesssim |\epsilon|
\end{align*}
for every $\lambda \in \closure{\mathbb{H}_{-\omega_0}}\, \backslash B_{1\slash2}(1)$ and every $\abs{\epsilon}\leq\epsilon^*$. We now choose $\epsilon^{**} \leq \epsilon^*$ small enough so that $\norm{(\mb{L}_{\epsilon}-\mb{L}_0)\mb{R}_{\mb{L}_0}(\lambda)} < 1$. Hence,  we  conclude from \eqref{Identity for Neumann argument} that $\lambda$ belongs to the resolvent set of $\mb{L}_{\epsilon}$. Furthermore, we obtain from the same equation that the resolvent has the following representation
\begin{align*}
    \mb{R}_{\mb{L}_{\epsilon}}(\lambda) =  \mb{R}_{\mb{L}_0}(\lambda) \left(\mb{I}-(\mb{L}_{\epsilon}-\mb{L}_0)\mb{R}_{\mb{L}_0}(\lambda)\right)^{-1}
\end{align*}
from which Eq.~\eqref{Eq:Est_Resol_eps} follows.

We will now conclude that the spectrum of $\mb{L}_{\epsilon}$ is contained in a left half-plane except for the unstable eigenvalue $\lambda=1$.
From our previous considerations we know that  $\partial B_{1\slash 2}(1)$ belongs to the resolvent set of $\mb{L}_{\epsilon}$. Therefore, the Riesz projection $\mb{P}_{\epsilon}$ from \eqref{Riesz projection} is well-defined. We furthermore obtain from the resolvent identity, the Lipschitz continuity of $\mb{L}_{\epsilon}$ and the uniform bound for the resolvent for every $\lambda \in \partial B_{1\slash 2}(1)$,
\begin{align*}
    \norm{\mb{R}_{\mb{L}_{\epsilon}}(\lambda)-\mb{R}_{\mb{L}_{\kappa}}(\lambda)} \leq \norm{\mb{R}_{\mb{L}_{\epsilon}}(\lambda)} \norm{\mb{L}_{\epsilon}-\mb{L}_{\kappa}} \norm{\mb{R}_{\mb{L}_{\kappa}}(\lambda)} \lesssim \abs{\epsilon-\kappa}.
\end{align*}
Hence $\mb{R}_{\mb{L}_{\epsilon}}$ is Lipschitz continuous with respect to the parameter $\epsilon$ and therefore the same holds true for the Riesz projection $\mb{P}_{\epsilon}$. Therefore, we know from \cite{Kat95}, p.34, Lemma 4.10 that there holds for every $\abs{\epsilon} \leq \epsilon^{**}$,
\begin{align*}
    \dim \ran (\mb{P}_{\epsilon}) = \dim \ran (\mb{P}_0) = 1.
\end{align*}
But since we have the inclusion $\langle \mb{g}_{\epsilon} \rangle \subset \ker(\mb{I} - \mb{L}_{\epsilon}) \subset \ran (\mb{P}_{\epsilon})$ equality must hold due to dimensional reasons and there can also not exist any other spectral points of $\mb{L}_{\epsilon}$ in $B_{1\slash2}(1)$. 
    
\end{proof}

Now we are able to characterize the linearized evolution generated by $\mb{L}_{\epsilon}$. 

\begin{proposition}\label{projection properties}
Let $n \geq 5$ and $(s,k) \in \R \times \N$ satisfy \eqref{condition on exponents}. Then, for every arbitrary but fixed $ 0 < \omega_0 < \widetilde{\omega}$, where $\widetilde{\omega}$ is the constant from Proposition \ref{Prop: Spectral properties of L_0}, and every $\abs{\epsilon}\leq \epsilon^{**}$, where $\epsilon^{**}$ is the constant from Proposition \ref{proposition: stability of spectrum of Le}, the projection $\mb{P}_{\epsilon}$ commutes with the respective semigroup $\mb{S}_{\epsilon}$. In particular,  
    \begin{align}\label{Eq:Evol_unstable}
        \mb{P}_{\epsilon}\, \mb{S}_{\epsilon}(\tau) = \mb{S}_{\epsilon}(\tau) \mb{P}_{\epsilon} = e^{\tau}\,\mb{P}_{\epsilon} 
    \end{align}
    for all $\tau \geq 0.$ We furthermore have
    \begin{align}\label{exponential decay on stable subspace}
         \norm{\mb{S}_{\epsilon}(\tau) (\mb{I}-\mb{P}_{\epsilon}) \mb{u}}_{s,k} \lesssim e^{-\omega_0 \, \tau} \norm{(\mb{I}-\mb{P}_{\epsilon})\mb{u}}_{s,k}
    \end{align}
    as well as
    \begin{align}\label{Lipschitz property of semigroup on stable subspace}
        \norm{\mb{S}_{\epsilon}(\tau) (\mb{I}-\mb{P}_{\epsilon}) - \mb{S}_{\kappa}(\tau) (\mb{I}-\mb{P}_{\kappa})} \lesssim e^{-\omega_0 \, \tau} \abs{\epsilon-\kappa} 
    \end{align}
for all $\mb{u}\in\mc{H}_r^{s,k}, \tau\geq 0$ and $ \abs{\epsilon},\abs{\kappa}\leq \epsilon^{**}$.
\end{proposition}

\begin{proof}
Since the semigroup $\mb{S}_{\epsilon}(\tau)$ commutes with its generator $\mb{L}_{\epsilon}$ it also commutes with its corresponding resolvent $\mb{R}_{\mb{L}_{\epsilon}}$ and therefore also with the projection $\mb{P}_{\epsilon}$. Eq.~\eqref{Eq:Evol_unstable} follows from the fact that $\mb{S}_{\epsilon}(\tau)\mb{P}_{\epsilon}\mb{f}$ as well as $e^{\tau}\mb{P}_{\epsilon}\mb{f}$ are solutions of the uniquely solvable Cauchy problem
    \begin{align*}
        \begin{cases}
            \begin{aligned}
            \partial_{\tau}\mb{u}(\tau)&=\mb{L}_{\epsilon}\mb{u}(\tau), \\
            \mb{u}(0)&= \mb{P}_{\epsilon}\mb{f}.
            \end{aligned}
        \end{cases}
    \end{align*}
Eq.~\eqref{exponential decay on stable subspace} follows from \cite{Ost23} Theorem A.1 provided that one can show the existence of a constant $M_{\omega_0}>0$ with
    \begin{align}\label{uniform bound of resolvent on stable subspace}
        \norm{\mb{R}_{\mb{L}_{\epsilon}}(\lambda)\left(\mb{I}-\mb{P}_{\epsilon}\right)} \leq M_{\omega_0}
    \end{align}
    for all $\lambda \in \closure{\mathbb{H}_{-\omega_0}}$ and $\abs{\epsilon} \leq \epsilon^{**}.$\\
    From Proposition \ref{proposition: stability of spectrum of Le} we already know that there exists a constant $M_{\omega_0}$ such that
    \begin{align*}
        \norm{\mb{R}_{\mb{L}_{\epsilon}}(\lambda)} \leq M_{\omega_0}
    \end{align*}
    holds for every $\lambda \in \closure{\mathbb{H}_{-\omega_0}}\, \backslash B_{1\slash2}(1)$ and $\abs{\epsilon} \leq \epsilon^{**}$. Now we know that $\mb{R}_{\mb{L}_{\epsilon}}(\lambda)\left(\mb{I}-\mb{P}_{\epsilon}\right)$ is an analytic map in $\mathbb{H}_{-\omega_0}$ and since the resolvent of the restriction of $\mb{L}_{\epsilon}$ to the range of $\mb{I}-\mb{P}_{\epsilon}$ coincides with $\mb{R}_{\mb{L}_{\epsilon}}(\lambda)\left(\mb{I}-\mb{P}_{\epsilon}\right)$ we obtain \eqref{uniform bound of resolvent on stable subspace} also for every $\lambda$ lying in the compact ball $\closure{B_{1\slash2}(1)}$.\\ 
    To prove Eq.~\eqref{Lipschitz property of semigroup on stable subspace} we argue similarly to \cite{DonSch16}, Lemma 4.9. By taking $\mb{u} \in \mc{D}(\mb{L}_{\epsilon})$ and then observing
    \begin{align*}
        \partial_{\tau} \mb{S}_{\epsilon}(\tau)(\mb{I}-\mb{P}_{\epsilon}) \mb{u} = \mb{L}_{\epsilon} \mb{S}_{\epsilon}(\tau)(\mb{I}-\mb{P}_{\epsilon}) \mb{u} = \mb{L}_{\epsilon} (\mb{I}-\mb{P}_{\epsilon})\mb{S}_{\epsilon}(\tau)(\mb{I}-\mb{P}_{\epsilon}) \mb{u}
    \end{align*}
yields
    \begin{align*}
        &\partial_{\tau}\left[ \mb{S}_{\epsilon}(\tau)(\mb{I}-\mb{P}_{\epsilon}) \mb{u} - \mb{S}_{\kappa}(\tau)(\mb{I}-\mb{P}_{\kappa}) \mb{u}\right]\\ =& \mb{L}_{\epsilon} (\mb{I}-\mb{P}_{\epsilon})\left[\mb{S}_{\epsilon}(\tau)(\mb{I}-\mb{P}_{\epsilon}) \mb{u}-\mb{S}_{\kappa}(\tau)(\mb{I}-\mb{P}_{\kappa}) \mb{u}\right] + \left[\mb{L}_{\epsilon}(\mb{I}-\mb{P}_{\epsilon})-\mb{L}_{\kappa}(\mb{I}-\mb{P}_{\kappa})\right]\mb{S}_{\kappa}(\tau)(\mb{I}-\mb{P}_{\kappa})\mb{u}.
    \end{align*}
    Consequently, the function
    \begin{align*}
        \Phi_{\epsilon,\kappa}(\tau) := \frac{\mb{S}_{\epsilon}(\tau)(\mb{I}-\mb{P}_{\epsilon}) \mb{u} - \mb{S}_{\kappa}(\tau)(\mb{I}-\mb{P}_{\kappa}) \mb{u}}{\abs{\epsilon-\kappa}}
    \end{align*}
    satisfies the inhomogeneous equation
    \begin{align}\label{inhom ode}
        \partial_{\tau} \Phi_{\epsilon,\kappa}(\tau) = \mb{L}_{\epsilon} (\mb{I}-\mb{P}_{\epsilon}) \Phi_{\epsilon,\kappa}(\tau) + \frac{\mb{L}_{\epsilon}(\mb{I}-\mb{P}_{\epsilon})- \mb{L}_{\kappa}(\mb{I}-\mb{P}_{\kappa})}{\abs{\epsilon-\kappa}} \mb{S}_{\kappa}(\tau)(\mb{I}-\mb{P}_{\kappa}) \mb{u}
    \end{align}
    with initial data $\Phi_{\epsilon,\kappa}(0) = \frac{(\mb{I}-\mb{P}_{\epsilon}) - (\mb{I}-\mb{P}_{\kappa})}{\abs{\epsilon-\kappa}}\mb{u}.$
    By Duhamel's principle the integral equation to \eqref{inhom ode} is given by
    \begin{align}
        \Phi_{\epsilon,\kappa}(\tau) =\,& \mb{S}_{\epsilon}(\tau) (\mb{I}-\mb{P}_{\epsilon})\frac{\mb{P}_{\kappa}-\mb{P}_{\epsilon}}{\abs{\epsilon-\kappa}}\mb{u} \nonumber\\&+ \int_0^{\tau} \mb{S}_{\epsilon}(\tau-\tau') (\mb{I}-\mb{P}_{\epsilon})\frac{\mb{L}_{\epsilon}(\mb{I}-\mb{P}_{\epsilon})- \mb{L}_{\kappa}(\mb{I}-\mb{P}_{\kappa})}{\abs{\epsilon-\kappa}} \mb{S}_{\kappa}(\tau')(\mb{I}-\mb{P}_{\kappa}) \mb{u}\, d\tau'.\label{Duhamel}
    \end{align}
    The Lipschitz continuity of $\mb{L}_{\epsilon}$ and $\mb{P}_{\epsilon}$ now implies
    \begin{align*}
        \norm{\frac{\mb{L}_{\epsilon}(\mb{I}-\mb{P}_{\epsilon})- \mb{L}_{\kappa}(\mb{I}-\mb{P}_{\kappa})}{\abs{\epsilon-\kappa}}} \lesssim 1, \quad \text{for all} \quad \abs{\epsilon},\abs{\kappa} \leq \epsilon^{**}.
    \end{align*}
    Therefore, we get from Duhamel's formula \eqref{Duhamel} and the just proven decay of the semigroup on the stable subspace \eqref{exponential decay on stable subspace}
    \begin{align*}
        \norm{\Phi_{\epsilon,\kappa}(\tau)}_{s,k} \lesssim (1+\tau)e^{-\frac{\omega_0 +\widetilde{\omega}}{2}\,\tau}\norm{\mb{u}}_{s,k} \lesssim e^{-\omega_0\,\tau} \norm{\mb{u}}_{s,k}
    \end{align*}
    for all $\mb{u} \in \mc{D}(\mb{L}_{\epsilon})$. Due to the density of $\mc{D}(\mb{L}_{\epsilon})$ in $\mc{H}_r^{s,k}$ this result now extends to all of $\mc{H}_r^{s,k}$ so that we have proven \eqref{Lipschitz property of semigroup on stable subspace} and therefore the Proposition.
    
\end{proof}

\subsection{The abstract nonlinear Cauchy problem}\label{Section: Nonlinear Cauchy Problem}

In the following, for $(s,k) \in \R \times \N$ satisfying \eqref{condition on exponents} we fix $\omega := \widetilde{\omega}\slash 2$, where $\widetilde{\omega}$ is the constant from Proposition \ref{Prop: Spectral properties of L_0} and denote by $\overline{\epsilon} := \epsilon^{**}(\omega)$ the constant associated via Proposition \ref{proposition: stability of spectrum of Le}.
First, we state  Lipschitz estimates for the nonlinearity $\widehat{\mb{N}}_{\epsilon}$ defined in Eq.~\eqref{Eq:Def_NonlinEvol}.

\begin{lemma}\label{Local Lipschitz continuity for nonlinearity II}
Let $n \geq 5$ and  $(s,k) \in \R \times \N$ satisfy \eqref{condition on exponents}. Then, for any $\epsilon \in \R$ with $|\epsilon| \leq \overline{\epsilon}$, $\mb{\widehat{N}}_{\epsilon}:\mc{H}_r^{s,k} \to \mc{H}_r^{s,k+1}$ and the estimate
       \begin{align}\label{estimate for nonlinearity II}
        \norm{\mb{\widehat{N}}_{\epsilon}(\mb{u})-\mb{\widehat{N}}_{\kappa}(\mb{v})}_{s,k+1} \lesssim \left(\norm{\mb{u}}_{s,k}+\norm{\mb{v}}_{s,k}\right)\norm{\mb{u}-\mb{v}}_{s,k} + \left(\norm{\mb{u}}_{s,k}^2+ \norm{\mb{v}}_{s,k}^2\right) \abs{\epsilon-\kappa}
    \end{align}
holds for all $\abs{\epsilon},\abs{\kappa}\leq \overline{\epsilon}$ and all $\mb{u},\mb{v}\in \mc{B}_{\delta} \subset \mc{H}_r^{s,k}$ for $0<\delta\leq1$.
\end{lemma}

\begin{proof}
    The proof is analogous to the proof of Lemma \ref{Local Lipschitz estimate for nonlinearity} using the Schauder estimate from Proposition \ref{Prop: Schauder}. 
    
\end{proof}

We now focus on the existence and uniqueness of solutions to the Cauchy problem 
\begin{align}\label{Cauchy problem for arbitrary initial data}
    \begin{cases}
        \partial_{\tau} \Phi_{\epsilon}(\tau) = \mb{\widetilde{L}}_{\epsilon}(\Phi_{\epsilon}(\tau)) + \mb{\widehat{N}}_{\epsilon} (\Phi_{\epsilon}(\tau)), \quad &\tau \in (0,\infty),\\
        \Phi_{\epsilon}(0)=\mb{u} \quad & \mb{u} \in \mc{H}_r^{s,k},\\
    \end{cases}
\end{align}
for $\epsilon \in \R$, $|\epsilon| \leq \overline{\epsilon}$ by considering the corresponding integral equation
\begin{align}\label{integral equation}
    \Phi_{\epsilon}(\tau) = \mb{S}_{\epsilon}(\tau)\mb{u} + \int_0^{\tau} \mb{S}_{\epsilon}(\tau-\tau') \mb{\widehat{N}}_{\epsilon}(\Phi_{\epsilon}(\tau'))\, d\tau' \quad \text{for all} \quad \tau\geq 0 \quad \text{and} \quad \mb{u} \in \mc{H}_r^{s,k}.
\end{align}

We introduce the Banach space
\begin{align*}
    \mc{X} :=\{ \Phi \in C([0,\infty),\mc{H}_r^{s,k}): \norm{\Phi}_{\mc{X}} := \sup_{\tau>0} e^{\omega\tau} \norm{\Phi(\tau)}_{s,k}<\infty\},
\end{align*}
as well as
\begin{align*}
    \mc{X}_{\delta} := \{ \Phi \in \mc{X} : \norm{\Phi}_{\mc{X}} \leq \delta \} = \{ \Phi \in C([0,\infty),\mc{H}_r^{s,k}): \norm{\Phi(\tau)}_{s,k} \leq \delta\, e^{-\omega\tau}, \forall \, \tau > 0 \}.
\end{align*}
Following the standard approach, we introduce the correction term
\begin{align}\label{correction term}
    \mb{C}(\Phi,\epsilon,\mb{u}) := \mb{P}_{\epsilon} \left( \mb{u} + \int_0^{\infty}e^{-\tau'}\, \mb{\widehat{N}}_{\epsilon}(\Phi(\tau'))\,d\tau'\right)
\end{align}
to suppress the exponential growth of the semigroup on the unstable subspace. Consequently, we consider the fixed-point problem
\begin{align}\label{fixpoint problem}
    \Phi(\tau) = \mb{K}(\Phi,\epsilon,\mb{u})(\tau),
\end{align}
where $\mb{K}(\Phi,\epsilon,\mb{u})$ is defined as
\begin{align}
    \mb{K}(\Phi,\epsilon,\mb{u})(\tau) := \mb{S}_{\epsilon}(\tau)\left[ \mb{u}-\mb{C}(\Phi,\epsilon,\mb{u})\right] + \int_0^{\tau} \mb{S}_{\epsilon}(\tau-\tau')\mb{\widehat{N}}_{\epsilon}(\Phi(\tau'))\,d\tau'.
\end{align}

This modification stabilizes the evolution as the following result shows:

\begin{proposition}\label{Prop: Cauchy problem with modified initial data}
Let $n \geq 5$,  $(s,k) \in \R \times \N$ satisfy \eqref{condition on exponents}. There are constants $0 < \delta_0 < 1$ and $C_0 > 1$ such that for all $0 < \delta \leq \delta_0$, $C \geq C_0$, all $\epsilon \in \R$ with $\abs{\epsilon} \leq \overline{\epsilon}$ and all $\mb{u} \in \mc{H}_r^{s,k}$ with $\norm{\mb{u}}_{s,k}\leq \frac{\delta}{C}$ there exists a unique function $\Phi_{\epsilon}(\mb{u})\in \mc{X}_{\delta}$ such that \textup{(\ref{fixpoint problem})} holds for all $\tau\geq0$. Furthermore, the solution map $(\mb{u},\epsilon) \mapsto \Phi_{\epsilon}(\mb{u})$ is Lipschitz continuous, i.e.
    \begin{align*}
        \norm{\Phi_{\epsilon}(\mb{u}) - \Phi_{\kappa}(\mb{v})}_{\mc{X}} \lesssim \norm{\mb{u}-\mb{v}}_{s,k} + \abs{\epsilon-\kappa}
    \end{align*}
for all $\mb{u}, \mb{v} \in \mc{H}_r^{s,k}$ with $\norm{\mb{u}}_{s,k}, \norm{\mb{v}}_{s,k} \leq \frac{\delta}{C}$ and all $\abs{\epsilon}, \abs{\kappa} \leq \overline{\epsilon}$.
\end{proposition}

\begin{proof}
First, we show that the map $\mb{K}_{(\mb{u},\epsilon)}(\Phi) := \mb{K}(\Phi,\epsilon,\mb{u})$ is a well-defined contraction on $\mc{X}_{\delta}$ for all sufficiently large  $C>1$, sufficiently small $\delta >0$ and all $\mb{u} \in \mc{H}_r^{s,k}$ with $\norm{\mb{u}}_{s,k}\leq\frac{\delta}{C}$.\\
For this we will take $\Phi \in \mc{X}_{\delta}$ and $\tau \geq 0$ and notice that we can write $\mb{K}_{(\mb{u},\epsilon)}$ in the following way
\begin{align}
    &\mb{K}_{(\mb{u},\epsilon)}(\Phi)(\tau)\nonumber\\ = &\,\mb{S}_{\epsilon}(\tau)(\mb{I}-\mb{P}_{\epsilon})\mb{u} - \int_{\tau}^{\infty}e^{\tau-\tau'}\,\mb{P}_{\epsilon}\,\widehat{\mb{N}}_{\epsilon}(\Phi(\tau'))\,d\tau' + \int_0^{\tau}\mb{S}_{\epsilon}(\tau-\tau')(\mb{I}-\mb{P}_{\epsilon}) \widehat{\mb{N}}_{\epsilon}(\Phi(\tau'))\,d\tau'.\label{rewritten fixed-point operator} 
\end{align}
From this we obtain with Proposition \ref{projection properties} and Lemma \ref{Local Lipschitz continuity for nonlinearity II}
\begin{align*}
    \norm{\mb{K}_{(\mb{u},\epsilon)}(\Phi)(\tau)}_{s,k} \lesssim \frac{\delta}{C}e^{-\omega\,\tau} + \delta^2 e^{-2\omega\,\tau} + \delta^2 e^{-\omega\,\tau} (e^{-\omega\,\tau} + 1) \lesssim \left(\frac{1}{C}+\delta\right) \delta e^{-\omega\,\tau}
\end{align*}
so that we have
\begin{align*}
    \norm{\mb{K}_{(\mb{u},\epsilon)}(\Phi)(\tau)}_{s,k} \leq \delta e^{-\omega\,\tau}
\end{align*}
if we choose $C \geq C_0$ and $0 < \delta \leq \delta_0$ with $C_0$ sufficiently large and $\delta_0>0$ sufficiently small. 

Since the continuity of the mapping $\tau \mapsto \mb{K}_{(\mb{u},\epsilon)}(\Phi)(\tau)$ follows from the definition and dominated convergence we conclude that $\mb{K}_{(\mb{u},\epsilon)}: \mc{X}_{\delta} \to  \mc{X}_{\delta}$ is well-defined. 

To show that $\mb{K}_{(\mb{u},\epsilon)}$ is a contraction on $\mc{X}_{\delta}$ (for potentially even smaller $\delta>0$) we take $\Phi,\Psi \in \mc{X}_{\delta}$ and calculate for every $\tau \geq 0$ by using again the representation of $\mb{K}_{(\mb{u},\epsilon)}$ from Eq.~\eqref{rewritten fixed-point operator}
    \begin{align*}
        &\norm{\mb{K}_{(\mb{u},\epsilon)}(\Phi)(\tau)-\mb{K}_{(\mb{u},\epsilon)}(\Psi)(\tau)}_{s,k}\\ \lesssim &\,\int_{\tau}^{\infty}e^{\tau-\tau'}\lVert\widehat{\mb{N}}_{\epsilon}(\Phi(\tau'))-\widehat{\mb{N}}_{\epsilon}(\Psi(\tau'))\rVert_{s,k}\,d\tau' + \int_0^{\tau}e^{-\omega\,(\tau-\tau')}\lVert\widehat{\mb{N}}_{\epsilon}(\Phi(\tau'))-\widehat{\mb{N}}_{\epsilon}(\Psi(\tau'))\rVert_{s,k}\,d\tau'\\ \lesssim & \, \left(\delta e^{-2\omega\,\tau}+\delta e^{-\omega\,\tau}(e^{-\omega\,\tau}+1)\right)\norm{\Phi-\Psi}_{\mc{X}}.
    \end{align*}
If we now choose $\delta_0>0$ sufficiently small we get\[\norm{\mb{K}_{(\mb{u},\epsilon)}(\Phi)-\mb{K}_{(\mb{u},\epsilon)}(\Psi)}_{\mc{X}} \leq \frac{1}{2}\norm{\Phi-\Psi}_{\mc{X}},\]
for all $0 < \delta \leq \delta_0$, $C \geq C_0$ and all $\epsilon \in \R$ with $\abs{\epsilon} \leq \overline{\epsilon}$. The existence of a unique solution in $\mc{X}_{\delta}$ now follows by application of the contraction mapping principle. 

What is now left to show is the Lipschitz continuity of the solution map $(\mb{u},\epsilon) \mapsto \Phi_{\epsilon}(\mb{u}) \in \mc{X}_{\delta}$.\\
For this we take $(\mb{u},\epsilon),(\mb{v},\kappa) \in \mc{B}_{\frac{\delta}{C}}\times [-\overline{\epsilon},\overline{\epsilon}]$ and obtain by the previous considerations functions $\Phi_{\epsilon}(\mb{u}),\Phi_{\kappa}(\mb{v}) \in \mc{X}_{\delta}$ solving
\begin{align*}
     \Phi_{\epsilon}(\mb{u})(\tau) = \mb{K}(\Phi_{\epsilon}(\mb{u})(\tau),\epsilon,\mb{u})(\tau) \quad \text{and} \quad \Phi_{\kappa}(\mb{v})(\tau) = \mb{K}(\Phi_{\kappa}(\mb{v})(\tau),\kappa,\mb{v})(\tau) \quad \forall \, \tau \geq 0. 
\end{align*}
We now show that
\[\norm{\mb{K}(\Phi_{\epsilon}(\mb{u}),\epsilon,\mb{u})-\mb{K}(\Phi_{\kappa}(\mb{v}),\kappa,\mb{v})}_{\mc{X}} \lesssim \norm{\mb{u}-\mb{v}}_{s,k} + \abs{\epsilon-\kappa}.\] For this we take $\tau \geq 0$ and estimate the terms in \eqref{rewritten fixed-point operator} separately. For the first term we simply get from \eqref{Lipschitz property of semigroup on stable subspace}
\begin{align*}
    \norm{\mb{S}_{\epsilon}(\tau)(\mb{I}-\mb{P}_{\epsilon})\mb{u} - \mb{S}_{\kappa}(\tau)(\mb{I}-\mb{P}_{\kappa})\mb{v}} \lesssim \frac{\delta}{C}e^{-\omega\,\tau}\abs{\epsilon-\kappa} + e^{-\omega\,\tau}\norm{\mb{u}-\mb{v}}_{s,k}.
\end{align*}
For the second term we apply Lemma \ref{Local Lipschitz continuity for nonlinearity II} and the Lipschitz continuity of $\mb{P}_{\epsilon}$ to get
\begin{align*}
    &\int_{\tau}^{\infty}e^{\tau-\tau'}\lVert\mb{P}_{\epsilon}\,\widehat{\mb{N}}_{\epsilon}(\Phi_{\epsilon}(\mb{u})(\tau'))-\mb{P}_{\kappa}\,\widehat{\mb{N}}_{\kappa}(\Phi_{\kappa}(\mb{v})(\tau'))\rVert_{s,k}\,d\tau'\\ \lesssim & \, \abs{\epsilon-\kappa}\int_{\tau}^{\infty}e^{\tau-\tau'} \lVert\widehat{\mb{N}}_{\kappa}(\Phi_{\kappa}(\mb{v})(\tau'))\rVert_{s,k}\,d\tau' + \int_{\tau}^{\infty}e^{\tau-\tau'}\lVert\widehat{\mb{N}}_{\epsilon}(\Phi_{\epsilon}(\mb{u})(\tau'))-\widehat{\mb{N}}_{\kappa}(\Phi_{\kappa}(\mb{v})(\tau'))\rVert_{s,k}\,d\tau' \\ \lesssim & \, \abs{\epsilon-\kappa}\, \delta^2 \, e^{-2\omega\,\tau} + \delta \, e^{-2\omega\,\tau} \norm{\Phi_{\epsilon}(\mb{u})-\Phi_{\kappa}(\mb{v})}_{\mc{X}}
\end{align*}
and for the last term we similarly get
\begin{align*}
    &\int_0^{\tau}\norm{\mb{S}_{\epsilon}(\tau-\tau')(\mb{I}-\mb{P}_{\epsilon}) \widehat{\mb{N}}_{\epsilon}(\Phi_{\epsilon}(\mb{u})(\tau')) - \mb{S}_{\kappa}(\tau-\tau')(\mb{I}-\mb{P}_{\kappa}) \widehat{\mb{N}}_{\kappa}(\Phi_{\kappa}(\mb{v})(\tau'))}_{s,k}\,d\tau' \\ \lesssim & \,\abs{\epsilon-\kappa} \int_0^{\tau}e^{-\omega\,(\tau-\tau')}\lVert\widehat{\mb{N}}_{\kappa}(\Phi_{\kappa}(\mb{v})(\tau'))\rVert_{s,k}\,d\tau'\\ & + \int_0^{\tau}e^{-\omega\,(\tau-\tau')}\lVert\widehat{\mb{N}}_{\epsilon}(\Phi_{\epsilon}(\mb{u})(\tau')) - \widehat{\mb{N}}_{\kappa}(\Phi_{\kappa}(\mb{v})(\tau'))\rVert_{s,k}\,d\tau' \\ \lesssim & \, \abs{\epsilon-\kappa} \, \delta^2 \, e^{-\omega\,\tau} + \abs{\epsilon-\kappa} \, \delta \, e^{-\omega\,\tau} \norm{\Phi_{\epsilon}(\mb{u}) - \Phi_{\kappa}(\mb{v})}_{\mc{X}}.
\end{align*}
With that we obtain 
    \begin{align*}
        \norm{\Phi_{\epsilon}(\mb{u})-\Phi_{\kappa}(\mb{v})}_{\mc{X}} =& \norm{\mb{K}(\Phi_{\epsilon}(\mb{u}),\epsilon,\mb{u})-\mb{K}(\Phi_{\kappa}(\mb{v}),\kappa,\mb{v})}_{\mc{X}} \\ \lesssim & \, \abs{\epsilon-\kappa} + \norm{\mb{u}-\mb{v}}_{s,k} + \delta \norm{\Phi_{\epsilon}(\mb{u})-\Phi_{\kappa}(\mb{v})}_{\mc{X}}.
    \end{align*}
    For small enough $\delta>0$ we get $\norm{\Phi_{\epsilon}(\mb{u})-\Phi_{\kappa}(\mb{v})}_{\mc{X}} \lesssim \abs{\epsilon-\kappa} + \norm{\mb{u}-\mb{v}}_{s,k}$ as desired.
    
\end{proof}

Now we will go back to considering the specific form of the initial data  of the Cauchy problem \eqref{Cauchy problem for perturbation stability analysis}. For this, we define the following initial data operator 
\begin{align}
\mb{U}_{\epsilon}(\mb{v},T) := \mb{v}^{\mb{T}} + \mb \Psi_{\epsilon}^{\mb{T}} - \mb \Psi_{\epsilon} :=
    \begin{pmatrix}
            Tv_1(T\cdot)\\
            T^2v_2(T\cdot)
        \end{pmatrix} + \begin{pmatrix}
            T\psi_{\epsilon,1}(T\cdot) - \psi_{\epsilon,1}\\
            T^2\psi_{\epsilon,2}(T\cdot) - \psi_{\epsilon,2}
        \end{pmatrix}.
\end{align}

\begin{lemma}\label{Lemma: Initial data operator}
Let $n \geq 5$ and  $(s,k) \in \R \times \N$ satisfy \eqref{condition on exponents}. Let $0 < \delta \leq \frac{1}{2}$. For every $\epsilon \in \R$, $\abs{\epsilon} \leq \overline{\epsilon}$ the map
\[T \mapsto \mb U_{\epsilon}(\mb v, T): [1-\delta,1+\delta] \to  \mc{H}_r^{s,k} \]
is continuous for $\mb v \in  \mc{H}_r^{s,k}$. Furthermore, for every $T \in [\frac{1}{2},\frac{3}{2}]$ the initial data operator can be written as
\begin{align}\label{equation for initial data operator}
   \mb  U_{\epsilon}(\mb{v},T) = \mb{v}^T + (T-1)\mb{g}_{\epsilon} + \mb R_{\epsilon}(T),
\end{align}
and there exists a constant $M_{\epsilon} > 0$ such that 
\[ \| \mb R_{\epsilon}(T)\| \leq M_{\epsilon}|T-1|^2. \]
\end{lemma}

\begin{proof}
Eq.~\eqref{equation for initial data operator} follows from Taylor's Theorem applied to the map $[\frac{1}{2},\frac{3}{2}] \to \mc{H}_r^{s,k}, T \mapsto \Psi^T_{\epsilon}$ using the fact that
\begin{align*}
    \partial_T \restr{\begin{pmatrix}
        T\psi_{\epsilon,1}(T\cdot)\\
        T^2\psi_{\epsilon,2}(T\cdot)
        \end{pmatrix}}{T=1} = \mb{g}_{\epsilon}. 
\end{align*}
The components $R_{\epsilon,1}(T)$ and $R_{\epsilon,2}(T)$ of the remainder term $\mb R_{\epsilon}(T)$ satisfy 
\begin{align*}
    \norm{R_{\epsilon,i}(T)}_{\dot{H}^{s-(i-1)}\cap\dot{H}^{k-(i-1)}(\R^n)} \lesssim (T-1)^2 \sum_{j=0}^2 \norm{\Lambda^j\psi_{\epsilon,i}}_{\dot{H}^{s-(i-1)}\cap\dot{H}^{k-(i-1)}(\R^n)} \quad \text{for} \quad i = 1,2.
\end{align*}
The norms on the right hand side are finite, due to the decay of $\psi_{\epsilon,i}$ described in Proposition \ref{Prop: decay of blowup solution}. This determines the constant $M_{\epsilon} > 0$. The continuity follows by a standard argument, see for example \cite{Glo23}, Lemma 8.2.
\end{proof}

Now, we are in the position to prove the central result of this section.

\begin{theorem}\label{Thm:CoMain}
Let $n \geq 5$ and  $(s,k) \in \R \times \N$ satisfy \eqref{condition on exponents}. For any $\epsilon \in \R$ with $\abs{\epsilon} \leq \overline{\epsilon}$, there are constants $0 < \delta_{\epsilon} < 1$ and $C_{\epsilon} > 1$ such that for all $0 < \delta \leq \delta_{\epsilon}$ and all $C \geq C_{\epsilon}$ the following statement holds:
    If $\mb{v} \in \mc{H}_r^{s,k}$ is real-valued with $\norm{\mb{v}}_{s,k} \leq \frac{\delta}{C^2}$	then there exists a $T_{\epsilon} = T_{\epsilon}(\mb v) \in [1-\frac{\delta}{C}, 1+\frac{\delta}{C}]$ and a unique solution $\Phi_{\epsilon} \in C([0,\infty);  \mc{H}_r^{s,k})$ satisfying
 \begin{align}\label{integral equation with modification}
     \Phi_{\epsilon}(\tau) = \mb{S}_{\epsilon}(\tau) \mb{U}_{\epsilon}(\mb{v},T_{\epsilon})+\int_0^{\tau}\mb{S}_{\epsilon}(\tau-\tau')\widehat{\mb{N}}_{\epsilon}(\Phi_{\epsilon}(\tau'))\,d\tau' \quad \text{for all} \quad \tau \geq 0.
 \end{align}
Furthermore,
\[ \| \Phi_{\epsilon}(\tau) \|_{s,k} \leq \delta e^{- \omega \tau}, \quad \forall \tau \geq 0. \]
\end{theorem}

\begin{proof}
Let $0 < \delta \leq \delta_0$ and $C \geq C_0 \geq 1$ with $\delta_0$ and $C_0$ as in Proposition \ref{Prop: Cauchy problem with modified initial data}. Let  $\abs{\epsilon} \leq \overline{\epsilon}$ and $\mb{v} \in \mc{H}_r^{s,k}$ with $\norm{\mb{v}}_{s,k} \leq \frac{\delta}{C^2}$ . Then  we obtain from Lemma \ref{Lemma: Initial data operator} 
\begin{align*}
    \norm{\mb U_{\epsilon}(\mb{v},T)}_{s,k} \lesssim &\norm{\mb{v}^T}_{s,k} + \abs{T-1}\norm{\mb{g}_{\epsilon}}_{s,k} + \norm{\mb R_{\epsilon}(T)}_{s,k}\\ \lesssim &\frac{\delta}{C^2} + \frac{\delta}{C} L_{\epsilon} + \frac{\delta^2}{C^2} M_{\epsilon}
\end{align*}
for every $T \in [1-\frac{\delta}{C}, 1+\frac{\delta}{C}]$ where $L_{\epsilon}, M_{\epsilon} > 0$ are some constants depending on $\epsilon$. If we now choose $\delta$ sufficiently small and $C$ sufficiently large we obtain for every $T \in [1-\frac{\delta}{C}, 1+\frac{\delta}{C}]$  from Proposition \ref{Prop: Cauchy problem with modified initial data} the existence of a unique $\Phi_{\epsilon} = \Phi_{\epsilon}(\mb{v},T) \in \mc{X}_{\delta}$ which solves
\begin{align}\label{equation with correction term}
    \Phi_{\epsilon}(\tau) = \mb{S}_{\epsilon}(\tau)\left[ \mb{U}_{\epsilon}(\mb{v},T)-\mb{C}(\Phi_{\epsilon},\epsilon,\mb{U}_{\epsilon}(\mb{v},T))\right] + \int_0^{\tau} \mb{S}_{\epsilon}(\tau-\tau')\mb{\widehat{N}}_{\epsilon}(\Phi_{\epsilon}(\tau'))\,d\tau'.
\end{align}
Since $\mb{C}$ takes values in $\ran \mb{P}_{\epsilon} = \langle \mb{g}_{\epsilon} \rangle$ it is enough to show,  given $\mb v$, the existence of a $T$ such that
\begin{align}\label{Eq: correction term orthogonal to unstable subspace}
    \langle \mb{C}(\Phi_{\epsilon}(\mb{v},T),\epsilon,\mb{U}_{\epsilon}(\mb{v},T)), \mb{g}_{\epsilon}\rangle_{s,k} = 0.
\end{align}
Due to Lemma \ref{Lemma: Initial data operator} and the definition of $\mb{C}$ this equation reads as
\begin{align*}
    0 = \langle \mb{P}_{\epsilon} \mb{v}^T, \mb{g}_{\epsilon} \rangle_{s,k} + (T-1) \norm{\mb{g}_{\epsilon}}^2_{s,k} + \langle \mb{P}_{\epsilon}\mb R_{\epsilon}(T), \mb{g}_{\epsilon} \rangle_{s,k} +  \langle \mb{P}_{\epsilon} \int_0^{\infty} e^{-\tau'} \widehat{\mb{N}}_{\epsilon}(\Phi_{\epsilon}(\tau'))\,d\tau', \mb{g}_{\epsilon} \rangle_{s,k},
\end{align*}
which can be written as a fixed-point equation for $T \in [1-\frac{\delta}{C}, 1+\frac{\delta}{C}]$,
\begin{align}\label{fixed-point equation for T}
    T = 1 + \langle \mb{P}_{\epsilon} \mb{v}^T, \widehat{\mb{g}}_{\epsilon} \rangle_{s,k} + \langle \mb{P}_{\epsilon} \mb R_{\epsilon}(T), \widehat{\mb{g}}_{\epsilon} \rangle_{s,k} +  \langle \mb{P}_{\epsilon} \int_0^{\infty} e^{-\tau'} \widehat{\mb{N}}_{\epsilon}(\Phi_{\epsilon}(\tau'))\,d\tau', \widehat{\mb{g}}_{\epsilon} \rangle_{s,k}, 
\end{align}
where we have set $\widehat{\mb{g}}_{\epsilon} = \mb{g}_{\epsilon} / \norm{\mb{g}_{\epsilon}}^2_{s,k}$. Now we obtain from the assumptions on $\mb{v}$, the fact that $\Phi_{\epsilon}$ belongs to $\mc{X}_{\delta}$ and Lemma \ref{Lemma: Initial data operator} as well as Lemma \ref{Local Lipschitz continuity for nonlinearity II} the following estimate
\begin{align*}
    &\abs{\langle \mb{P}_{\epsilon} \mb{v}^T, \widehat{\mb{g}}_{\epsilon} \rangle_{s,k}} + \abs{\langle \mb{P}_{\epsilon} \mb R_{\epsilon}(T), \widehat{\mb{g}}_{\epsilon} \rangle_{s,k}} +  \abs{\langle \mb{P}_{\epsilon} \int_0^{\infty} e^{-\tau'} \widehat{\mb{N}}_{\epsilon}(\Phi_{\epsilon}(\tau'))\,d\tau', \widehat{\mb{g}}_{\epsilon} \rangle_{s,k}} \\ \lesssim &\frac{\delta}{C^2}L_{\epsilon} + \frac{\delta^2}{C^2} M_{\epsilon} + \delta^2 N_{\epsilon}
\end{align*}
for again some constants $L_{\epsilon}, M_{\epsilon}$ and $N_{\epsilon} >0$.
If we now choose  $C \geq C_{\epsilon}$ and $0< \delta < \delta_{\epsilon} $  with $ C_{\epsilon}> 1$ sufficiently large and $\delta_{\epsilon} < 1$ sufficiently small we get that the right-hand side of \eqref{fixed-point equation for T} is a continuous mapping from $[1-\frac{\delta}{C}, 1+\frac{\delta}{C}]$ into itself so that we obtain by the fixed-point theorem of Brouwer a $T_{\epsilon} \in \left[ 1- \frac{\delta}{C}, 1 + \frac{\delta}{C}\right]$ such that equation \eqref{Eq: correction term orthogonal to unstable subspace} is fulfilled. We therefore conclude that the corresponding solution $\Phi_{\epsilon}(\mb{v},T_{\epsilon})$ is a solution to \eqref{integral equation with modification}. The claimed uniqueness follows along the lines of the proof of Theorem $5.4$ in \cite{GloKisSch23}.

\end{proof}

Now we will show the regularity of the just constructed solution.

\begin{proposition}\label{Upgrade to classical solution}
Let $\mb v$ satisfy the assumptions of Theorem \ref{Thm:CoMain}. If  $\mb{v} \in \mc{S}(\R^n) \times \mc{S}(\R^n)$ then the solution $\Phi_{\epsilon}$ of Eq.~\eqref{integral equation with modification} guaranteed by  Theorem \ref{Thm:CoMain} is smooth. More precisely, $\Phi_{\epsilon}(\tau)(\xi) = (\phi_{\epsilon,1}(\tau,\xi), \phi_{\epsilon,2}(\tau,\xi))$ with $\phi_{\epsilon,i} \in C^{\infty}([0,\infty) \times \R^n)$. Furthermore, the components satisfy
\begin{align}\label{Eq:phi_pointwise}
\begin{pmatrix}
\partial_{\tau} \phi_{\epsilon,1}(\tau, \xi) \\
\partial_{\tau} \phi_{\epsilon,2}(\tau, \xi) 
\end{pmatrix} = \begin{pmatrix}
\phi_{\epsilon,2}(\tau, \xi) - \Lambda \, \phi_{\epsilon,1}(\tau, \xi) - \phi_{\epsilon,1}(\tau, \xi) \\
\Delta_{\xi} \phi_{\epsilon,1}(\tau, \xi) - \Lambda \, \phi_{\epsilon,2}(\tau, \xi) - 2 \phi_{\epsilon,2}(\tau, \xi) + \hat N_{\varepsilon} (\phi_{\epsilon,1}(\tau, \cdot))(\xi)
\end{pmatrix}
\end{align}
for all $\xi \in \R^d$ and all $\tau \geq 0$, and
\begin{align}\label{Eq:data}
  \begin{pmatrix}  \phi_{\epsilon,1}(0, \cdot) \\
 \phi_{\epsilon,2}(0, \cdot) 
\end{pmatrix} =  \begin{pmatrix}
            T_{\epsilon} v_1(T_{\epsilon} \cdot) + T_{\epsilon} \psi_{\epsilon,1}(T_{\epsilon}\cdot) - \psi_{\epsilon,1}  \\
            T_{\epsilon}^2v_2(T_{\epsilon}\cdot)  +   T_{\epsilon}^2\psi_{\epsilon,2}(T_{\epsilon}\cdot) - \psi_{\epsilon,2}
        \end{pmatrix} .
        \end{align}
\end{proposition}

\begin{proof}

Assume that all of the constants are chosen such that Theorem \ref{Thm:CoMain} is fulfilled. Then there exists a solution $\Phi_{\epsilon} = \Phi_{\epsilon}(\mb{v},T_{\epsilon}) \in \mc{X}_{\delta}$ satisfying
    \begin{align*}
     \Phi_{\epsilon}(\tau) = \mb{S}_{\epsilon}(\tau) \mb{U}_{\epsilon,T_{\epsilon}}(\mb{v})+\int_0^{\tau}\mb{S}_{\epsilon}(\tau-\tau')\widehat{\mb{N}}_{\epsilon}(\Phi_{\epsilon}(\tau'))\,d\tau' \quad \text{for all} \quad \tau \geq0.
 \end{align*}
We will now show via an inductive argument that $\Phi_{\epsilon}(\tau)$ belongs to $\mc{H}^{s,\ell}$ for every $\ell \geq k$ and every $\tau \geq 0$. We already know from Lemma \ref{Local Lipschitz continuity for nonlinearity II} that $\widehat{\mb{N}}_{\epsilon}$ maps $\mc{H}^{s,\ell}$ into $\mc{H}^{s,\ell+1}$ and that the initial data operator $\mb{U}_{\epsilon,T_{\epsilon}}(\mb{v})$ belongs to $\mc{H}^{s,\ell}$ for every $\ell \geq k$. Since the restriction of the semigroups are equal to the restricted semigroups, see Lemma C.1 from \cite{Glo23}, we can already inductively conclude that $\Phi_{\epsilon}(\tau)$ belongs to $\mc{H}^{s,\ell}$ for every $\ell \geq k$ and every $\tau \geq 0$. From the Sobolev embedding \eqref{C^m embedding two components}  we can therefore conclude that $\Phi_{\epsilon}(\tau)$ belongs to $C^{\infty}(\R^n) \times C^{\infty}(\R^n)$ for every $\tau \geq 0.$

Now we show that $\mb U_{\epsilon}(\mb{v},T) \in \mc D(\mb{L}_{\epsilon})$ for $\mb{v} \in \mc{S}(\R^n) \times \mc{S}(\R^n)$ and $T \in [\frac{1}{2},\frac{3}{2}]$. This is not entirely obvious as the blowup solution is not explicit. Set $\hat {\psi_1}:= T\psi_{\epsilon,1}(T\cdot) - \psi_{\epsilon,1}$, $\hat {\psi_2}:= T^2 \psi_{\epsilon,2}(T\cdot) - \psi_{\epsilon,2}$. We verify that for every $\beta \in \N_0^n$ we have 
\begin{align}\label{decay_data}
    \abs{\partial^{\beta}\hat \psi_{1}(\xi)} \lesssim \langle \xi \rangle^{s-\frac{n}{2}-\abs{\beta}-1} \quad \text{and} \quad \abs{\partial^{\beta} \hat \psi_{2}(\xi)} \lesssim \langle \xi \rangle^{s-\frac{n}{2}-\abs{\beta}-2}
\end{align}
and apply Lemma 4.5 from \cite{Glo23} as in the proof of Lemma \ref{Le:g}.
Let $\tilde \psi_{\epsilon,i}$, $i \in \{1,2\}$ denote the radial representative of $\psi_{\epsilon,i}$. We start with the second component: As outlined in the proof of Proposition \ref{Prop: decay of blowup solution} we have 
\[ \tilde \psi_{\epsilon,2} = c_2 w_1 + \tilde c_2 w_2, \]
with $w_i$ given in Eq.~\eqref{Eq:fundamental_sys} and $c_2,\tilde c_2 \in \C$. Now, for large $\rho > 0$, 
\begin{align}\label{representation of second component}
\tilde \psi_{\epsilon,2}\left (\rho \right ) = c_2 \rho^{-3} h_1(\rho^{-1}) +\tilde c_2 \rho^{-2} h_2(\rho^{-1}) + \tilde c \rho^{-3} \log(\rho) h_1(\rho^{-1}) 
\end{align}
for some $\tilde c \in \C$ with $h_1,h_2$ analytic around zero and $h_i(0) = 1$, $i \in \{1,2\}$. By the scaling behavior of the second term, we find that the bad behavior cancels and thus 
 \[ |T^2 \tilde \psi_{\epsilon,2}(T \rho) -  \tilde \psi_{\epsilon,2}(\rho)|  \lesssim_{T} \rho^{-3} \log(\rho),\]
which implies the bound for  $\hat \psi_{2}$ in the case $\beta$ equal to zero since we assume $s > \frac{n}{2} - 1$. For higher derivatives, the analogous bounds follow using the analyticity of $h$. \\ 
For the first component we set $f_{\varepsilon}(\rho) = \rho \tilde \psi_{\epsilon,1}(\rho)$ and infer with Proposition \ref{Prop: decay of blowup solution} that there are constants $c_1,\tilde c_1 \in \R$ such that 
\[ \lim_{\rho \to \infty} f_{\varepsilon}(\rho) = c_1, \quad \lim_{\rho \to \infty} \rho^2 f'_{\varepsilon}(\rho) = \tilde c_1.\]
Thus, $v_{\varepsilon}(y):= f_{\varepsilon}(\frac{1}{y})$ can be extended to a continuously differentiable function $v_{\varepsilon} \in C^1[0,1]$.  By Taylor's theorem,  
$v_{\varepsilon}(y) = c_1 - \tilde c_1 y + o(y)$,
for $y > 0$ close to zero and thus $f_{\varepsilon}(\rho) = c_1 - \tilde c_1 \rho^{-1} + o(\rho^{-1})$.
Consequently,
\[ |T \tilde \psi_{\epsilon,1}(T \rho) -  \tilde \psi_{\epsilon,1}(\rho)| = \rho^{-1} |f_{\varepsilon}(T\rho) - f_{\varepsilon}(\rho)| \lesssim_T \rho^{-2} \] 
for large values of $\rho$, which implies the bound for $\hat \psi_{1}$ in the case $\beta$ equal to zero.\\
For the case $\abs{\beta} =1$  we calculate for every $j \in \{1,\ldots,n\}$ using the crucial fact $f_{\epsilon}'(\rho) = \widetilde{\psi}_{\epsilon,2}(\rho)$
\begin{align*}
\abs{\partial^{e_j}\widehat{\psi}_1(\xi)} &= \abs{T^2\partial^{e_j}\psi_{\epsilon,1}(T\xi) - \partial^{e_j}\psi_{\epsilon,1}(\xi)} = \frac{\abs{\xi_j}}{\abs{\xi}}\abs{T^2\widetilde{\psi}_{\epsilon,1}'(T\abs{\xi}) - \widetilde{\psi}'_{\epsilon,1}(\abs{\xi})}\\ &= \frac{\abs{\xi_j}}{\abs{\xi}^3}\abs{T\abs{\xi}\widetilde{\psi}_{\epsilon,2}(T\abs{\xi}) - \abs{\xi}\widetilde{\psi}_{\epsilon,2}(\abs{\xi}) + f_{\epsilon}(\abs{\xi}) - f_{\epsilon}(T\abs{\xi})}\\ & = \frac{\abs{\xi_j}}{\abs{\xi}^3}\abs{T\abs{\xi}\widetilde{\psi}_{\epsilon,2}(T\abs{\xi}) - \abs{\xi}\widetilde{\psi}_{\epsilon,2}(\abs{\xi}) +\abs{\xi}\widehat{\psi}_1(\xi)} \lesssim_T \abs{\xi}^{-3}.
\end{align*}
The last inequality follows from the decay of $\widetilde{\psi}_{\epsilon,2}$ (here one does not need the cancellation of the worst behaving term) and the above shown decay for $\widehat{\psi}_1$ in the case $\abs{\beta} = 1$.\\
The decay for the higher derivatives can now be shown via the representation \eqref{representation of second component} of $\widetilde{\psi}_{\epsilon,2}$ and via induction on $\widehat{\psi}_1$ so that we can conclude that $\mb{U}_{\epsilon}(\mb{v},T)$ belongs to the domain of $\mb{L}_{\epsilon}$. Hence, we invoke [\cite{Paz83}, p.189, Theorem 1.6] to infer that  $\Phi_{\epsilon} \in C^1([0,\infty), \mc{H}_r^{s,k})$ is a classical solution of the operator equation and the claimed regularity follows form Sobolev embedding, the spatial regularity proved above and the Theorem of Schwarz as formulated in \cite{Rud76}, p. 235, Theorem 9.41. In particular,  the components of satisfy Eq.~\eqref{Eq:phi_pointwise} by definition of the operators involved. 
\end{proof}

Now we are finally able to prove our main stability results, Theorem \ref{Theorem: stability of blowup  solution} and Theorem 
\ref{Th:Stability_NormalCoord}.

\begin{proof}[Proof of Theorem \ref{Theorem: stability of blowup  solution}]

Under the assumption stated in Theorem \ref{Theorem: stability of blowup  solution} choose $\omega = \widetilde{\omega}\slash 2$ and $0 < \overline{\epsilon}$, depending on $\omega$, as at the beginning of Section \ref{Section: Nonlinear Cauchy Problem}

  For $\epsilon \in \R$ with $\abs{\epsilon} \leq \overline{\epsilon}$ let  $\delta=\delta_{\epsilon}$ and $C = C_{\epsilon}$ denote the constants from Theorem \ref{Thm:CoMain}. Let $(\varphi_0, \varphi_1) \in \mc S(\R^n) \times \mc S(\R^n)$ be radial, real-valued, functions satisfying
\begin{align*}
    \norm{(\varphi_0,\varphi_1)}_{\dot{H}^s\cap\dot{H}^k(\R^n)\times\dot{H}^{s-1}\cap\dot{H}^{k-1}(\R^n)} < \frac{\delta}{C^2}.
\end{align*}
Then, by Theorem \ref{Thm:CoMain} and Proposition \ref{Upgrade to classical solution} there is a  $T = T_{\epsilon} \in [1-\frac{\delta}{C}, 1 + \frac{\delta}{C}]$ and unique radial functions  $(\varphi_{\epsilon,1}, \varphi_{\epsilon,2}) \in C^{\infty}([0,\infty) \times \R^n) \times C^{\infty}([0,\infty) \times \R^n) $ solving the initial value problem Eq.~\eqref{Eq:phi_pointwise} - \eqref{Eq:data}. Moreover,
\[ \|(\varphi_{\epsilon,1}(\tau,\cdot), \varphi_{\epsilon,2}(\tau,\cdot) ) \|_{s,k} \leq \delta e^{-\omega \tau}, \]
for all $\tau \geq 0$. We set
\begin{align*}
 v(t,x) : = v_{\epsilon}^T(t,x) + \frac{1}{T-t} \varphi_{\epsilon,1} \left(\log\left(\frac{T}{T-t}\right) ,\frac{x}{T-t}   \right). 
\end{align*}

By construction, $v \in C^{\infty}([0,T)\times \R^n)$ satisfies Eq.~\eqref{n-dimensional semilinear wave eq} and
    \begin{align*}
        ( v(0,\cdot),\partial_t v(0,\cdot))=(v_{\epsilon}^1(0,\cdot), \partial_tv_{\epsilon}^1(0,\cdot)) + (\varphi_0,\varphi_1).
    \end{align*} 
Moreover for $r \in [s,k]$,
\begin{align*}
\norm{\varphi_{\epsilon,1}(- \log(T-t) + \log T, \cdot)}_{\dot{H}^r(\R^n)} &  \lesssim \| \Phi_{\varepsilon}(\tau) \|_{s,k} \lesssim \delta (T-t)^{\omega} 
\end{align*}
and 
\begin{align*}
\norm{(\partial_0   + \Lambda + 1)\varphi_{\epsilon,1}(- \log(T-t) + \log T, \cdot)}_{\dot{H}^{r-1}(\R^n)} &  = \norm{\varphi_{\epsilon,2}(\log T - \log(T-t), \cdot)}_{\dot{H}^{r-1}(\R^n)} \\
&  \lesssim \| \Phi_{\varepsilon}(\tau) \|_{s,k} \lesssim \delta (T-t)^{\omega} 
\end{align*}
by definition and Theorem \ref{Thm:CoMain}.

\end{proof}

\subsubsection*{Proof of Theorem \ref{Th:Stability_NormalCoord}}
By the assumptions of Theorem \ref{Th:Stability_NormalCoord} the initial data are of the form
\begin{align*}
U_0(x) = U^1_{\epsilon}(0,x) + x v_0(|x|), \quad U_1(x) = \partial_t U^1_{\epsilon}(0,x) + x v_1(|x|)
\end{align*}
Consequently, $v_0, v_1 \in C^{\infty}_{e}[0,\infty)$ and by setting  
\[ \varphi_j(y) :=   v_j(|y|) \]
for $y \in \R^{d+2}$, we obtain radially symmetric, real-valued functions $( \varphi_0,  \varphi_1) \in \mc S(\R^{d+2}) \times  \mc S(\R^{d+2})$. By Proposition A.5 and Remark A.6 of \cite{Glo22} there exists a constant $C > 0$ such that 
\begin{align}
 \|(\varphi_0, \varphi_1) \|_{\dot{H}^s\cap\dot{H}^k(\R^{d+2})\times\dot{H}^{s-1}\cap\dot{H}^{k-1}(\R^{d+2})} \leq C  \norm{(\nu_0,\nu_1)}_{\dot{H}^s\cap\dot{H}^k(\R^{d},\R^d)\times\dot{H}^{s-1}\cap\dot{H}^{k-1}(\R^{d},\R^d)}.
\end{align}
If the Sobolev exponents $(s,k)$ satisfy condition \eqref{condition}, then \eqref{condition on exponents in n-dimensions} holds for $n := d+2$. Let  $\omega, \overline{\epsilon}, \delta, M >0$ be the constants from Theorem \ref{Theorem: stability of blowup  solution}. By setting $M_0 := C M$ and requiring 
\[\norm{(\nu_0,\nu_1)}_{\dot{H}^s\cap\dot{H}^k(\R^{d},\R^d)\times\dot{H}^{s-1}\cap\dot{H}^{k-1}(\R^{d},\R^d)} \leq \frac{\delta}{ M_0}, \]
we find that $(\varphi_0, \varphi_1)$ satisfy the assumption of Theorem \ref{Theorem: stability of blowup  solution}. Hence, there is a $T \in [1-\delta, 1+ \delta]$ and a unique radial solution $v \in C^{\infty}([0,T) \times \R^{d+2})$ to \eqref{NLWhigherdim}. If we set $v(t,\cdot) = \tilde v(t,|\cdot|)$ then $\tilde v$ solves Eq. \eqref{Cauchy problem 2} for $t \in [0,T)$ and can be written as 
 \[ \tilde v(t,|\cdot|) = \frac{1}{|\cdot|} f_{\epsilon} \left (\frac{|\cdot|}{T-t} \right )  + \frac{1}{T-t} \tilde \varphi \left (\log\left(\frac{T}{T-t}\right), \frac{|\cdot|}{T-t} \right)  \]
for $\tilde \varphi(t,\cdot) \in  C^{\infty}_{e}[0,\infty)$ satisfying
\begin{align}\label{Est:Phirad}
\begin{split}
\|\tilde \varphi (-\log(T-t) & + \log T,|\cdot |) \|_{\dot H^{r}(\R^{d+2})}  \\
& + \norm{(\partial_0   + \Lambda + 1)\tilde \varphi(- \log(T-t) + \log T, |\cdot|)}_{\dot{H}^{r-1}(\R^{d+2})}\lesssim \delta (T-t)^{\omega}
\end{split}
\end{align}
for all $r \in [s,k]$. We define for $x \in \R^d$, $U(t,x) := x \tilde v(t,|x|) \in C^{\infty}([0,T) \times \R^d, \R^d)$ and find that $U$ can be written as 
\[ U(t,x) = U_{\epsilon}(t,x) + \nu \left (t, \frac{x}{T-t} \right), \]
where $\nu$ is a co-rotational function defined via $\nu(t,x) = x \, \tilde \varphi (-\log(T-t) + \log T, \abs{x})$. The inequality from \eqref{Est:Phirad} now implies \eqref{Decay_Nu} by applying Proposition A.5 and Remark A.6 from \cite{Glo22} and the local uniform convergence follows immediately from Sobolev embedding.

\appendix

\section{Schauder-type estimates for parameter depending nonlinear operators}\label{AppendixA}

Here we will show an adaptation of a Schauder-type estimate whose original form can be found in \cite{Glo23}, p.26, Proposition A.1. 

\begin{proposition}\label{Prop: Schauder}
Let $n \geq 5$ and $\epsilon_0 > 0$. For $\epsilon \in \R$,  $\abs{\epsilon}\leq \epsilon_0 $ let $F_{\epsilon} \in C^\infty(\R)$ be a family of even functions such that for all $\ell \in \mathbb{N}_0$ there exists a constant $C_{\ell}\geq 0$ such that
\begin{equation} \label{boundedness of derivatives}
	|F_{\epsilon}^{(\ell)}(x)-F_{\kappa}^{(\ell)}(y)| \leq C_{\ell}\,\left( \abs{\epsilon - \kappa} + \abs{x-y} \right)
\end{equation}
holds for all $x,y \in \R$ and all $\abs{\epsilon}, \abs{\kappa}\leq \epsilon_0$. Then, for every $s\in\R$ and $k\in\N$ that satisfy
\begin{align}\label{condition on exponents II}
    \frac{n}{2}-1 < s \leq \frac{n}{2}-1 + \frac{1}{2(n-1)}, \quad k>n
\end{align}
we have
\begin{equation}\label{Schauder estimate}
	\norm{u_1u_2u_3 \left(F_{\epsilon}(|\cdot|v)-F_{\kappa}(|\cdot|v)\right)}_{\dot{H}^{s-1} \cap \dot{H}^{k}(\mathbb{R}^n)} \lesssim \abs{\epsilon-\kappa} \prod_{i=1}^{3} \norm{u_i}_{\dot{H}^{s} \cap \dot{H}^{k}(\mathbb{R}^n)} \sum_{j=0}^{k}\norm{v}_{\dot{H}^{s} \cap \dot{H}^{k}(\mathbb{R}^n)}^{2j}
\end{equation}
as well as
\begin{align}\label{Schauder II}
\begin{split}
	&\norm{u_1u_2u_3 \left(F_{\epsilon}(|\cdot|v_1)-F_{\epsilon}(|\cdot|v_2\right)}_{\dot{H}^{s-1} \cap \dot{H}^{k}(\mathbb{R}^n)}\\ &\lesssim \prod_{i=1}^{3} \norm{u_i}_{\dot{H}^{s} \cap \dot{H}^{k}(\mathbb{R}^n)} P(\norm{v_1}_{\dot{H}^{s} \cap \dot{H}^{k}(\mathbb{R}^n)},\norm{v_2}_{\dot{H}^{s} \cap \dot{H}^{k}(\mathbb{R}^n)})\norm{v_1-v_2}_{\dot{H}^{s} \cap \dot{H}^{k}(\mathbb{R}^n)}
\end{split}
\end{align}
for all $\abs{\epsilon},\abs{\kappa} \leq 1$ and all $u_1,u_2,u_3,v,v_1,v_2 \in \dot{H}^{s}_r(\R^n) \cap \dot{H}^{k}_r(\mathbb{R}^n)$ where $v,v_1$ and $v_2$ are real-valued and $P$ is a polynomial of degree $\leq 2k+1$.

\end{proposition}
\begin{proof}
We will start with the case where $u_1,u_2,u_3,v,v_1$ and $v_2$ belong to $C^{\infty}_{c,r}(\R^n)$. By repeated application of the fundamental theorem of calculus, see for example \cite{Ost24}, there exists for every $\epsilon$ a function  $G_{\epsilon} \in C^{\infty}[0,\infty)$ with $G_{\epsilon}(x^2)= F_{\epsilon}(x)$. By the assumptions on $F_{\epsilon}$ we obtain for every $\ell \in \N_0$ the existence of a constant $\widetilde{C}_{\ell}\geq 0$ such that
\begin{align*}
|G_{\epsilon}^{(\ell)}(x^2)-G_{\kappa}^{(\ell)}(y^2)| \leq \widetilde{C}_{\ell}\,\left( \abs{\epsilon - \kappa} + \abs{x-y} \right).
\end{align*}
Now we choose $s,k$ according to (\ref{condition on exponents II}) and note that by Lemma \ref{Sobolev embedding} it is enough to bound the $\dot{H}^{\lfloor s-1 \rfloor} \cap \dot{H}^k$-norm of the respective left-hand side of \eqref{Schauder estimate} and \eqref{Schauder II}. Therefore,  we can reduce the analysis to estimating 
\begin{align}\label{Eq:Schauder_1}
\partial^{\alpha}\left(u_1 u_2 u_3 \left(G_{\epsilon}(\abs{\cdot}^2v^2) - G_{\kappa}(\abs{\cdot}^2v^2)\right)\right)
\end{align}

 and
  \begin{align}\label{Eq:Schauder_2}
 \partial^{\alpha}\left(u_1 u_2 u_3 \left(G_{\epsilon}(\abs{\cdot}^2v_1^2) - G_{\epsilon}(\abs{\cdot}^2v_2^2)\right)\right)
 \end{align}
 in $L^2(\R^n)$ for $\abs{\alpha} \in \{\lfloor s-1 \rfloor,k\}$. 
For the first term we find that after applying the Leibniz rule and \eqref{boundedness of derivatives} it suffices to prove a suitable bound for
	 \begin{align}\label{Eq:Schauder_Leibniz}
	 	I(x):=x^\gamma \, \partial^{\alpha_1}u_1\, \partial^{\alpha_2}u_2 \, \partial^{\alpha_3}u_3 \, \prod_{j=1}^{2\ell}\partial^{\beta_j}v,
	 \end{align}
 where 
 \begin{equation}\label{Eq:Conds_sum_beta}
 	\sum_{i=1}^{3}|\alpha_i|+ 	\sum_{j=1}^{2\ell}|\beta_j|+ 2\ell- |\gamma| = \abs{\alpha}
 \end{equation}
with the condition that $\ell \leq \abs{\alpha}$  and $|\gamma| \leq 2\ell$. This corresponds exactly to the situation considered in \cite{Glo23}, Proposition A.1 and arguing along these lines we infer that
\begin{align}\label{estimate for J0}
    \norm{I}_{L^2(\R^n)} \lesssim \prod_{i=1}^3\norm{u_i}_{\dot{H}^{s} \cap \dot{H}^{k}(\mathbb{R}^n)} \prod_{j=1}^{2\ell}\norm{v}_{\dot{H}^{s} \cap \dot{H}^{k}(\mathbb{R}^n)}
\end{align}
for $\abs{\alpha} \in \{\lfloor s-1 \rfloor,k\}$. To handle the expression in \eqref{Eq:Schauder_2} we prove the bound
\begin{align}\label{estimate for J}
    \norm{J}_{L^2(\R^n)} \lesssim \prod_{i=1}^3\norm{u_i}_{\dot{H}^{s} \cap \dot{H}^{k}(\mathbb{R}^n)} \prod_{j=1}^{2\ell+1}\norm{v_j}_{\dot{H}^{s} \cap \dot{H}^{k}(\mathbb{R}^n)}
\end{align}
for
\begin{align}\label{Definition J}
     J(x) := x^{\gamma} \, \partial^{\alpha_1}u_1 \, \partial^{\alpha_2}u_2 \, \partial^{\alpha_3}u_3 \prod_{j=1}^{2\ell+1}\partial^{\beta_j}v_j
\end{align}
with functions $u_1,u_2,u_3,v_1,\ldots,v_{2\ell+1} \in C^{\infty}_{c,r}(\R^n)$,
and indices\\ $\ell \leq \abs{\alpha}$, $\alpha_1,\alpha_2,\alpha_3,\beta_1,\ldots,\beta_{2\ell+1},\gamma \in \N^n_0 $ with $\abs{\gamma} \leq 2\ell+1$ satisfying 
\begin{align*}
    \abs{\alpha_1} + \abs{\alpha_2} + \abs{\alpha_3} + \sum_{j=1}^{2\ell+1}\abs{\beta_j} + 2\ell+1-\abs{\gamma} = \abs{\alpha}.
\end{align*}
We will start with the case $\abs{\alpha} = \lfloor s \rfloor -1$. Here we define for $i = 2,3$ and $j = 1,\ldots,2\ell+1$
\begin{align*}
    a_i := \abs{\alpha_i} + \frac{2\ell+1+s-\abs{\alpha}}{2\ell+3} \quad \text{and} \quad b_j := \abs{\beta_j} + \frac{2\ell+1+s-\abs{\alpha}}{2\ell+3}
\end{align*}
and get 
\begin{align*}
     \norm{J}_{L^2(\R^n)} \lesssim \norm{\abs{\cdot}^{\abs{\alpha_1}-s}\partial^{\alpha_1}u_1}_{L^2(\mathbb{R}^n)} \prod_{i=2}^3\norm{\abs{\cdot}^{a_i}\partial^{\alpha_i}u_i}_{L^{\infty}(\mathbb{R}^n)} \prod_{j=1}^{2\ell+1}\norm{\abs{\cdot}^{b_j}\partial^{\beta_j}v_j}_{L^{\infty}(\mathbb{R}^n)}
\end{align*}
due to $\abs{\gamma} = \abs{\alpha_1} + a_2 + a_3 + \sum_{j=1}^{2\ell+1}b_j -s$.
Since $0 \leq s - \abs{\alpha_1} < \frac{n}{2}$ we can use Hardy's inequality (\cite{MusSch13}, p. 243, Theorem 9.5) for the first term and for the rest of the terms we use the generalized Strauss inequality \eqref{Eq: generalized Strauss inequality} to obtain \eqref{estimate for J}. We are allowed to use these inequalities since we have $0 < a_i,b_j < \frac{n-1}{2}$ as well as
\begin{align}\label{restriction on s}
    s \leq \frac{n}{2}-a_i + \abs{\alpha_i} < \frac{n}{2} \quad \text{and} \quad s \leq \frac{n}{2}-b_j + \abs{\beta_j} < \frac{n}{2}.
\end{align}
Since we are in the case where $\abs{\alpha}$ is strictly smaller than $s$ we immediately obtain $0 <a_i,b_j$. For the upper bound on $a_i$ and $b_j$ one has
\begin{align*}
	a_i = \abs{\alpha_i} + 1 + \frac{s-\lfloor \frac{n}{2} \rfloor}{2\ell+3} \leq \frac{n}{2} - 1 + \frac{s-\lfloor \frac{n}{2} \rfloor}{2\ell+3} \leq \frac{n}{2} - 1 < \frac{n-1}{2}
\end{align*}
and the same holds true for $b_j$. Since $a_i$ and $b_j$ are strictly larger than $\abs{\alpha_i}$ and $\abs{\beta_j}$ respectively the upper bound from \eqref{restriction on s} is immediate.
The most restrictive inequality (for $s$) is  $s \leq \frac{n}{2}-a_i + \abs{\alpha_i}$. This inequality is equivalent to $s + 1 + \frac{s-\lfloor\frac{n}{2}\rfloor}{2\ell+3}\leq \frac{n}{2}$ which can be reformulated as $2(\ell+2)s + 2 \ell + 3 \leq (\ell+\frac{3}{2})n + \lfloor \frac{n}{2} \rfloor$. By making a case distinction between odd and even $n$ and using the fact that $0 \leq \ell \leq  \lfloor s \rfloor - 1$ this inequality holds for $s \leq \frac{n}{2}-1+\frac{1}{2n-2}$.

The case $\abs{\alpha} = k$ will be divided into two sub-cases. First, we assume that the highest derivative in \eqref{Definition J} is of order at least $\frac{n}{2}-1$. Without loss of generality we assume that this derivative is $\alpha_1.$ We now split the $L^2-$norm of $J$ into the unit ball and its complement. For $\abs{x} \leq 1$ we get for $a_1 := \min \{ 1,k-\abs{\alpha_1}\}$
\begin{align*}
    \abs{J(x)} \lesssim \abs{x}^{-a_1}\, \abs{\partial^{\alpha_1}u_1} \, \abs{\partial^{\alpha_2}u_2} \, \abs{\partial^{\alpha_3}u_3} \, \prod_{j=1}^{2\ell+1}\abs{\partial^{\beta_j}v_j}. 
\end{align*}
For the first term we use Hardy's inequality and for the rest of the terms we want to use Lemma \ref{C^m embedding} to estimate
\begin{align*}
    \norm{\partial^{\alpha_i}u_i}_{L^{\infty}(\R^n)} \lesssim \norm{u_i}_{\dot{H}^{s} \cap \dot{H}^{k}(\mathbb{R}^n)} \quad \text{and} \quad \norm{\partial^{\beta_j}v_j}_{L^{\infty}(\R^n)} \lesssim \norm{v_j}_{\dot{H}^{s} \cap \dot{H}^{k}(\mathbb{R}^n)}
\end{align*}
for every $i = 2,3$ and $j = 1,\ldots, 2\ell+1$. To see this we consider first the case where $\abs{\alpha_1}$ is strictly smaller than $k-\frac{n}{2}$. Since $\alpha_1$ is by assumption the highest occuring derivative we also have $\abs{\alpha_i},\abs{\beta_j} < k - \frac{n}{2}$ so that we can use the embedding. Is $\abs{\alpha_1}$ greater or equal than $k-\frac{n}{2}$ we use the fact that there can only fall at most $k- \abs{\alpha_1}$ derivatives onto $u_i$ and $v_j$ and also that we have chosen $k > n$ so that we obtain
\begin{align*}
	\abs{\alpha_i}, \abs{\beta_j} \leq k - \abs{\alpha_1} \leq \frac{n}{2} < k - \frac{n}{2}.
\end{align*}
Now for the complement of the unit ball we first of all define
\begin{align*}
    a_i :=\frac{1}{2} + \min \{ 1,\abs{\alpha_i}\}, \quad b_j :=\frac{1}{2} + \min \{ 1,\abs{\beta_j}\}
\end{align*}
for every $i = 2,3$ and $j = 1,\ldots, 2\ell+1$. Since we now have $\abs{\gamma} \leq a_2 + a_3 + \sum\limits_{j=1}^{2\ell+1}b_j -a_1$  due to $\abs{\gamma} \leq \ell + \sum \limits_{j=1}^{2\ell+1} \min \{ 1, \abs{\beta_j}\}$ we can estimate
\begin{align*}
    \abs{J(x)} \lesssim \abs{x}^{-a_1}\, \abs{\partial^{\alpha_1}u_1} \, \abs{x}^{a_2} \abs{\partial^{\alpha_2}u_2} \, \abs{x}^{a_3}\abs{\partial^{\alpha_3}u_3} \, \prod_{j=1}^{2\ell+1} \abs{x}^{b_j}\abs{\partial^{\beta_j}v_j}
\end{align*}
for all $\abs{x} \geq 1.$ The $L^2-$norm of the first term gets again estimated by Hardy's inequality and then by the $\dot{H}^s \cap \dot{H}^k$-norm of $u_1$ since we have $s \leq a_1 + \abs{\alpha_1} \leq k$. We estimate the other terms in the $L^{\infty}-$norm by \eqref{Eq: generalized Strauss inequality}. We can apply this Lemma due to the fact that we again have $0 < a_i,b_j < \frac{n-1}{2}$. The $\dot{H}^s\cap \dot{H}^k$-estimate then follows from $s \leq \frac{n}{2} -a_i +\abs{\alpha_i} \leq k$ where the upper bound can be seen by making a case distinction for $\abs{\alpha_1}$ strictly greater than $k - \frac{n-3}{2}$ or smaller than that term. Over all we have finished the proof in the case where one derivative is at least of order $\frac{n}{2}-1$.\\
We now assume that all of the derivatives in \eqref{Definition J} are of order strictly smaller than $\frac{n}{2} - 1$, which implies $|\alpha_i|, |\beta_j| \leq \frac{n-3}{2}$ for all $i = 1,2,3$ and $j = 1, \ldots,2 \ell +1$. On the unit ball we can immediately estimate
\begin{align*}
    \abs{J(x)} \lesssim \abs{x}^{\abs{\alpha_1}-\frac{n-1}{2}}\, \abs{\partial^{\alpha_1}u_1} \, \abs{\partial^{\alpha_2}u_2} \, \abs{\partial^{\alpha_3}u_3} \, \prod_{j=1}^{2\ell+1}\abs{\partial^{\beta_j}v_j}
\end{align*}
and obtain \eqref{estimate for J} with Hardy's inequality and the embedding from Lemma \ref{C^m embedding}.\\
For the complement of the unit ball we define $\tilde a_i := \abs{\alpha_i} + \frac{3}{4}, \tilde b_j := \abs{\beta_j} + \frac{3}{4}$ and estimate
\begin{align*}
    \abs{J(x)} \lesssim \abs{x}^{\abs{\alpha_1}-\frac{n-1}{2}}\, \abs{\partial^{\alpha_1}u_1} \, \abs{x}^{\tilde a_2} \abs{\partial^{\alpha_2}u_2} \, \abs{x}^{\tilde a_3}\abs{\partial^{\alpha_3}u_3} \, \prod_{j=1}^{2\ell+1} \abs{x}^{\tilde b_j}\abs{\partial^{\beta_j}v_j}
\end{align*}
for all $\abs{x} \geq 1$ since we have $\abs{\gamma} \leq \abs{\alpha_1} - \frac{n-1}{2} + \tilde a_2 + \tilde a_3 + \sum\limits_{j=1}^{2\ell+1} \tilde b_j$. The claim now follows by again applying Hardy's inequality and \eqref{Eq: generalized Strauss inequality} due to the fact that we again have $0 < \tilde a_i, \tilde b_j < \frac{n-1}{2}$ as well as $s \leq \frac{n}{2}-\frac{3}{4} \leq k$. We have therefore shown both inequalities in the case where all the functions belong to $C_{c,r}^{\infty}(\R^n)$.

The general case now follows via a density argument and the $L^{\infty}-$embedding of the Sobolev spaces. We will only give the details for the first inequality, the second one can be handled in a similar manner. Take $u_1,u_2,u_3,v \in \dot{H}_r^s(\R^n) \cap \dot{H}_r^k(\R^n)$. Then there exist sequences $(u_{1,j})_{j\in\N},(u_{2,j})_{j\in\N}, (u_{3,j})_{j\in\N}$ and $ (v_{j})_{j\in\N}$ in $C_{c,\,r}^{\infty}(\R^n)$ which converge to $u_1,u_2,u_3$ and $v$ in the $\dot{H}^s(\R^n) \cap\dot{H}^k(\R^n)-$norm, respectively. Eqns.~\eqref{Schauder estimate} and \eqref{Schauder II}  imply  that  
\[\left(u_{1,j}u_{2,j}u_{3,j} \left(F_{\epsilon}(|\cdot|v_j)-F_{\kappa}(|\cdot|v_j)\right)\right)_{j\in\N} \]
forms a Cauchy sequence in $\dot{H}^{s-1}(\R^n) \cap\dot{H}^k(\R^n)$. Due to the $L^{\infty}-$embedding one can easily show that this sequence converges pointwise to $u_1u_2u_3(F_{\epsilon}(\abs{\cdot}v)-F_{\kappa}(\abs{\cdot}v)$ and therefore also in the $\dot{H}^{s-1}\cap\dot{H}^k(\R^n)-$norm. With this observation \eqref{Schauder estimate} follows for arbitrary $u_1,u_2,u_3,v \in \dot{H}_r^s(\R^n) \cap \dot{H}_r^k(\R^n)$.

\end{proof}

\section{Proof of Lemma \ref{Lemma: decay implies Sobolev}}\label{AppendixB}

\begin{proof} 
To prove this Lemma we will argue along the lines of \cite{GloKisSch23}, p.12, Lemma 2.2. We take a radial cut-off function $\chi \in C^{\infty}_{c,\,r}(\R^n)$ with $0\leq\chi\leq1$ satisfying
  \begin{align*} \chi(x) =
      \begin{cases}
      1,\quad \abs{x} \leq 1,\\
      0, \quad \abs{x} \geq 2.  
  \end{cases}
  \end{align*}
Then we consider the sequence $(f_j)_{j\in\N} \subset C_{c,\,r}^{\infty}(\R^n)$ for $f_{j}(x) := f(x)  \chi(\frac{x}{j})$ for $x\in\R^n$ and $j\in\N$.\\
Now we take an arbitrary $s\geq0$ with $s > \frac{n}{2}-k$. We note that we have by the homogeneous Sobolev embedding (see \cite{Tao06}, p. 335)
\begin{align}\label{homogeneous Sobolev embedding}
    \norm{u}_{\dot{W}^{t,q}(\R^n)} \lesssim \norm{u}_{\dot{W}^{\ell,p}(\R^n)}
\end{align}
for every $u\in C^{\infty}_c(\R^n)$ whenever $1<p\leq q<\infty$ and $t,\ell\geq0$ obey the scaling condition $\ell - \frac{n}{p}= t-\frac{n}{q}$. Here, the space $\dot{W}^{s,p}(\R^n)$ for $s \geq 0$ and $1 \leq p < \infty$ is given by the completion of $C_c^{\infty}(\R^n)$ under the norm
\begin{align*}
   \norm{u}_{\dot{W}^{s,p}(\R^n)} := \norm{\mc{F}^{-1}[\abs{\cdot}^s\mc{F}u]}_{L^p(\R^n)}
\end{align*}
and it coincides with the one we introduced in section 2.1 for $p=2$ due to Plancherel. Furthermore, if $s\in \mathbb{N}_0$ is a non-negative integer we also have a representation via partial derivatives in the sense that
\begin{equation}
    \norm{u}_{\dot{W}^{s,p}(\mathbb{R}^n)}\simeq \sum_{\abs{\beta}=s}\norm{\partial^{\beta}u}_{L^p(\mathbb{R}^n)}
\end{equation}
holds for every $u \in C_c^{\infty}(\R^n)$.\\
If we now choose $\ell \in \N_0$ with $s \leq \ell < s + \frac{n}{2}$ (note that we have $n\geq2$) we obtain from \eqref{homogeneous Sobolev embedding}
\begin{align}\label{HLS inequality}
    \norm{f_{j}}_{\dot{H}^s(\R^n)} \lesssim \norm{f_j}_{\dot W^{\ell,p}(\mathbb{R}^n)}, 
\end{align}
where $p$ is defined via $\frac{1}{p} = \frac{1}{2} + \frac{\ell - s}{n}$ and fulfills $\frac{1}{2} \leq \frac{1}{p} <1$ by the choice of $\ell$. We now show that $(f_{j})_{j\in\N}$ is a Cauchy sequence in $\dot{W}^{\ell,p}_{r}(\R^n)$. Using the Leibniz rule it is enough to show 
\begin{align*}
\norm{\partial^{\gamma}f \, \partial^{\beta-\gamma}\left(\chi_j-\chi_i\right)}_{L^p(B_{i,j})} \to 0 \quad \text{as} \quad i,j \to \infty    
\end{align*}
for all multi-indices $\beta,\gamma \in \N_0^n$ satisfying $\abs{\beta} = \ell$ and $\gamma \leq \beta$ and $B_{i,j} := B_{2 \max \{i,j\}} \backslash B_{\min\{i,j\}}$.\\
For $\gamma = \beta$ we have
\begin{align*}
    \norm{\partial^{\beta}f \, \left(\chi_j-\chi_i\right)}_{L^p(B_{i,j})} \lesssim \int_{B_{i,j}} \abs{\partial^{\beta}f(x)}^p\, dx \lesssim \int_{\min\{i,j\}}^{2\max\{i,j\}} r^{-p\abs{\beta}-pk+n-1}dr
\end{align*}
and the last integral converges towards 0 as $i$ and $j$ approach infinity due to the assumptions on $p$ and $\beta$.\\
For $\gamma \neq \beta$ we have $\abs{\gamma} < \abs{\beta}$ and get
\begin{align*}
    \norm{\partial^{\gamma}f \, \partial^{\beta-\gamma}\left(\chi_j-\chi_i\right)}_{L^p(B_{i,j})} = \norm{\partial^{\gamma}f \left( j^{\abs{\gamma}-\abs{\beta}}\, \partial^{\beta-\gamma}\chi(\cdot \slash j)-i^{\abs{\gamma}-\abs{\beta}}\, \partial^{\beta-\gamma}\chi(\cdot \slash i)\right)}_{L^p(B_{i,j})}.
\end{align*}
Since both terms from above can be treated analogously we will only consider the first one:
\begin{align*}
    j^{\abs{\gamma}-\abs{\beta}} \norm{\partial^{\gamma}f \, \partial^{\beta-\gamma}\chi(\cdot \slash j)}_{L^p(B_{i,j})} &\lesssim j^{\abs{\gamma}-\abs{\beta}} \norm{\partial^{\gamma}f}_{L^p(B_{2j}\backslash B_1)} \\ \lesssim j^{\abs{\gamma}-\abs{\beta}} \left(\int_{B_{2j}\backslash B_1} \abs{\partial^{\gamma}f(x)}^p\, dx\right)^{\frac{1}{p}}  &\lesssim j^{\abs{\gamma}-\abs{\beta}} \left(\int_1^{2j} r^{-p\abs{\gamma}-pk+n-1}dr\right)^{\frac{1}{p}}.
\end{align*}
For $\frac{n}{p} \neq \abs{\gamma} + k$ we obtain
\begin{align*}
    j^{\abs{\gamma}-\abs{\beta}} \left(\int_1^{2j} r^{-p\abs{\gamma}-pk+n-1}dr\right)^{\frac{1}{p}} \lesssim j^{\abs{\gamma}-\abs{\beta}} \left( j^{-p\abs{\gamma}-pk+n} - 1\right)^{\frac{1}{p}} \lesssim j^{\frac{n}{p}-\abs{\beta}-k}    
\end{align*}
and for $\frac{n}{p} = \abs{\gamma} + k$ 
\begin{align*}
    j^{\abs{\gamma}-\abs{\beta}} \left(\int_1^{2j} r^{-p\abs{\gamma}-pk+n-1}dr\right)^{\frac{1}{p}} = j^{\abs{\gamma}-\abs{\beta}} \log(2j)^{\frac{1}{p}}. 
\end{align*}
Since both of these terms converge due to the assumptions on $s$ and the fact that we are in the case $\abs{\gamma} < \abs{\beta}$ to 0 as $j$ goes to infinity we have shown that $(f_{j})_{j\in\N}$ is a Cauchy sequence in $\dot{W}^{\ell,p}_{r}(\R^n)$.\\  
By \eqref{HLS inequality} we obtain that this sequence is also Cauchy in $\dot{H}^s_{r}(\R^n)$. Therefore, we have shown that $(f_j)_{j\in\N}$ is a Cauchy sequence in $\dot{H}^s_r(\R^n)$ for every $s\geq0$ with $s > \frac{n}{2}-k$. For every such $s$, we therefore infer the existence of a function $g \in \dot{H}_{r}^{s_1}(\R^n) \cap \dot{H}_{r}^{s_2}(\R^n)$ with $s_1\leq s < s_2$ where $s_1$ is strictly smaller than $\frac{n}{2}$ (here one needs that $k$ is strictly greater than 0) and $s_2$ is strictly greater than $\frac{n}{2}$, so that $(f_j)_{j\in\N}$ converges to $g$ in the $\dot H^{s_1}\cap \dot H^{s_2}-$norm. By \eqref{C^m embedding} we know that $(f_j)_{j\in\N}$ must therefore also converge to $g$ in $L^{\infty}(\R^n)$ (in particular pointwise) so that we conclude $f = g \in \dot{H}_{r}^{s_1}(\R^n) \cap \dot{H}_{r}^{s_2}(\R^n) \subset \dot{H}_{r}^s(\R^n)$ which finishes the proof.

\end{proof}

\bibliographystyle{plain}
\bibliography{bibliography.bib}

\end{document}